\documentclass[12pt]{amsart}
\usepackage{amsmath}
\usepackage{amssymb}
\usepackage{graphicx}
\usepackage{pinlabel}
\usepackage[all]{xy}
\usepackage{caption}
\usepackage{subcaption} 
\usepackage{epstopdf}
\usepackage{appendix}

\usepackage{amsmath,amssymb,amsfonts,amsthm,graphicx}

\usepackage{pinlabel} 

\usepackage[usenames,dvipsnames]{color}

\usepackage{hyperref}

\numberwithin{equation}{section}

\newtheorem{dummy}{dummy}[section]
\newtheorem{lemma}[dummy]{Lemma}
\newtheorem{theorem}[dummy]{Theorem}

\newtheorem{corollary}[dummy]{Corollary}
\newtheorem{proposition}[dummy]{Proposition}

\theoremstyle{definition}
\newtheorem{definition}[dummy]{Definition}

\newtheorem{example}[dummy]{Example}
\newtheorem{remark}[dummy]{Remark}

 \newcommand{\lag}{L} 
 
\newcommand{\R}{\mathbb {R}}

\newcommand{\C}{\mathbb {C}}
\newcommand{\Z}{\mathbb {Z}}
\newcommand{\F}{\mathbb {F}}
\newcommand{\M}{{\mathcal{M}}}

 \newcommand{\calF}{\mathcal {F}}
\newcommand{\e}{\epsilon}
\newcommand{\dd}{\partial}
\newcommand{\alg}{\mathcal{A}}
\newcommand{\aug}{\mathcal{A}ug}
\newcommand{\wt}{\widetilde}
\newcommand{\B}{\mathbb{B}}
\newcommand{\D}{\mathbb{D}}
\newcommand{\Sim}{\sim_{Aug_{+}}}

\newcommand{\x}{\times}

\newcommand{\std}{{\operatorname{std}}}
\newcommand{\Hom}{{\operatorname{Hom}}}
\newcommand{\Ho}{{\operatorname{H}}}

\newcommand{\ind}{{\operatorname{ind}}}
\newcommand{\leg}{\Lambda}
\def\skel{\operatorname{Skel}}

\definecolor{dgr}{rgb}{0,.8, 0}
\definecolor{darkmagenta}{rgb}{.5,0,.5}

\begin{document}

\title[Obstructions to reversing Lagrangian surgery]{Obstructions to reversing Lagrangian surgery \\ in Lagrangian Fillings}

\author[Capovilla-Searle]{Orsola Capovilla-Searle}
\author[Legout]{No\'emie Legout}
\author[Limouzineau]{Ma\"ylis Limouzineau}
\author[Murphy]{Emmy Murphy}
\author[Pan]{Yu Pan}
\author[Traynor]{Lisa Traynor}

\address{Department of Mathematics\\ UC Davis \\ Davis\\ California\\U.S.A.}
\email{ocapovillasearle@ucdavis.edu}

\address{Department of Mathematics\\ Uppsala University\\ Uppsala\\ Sweden.}
\email{noemie.legout@math.uu.se}

\address{}
\email{lim.maylis@gmail.com}

\address{Department of Mathematics\\Princeton University\\Princeton\\New Jersey}
\email{em7861@princeton.edu}

\address{Center of applied Mathematics\\ Tianjin University\\ Tianjin\\P.R.China}
\email{ypan@tju.edu.cn}

\address{Department of Mathematics\\Bryn Mawr College\\Bryn Mawr \\ Pennsylvania}
\email{ltraynor@brynmawr.edu}


\begin{abstract} Given an immersed, Maslov-$0$, exact Lagrangian filling of a Legendrian knot, if the filling has a vanishing index and action double point, then through Lagrangian surgery it is possible to obtain a new immersed, Maslov-$0$, exact Lagrangian filling with one less double point and with genus increased by one. We show that it is {\it not} always possible to reverse the Lagrangian surgery: not every immersed, Maslov-$0$, exact Lagrangian filling with genus $g \geq 1$ and $p$ double points can be obtained 
from such a Lagrangian surgery
on a filling of genus $g-1$ with $p+1$ double points. 
To show this, we establish the connection between the existence of an immersed, Maslov-$0$, exact Lagrangian filling of a Legendrian $\Lambda$ that has $p$ double points with action $0$ and the existence of an embedded, Maslov-$0$, exact Lagrangian cobordism from $p$ copies of
a Hopf link to $\Lambda$. We then prove that a count of augmentations provides an 
obstruction to the existence of embedded, Maslov-$0$, exact Lagrangian cobordisms between Legendrian links.\end{abstract}

\maketitle

\section{Introduction}
An important problem in smooth topology is to understand the  $4$-ball genus and the $4$-ball crossing number of a smooth knot. Through a variety of techniques, including Heegaard Floer homology, gauge theory, and instanton homology~\cite{OS, KroMro19}, the $4$-ball genus and crossing numbers have been calculated for all prime knots with crossing number $10$ or less. Less is known about these invariants for connect sums; see, for example,~\cite{LivVanC}. In general, the $4$-ball genus and crossing numbers give information about what combinations of genus and double points can be realized by surfaces in the $4$-ball with a fixed knot as their boundary: a transverse double point can be resolved at the cost of increasing the genus of the surfaces, and sometimes a disk that intersects the surface transversely along its boundary allows one to reduce the genus at the cost of increasing the number of double points.

One can study analogous problems when the knot and surface satisfy additional geometric conditions imposed by symplectic geometry. The development of symplectic field theory~\cite{EGH} motivated the study of  {\it Lagrangian} cobordisms between {\it Legendrian} submanifolds; these are embedded Lagrangian submanifolds in the symplectization of a contact manifold that have  cylindrical ends over the Legendrians, see Figure~\ref{fig:lagcob} for a schematic picture. 
Lagrangian  fillings occur when the bottom Legendrian is the empty set.

For a fixed Legendrian knot, obstructions to the existence of embedded, exact Lagrangian fillings arise from classical and non-classical invariants of the Legendrian; see, for example,~\cite{Cha, EkhSFT, DR, SabTra}. Legendrians that admit embedded, Lagrangian fillings are relatively rare and Lagrangian fillings that do exist are known to be more topologically rigid than their smooth counterparts: an embedded, oriented, exact Lagrangian filling will always realize the smooth $4$-ball genus of the knot~\cite{Cha}.

 \emph{Immersed} Lagrangian fillings are more plentiful: any Legendrian with rotation number $0$ will admit an immersed Lagrangian filling, see, for example,~\cite[Remark 4.2]{Cha}.  
Currently, there are fewer known obstructions for {immersed} Lagrangian fillings.
Classical invariants, linearized contact homology, and generating family homology can give some insight into the possible combinations of genus and double points that can be realized in an immersed, Maslov-$0$, exact Lagrangian filling of a Legendrian knot,~\cite{Cha, Samthesis, sam_lisa, PR}. Sometimes the existence of one such immersed filling will lead to the existence of another:  if $\leg$ admits
 an immersed,  Maslov-$0$, exact Lagrangian filling of genus $g$ with $p \geq 1$ double points such that one of the double points  has ``index and action equal to $0$'' (see Section \ref{sec:bg} for definitions), then 
 through Lagrangian surgery it is possible to construct a new immersed,  Maslov-$0$, exact Lagrangian filling of genus $g+1$ with $p-1$ double points. In this paper we address the following question: is it always possible to ``reverse'' the surgery process?
 Namely, can every immersed, Maslov-$0$, exact Lagrangian filling with genus $g\geq1$ and $p$ double points be obtained by Lagrangian surgery on an action-$0$ and index-$0$ double point of  an immersed, Maslov-$0$, exact Lagrangian filling of genus $g-1$ with $p+1$ double points? See Figure~\ref{fig:line} for a schematic of this question.

 \begin{figure}[!ht]
	\labellist
	\pinlabel ${g}$ at 100 -25
	\pinlabel ${g}$ at 445 -25
	\pinlabel ${p}$ at -25 100
	\pinlabel ${p}$ at 318 100
	\pinlabel \LARGE${?}$ at 470 165
	\footnotesize
	\pinlabel $0$ at 13 -5
	\pinlabel $1$ at 45 -5
	\pinlabel $2$ at 73 -5
	\pinlabel $3$ at 102 -5
	\pinlabel $4$ at 132  -5
	\pinlabel $5$ at 164 -5
	\pinlabel $6$ at 194 -5
	\pinlabel $0$ at 360 -5
	\pinlabel $1$ at 388 -5
	\pinlabel $2$ at 418 -5
	\pinlabel $3$ at 446 -5
	\pinlabel $4$ at 475.5 -5
	\pinlabel $5$ at 508 -5
	\pinlabel $6$ at 538 -5
	\pinlabel $0$ at 0 15
	\pinlabel $1$ at 0 43
	\pinlabel $2$ at 0 75
	\pinlabel $3$ at 0 104
	\pinlabel $4$ at 0 134
	\pinlabel $5$ at 0 170
	\pinlabel $6$ at 0 200
	\pinlabel $0$ at 345 15
	\pinlabel $1$ at 345 43
	\pinlabel $2$ at 345 75
	\pinlabel $3$ at 345 104
	\pinlabel $4$ at 345 134
	\pinlabel $5$ at 345 170
	\pinlabel $6$ at 345 200
	\endlabellist
	\normalsize
	\includegraphics[width=4in]{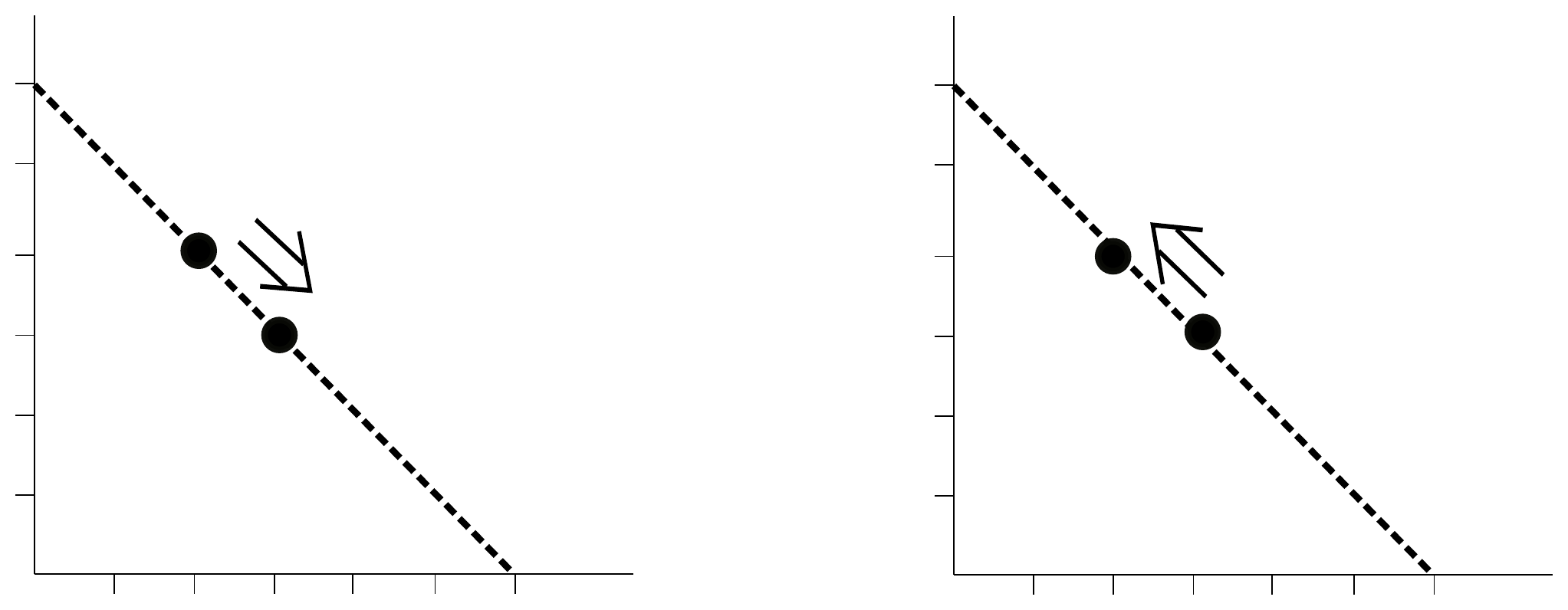}
	\vspace{0.3in}
	
	\caption{Asking if a filling arises from Lagrangian surgery is asking if it is possible to 
	decrease $g$ at the expense of increasing $p$.}\label{fig:line}
\end{figure}
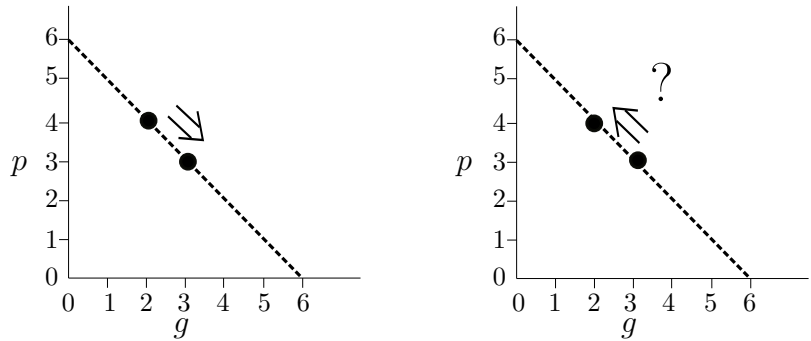
 
We answer this question by first translating the existence of an immersed, Maslov-$0$, exact Lagrangian filling with  action-$0$ double points to the existence of an embedded, Maslov-$0$, exact Lagrangian cobordism from  a disjoint union of Hopf links to $\leg$.   We then construct new obstructions to the existence of {\it embedded}, 
Maslov-$0$, exact Lagrangian cobordisms between Legendrian links in $\R^3_{\std}$ through the theory of  augmentations.
  Finally, we apply our obstruction techniques to find families of Legendrian knots admitting immersed, Maslov-$0$, exact Lagrangian fillings that do not arise from Lagrangian surgery as defined in Definition~\ref{defn:not-from-surgery}.

\subsection{Immersed to embedded Lagrangian cobordisms}
In \cite[Theorem 1.3]{Cha15} Chantraine showed that the existence of an immersed, exact Lagrangian filling of $\leg$ with a single action-$0$ double point  implies the existence of an embedded, exact  Lagrangian cobordism from a Hopf link  to $\leg$.   We give an extension of this result to more general cobordisms, more double points, and higher dimensions; Definition~\ref{defn:Hopf}  defines $\leg_{\Ho}^{k}$, the Hopf link with Maslov potential induced by the integer $k$.

\begin{theorem}\label{thm:maingen1} Suppose $\leg_\pm$ are Legendrian links in $\R^{2n-1}_{std}$, $n \geq 2$.
 If there exists an immersed, Maslov-$0$, exact Lagrangian cobordism $L^{\times}$ from $\leg_{-}$ to $\leg_{+}$ with genus $g$ and $p$ double points, $m$ of which, $x_{1}, \dots, x_{m}$, have action $0$,  then there exists an immersed,  Maslov-$0$,
exact Lagrangian cobordism $L$  of genus $g$ with $(p-m)$ double points  
from  $\bigsqcup_{k=1}^{m}\leg_{\Ho}^{ i_k} \cup \leg_{-}$ to $\leg_{+}$,
where the Maslov potential on the Hopf links are induced by the indices $i_{k}$ of $x_{i_{k}}$. 
  \end{theorem}

As a corollary, we see that if each of the $p$ double points of $L^{\times}$ has action $0$, then we can conclude the existence 
of an embedded, Maslov-$0$, exact Lagrangian cobordism $L$ of genus $g$
from $\sqcup_{k=1}^{p} \leg_{\Ho}^{ i_k} \cup \leg_{-}$ to $\leg_{+}$.

\begin{remark} \label{rem:contractible} 
The hypothesis that all the double points of the immersed exact Lagrangian cobordism 
have action $0$ is not generic.  Indeed, it corresponds to the assumption that all Reeb chords in the Legendrian lift $\widetilde \lag$
 of the Lagrangian cobordism have length $0$. One can instead generalize to consider 
a {\it contractible double point}, which is  
 a double point $X$   whose
corresponding Reeb chord $c_{X}$ is  contractible, i.e., its length can be shrunk to $0$ without the front projection of $\widetilde \lag$ needing to undergo any 
moves; see \cite[Definition 6.13]{EHK} for a precise description of a contractible Reeb chord.    
The notion of multiple action-$0$ double points can be generalized to multiple ``simultaneously contractible'' double points.  
 The Legendrian Hopf link from Figure \ref{fig:Hopf} illustrates that two individually contractible Reeb chords need not be simultaneously contractible: here, the two interstrand Reeb chords  $b_1$ and $b_2$ are not simultaneously contractible since they cobound a disk.  For any immersed, exact Lagrangian filling we can  apply a Legendrian isotopy so that all Reeb chords in the Legendrian lift have {\it nonzero} length without any births or deaths of pairs of Reeb chords; such a Legendrian isotopy on the Legendrian lift can be realized by a {\it safe Hamiltonian isotopy} of the Lagrangian filling, see \cite{CDRGG2}. However, in general there are obstructions
  in going from a set of contractible Reeb chords to a set of action-$0$ double points.
  \end{remark}

\begin{remark} Theorem~\ref{thm:maingen1} can be extended beyond transverse double points to more general singularities of exact Lagrangians. In particular, we can consider any Lagrangian singularity $f$ such that the boundary of a Darboux ball centered at the singularity, or a real morsification 
of the singularity, intersects the exact Lagrangian as a Legendrian and the primitive is constant on the Legendrian. See~\cite{Casals_sing} for some examples of such singularities.
 \end{remark}

\subsection{Obstructions to embedded exact Lagrangian cobordisms}

For a Legendrian link $\Lambda$
 in the standard contact manifold $\R^3_{\std}$, the Chekanov-Eliashberg DGA \cite{Che, Eli} $(\alg(\Lambda), \dd)$ is a powerful invariant that arises from symplectic field theory \cite{EGH}.
An augmentation $\e$ of $\alg(\Lambda)$ to a unital, commutative ring $\F$ is a DGA map $\e: (\alg(\Lambda), \dd) \to (\F, 0)$, where $(\F, 0)$ is a DGA with $\F$ in degree $0$ and differential identically $0$. 
Let $Aug(\leg; \F)$ denote the set of augmentations of $\alg(\leg)$ to $\F$.
An embedded, Maslov-$0$, exact Lagrangian cobordism $\lag$ from $\Lambda_-$ to $\Lambda_+$ induces a DGA map from $\alg(\leg_+)$ to $\alg(\leg_-)$~\cite{EHK} that by composition with an augmentation of $\alg(\leg_-)$ induces a map  
\begin{equation} \label{eqn:F}
	\mathcal{F}_\lag:Aug(\Lambda_-;\F)\to Aug(\Lambda_+; \F).
\end{equation}
Let $Aug(\Lambda;\F)/\Sim$ denote the set of augmentations up to the equivalence relation $\Sim$ given by  the natural equivalence given in the {augmentation category}
$\aug_{+}(\leg)$, see Definition~\ref{defn:equal} , or equivalently with respect to split-DGA homotopy, see Definition~\ref{defn:split-DGA-homotopy} and Proposition~\ref{prop:equivalences}. 
We will use $|Aug(\Lambda;\F)/\Sim|$   
to denote the cardinality of the set $Aug(\Lambda;\F)/\Sim$.
 \begin{theorem}\label{thm:main} Let $\Lambda_{\pm}$ be Legendrian links in $\R^3_\std$ such that there exists an embedded, Maslov-$0$,  exact Lagrangian cobordism $\lag$  from $\Lambda_-$ to $\Lambda_+$.
Suppose $\F$ is a commutative ring; if $\F$ does not have characteristic $2$ we further assume that $\lag$ is spin. 	Given augmentations $\e_{1}, \e_{2} \in Aug(\leg_{-}, \F)$, if 
$\mathcal F_\lag(\e_{1}), \mathcal F_\lag(\e_{2})$ are equivalent with respect to $\Sim$, then $\e_{1}, \e_{2}$ are equivalent with respect to $\Sim$.
In particular, 
\begin{equation} 
		|Aug(\Lambda_-;\F)/\Sim| \leq |Aug(\Lambda_+;\F)/\Sim|. \label{eqn:aug-count}
\end{equation}
If $\leg_{\pm}$ are single component Legendrian knots or $\F = \Z_{2}$, the map
 \begin{equation}
 \mathcal{F}_\lag:Aug(\Lambda_-;\F)/\Sim\to Aug(\Lambda_+; \F)/\Sim \label{eqn:knot-map}
 \end{equation}
exists and is injective.
\end{theorem}

Although the map $\mathcal{F}_\lag$ on the set of augmentations (see Equation~\ref{eqn:F}) always exists, the map $\mathcal{F}_{\lag}$ on the set of equivalence classes of augmentations (see Equation~\ref{eqn:knot-map}) does not exist for multi-component links or when $\F \neq \Z_{2}$. See Remark~\ref{rem:aug-equiv}. The fifth author \cite{Pan1} proved a result that implies Theorem~\ref{thm:main}  
when $\leg_\pm$ are Legendrian knots.   We will see that equation~(\ref{eqn:aug-count}) provides a practical way to obstruct the existence of embedded cobordisms when $\F=\Z_2$.
When $\F$ is not of characteristic $2$, then, as in \cite{CDRGG, EESorientation, karl, Seidel}, rigid holomorphic disks in the moduli spaces that arise in the proof of Theorem~\ref{thm:main} are counted with signs. 

Fillings induce augmentations, and so one of the many reasons to consider augmentations to a more general $\F$ is that they can give information on the number of fillings of a Legendrian link.   It is known that Hamiltonian isotopic, embedded, Maslov-$0$, exact Lagrangian fillings induce $\Sim$ equivalent augmentations  to $\mathbb Z$,~\cite{EHK,karl}.
 Examples of Legendrian links that have an infinite number of distinct fillings up to Hamiltonian isotopy were first given in~\cite{CG} and later also in~\cite{CZ, Gao1,Gao2, CN, orsola_thesis}.
 From the existence of a  Legendrian with an infinite number of distinct fillings distinguished by augmentations to $\mathbb Z$, we can apply Theorem~\ref{thm:main} to deduce the existence of more such Legendrians.

\begin{corollary} (c.f.~\cite[Proposition 7.5, Remark 7.6]{CN})  Let $N \in \mathbb N \cup \{\infty\}$.
Suppose  $\leg_{-}$ is a Legendrian link that  
has  
$N$ augmentations to $\mathbb Z$  
up to $\Sim$ equivalence that are induced by embedded, Maslov-$0$, exact Lagrangian fillings,
and  there exists an embedded, Maslov-$0$, exact Lagrangian cobordism from $\leg_{-}$ to $\leg_{+}$, 
then $\leg_{+}$ admits $N$
embedded, Maslov-$0$, exact Lagrangian fillings that are distinct up to Hamiltonian isotopy.
\end{corollary}

\begin{proof} 
Consider two embedded, Maslov-$0$, exact Lagrangian fillings of $\Lambda_-$ that  induce augmentations $\e_{1}, \e_{2} \in Aug(\leg_{-}, \mathbb Z)$ that are not equivalent with respect to $\Sim$.  
Concatenating these fillings with the cobordism $\lag$ from $\Lambda_-$ to $\Lambda_+$ produces two embedded, Maslov-$0$, exact Lagrangian fillings of $\Lambda_+$; 
the augmentations induced
by these fillings agree with $\mathcal F_{\lag}(\e_{1}), \mathcal F_{\lag}(\e_{2} ) \in Aug(\leg_{+},\mathbb Z)$. 
By Theorem \ref{thm:main},  
$\mathcal F_{\lag}(\e_{1}), \mathcal F_{\lag}(\e_{2} )$ are not equivalent with respect to $\Sim$, and thus the fillings of $\leg_{+}$ are not Hamiltonian isotopic. 
\end{proof}

In the case when $\leg_{\pm}$ are knots, Theorem~\ref{thm:main} was derived in \cite{Pan1} from studying the {\it augmentation category} $\aug_{+}(\leg)$, which is an $A_\infty$-category associated to a Legendrian $\Lambda$, see \cite{NRSSZ}. The objects of  $\aug_{+}(\leg)$ are augmentations $\e: \alg(\Lambda)\to \F$ and morphisms $Hom_{+}(\e^1, \e^2)$ are modules over Reeb chords between $\leg$ and its ``push-off''.  When $\leg_{\pm}$ are knots, the functoriality of the DGA under cobordisms naturally extends the map $\calF_{\lag}$ from Equation~(\ref{eqn:F})  to a functor 
$$
	(\calF_{\lag})_{+}: \aug_+(\Lambda_-)\to \aug_+(\Lambda_+) 
$$
between the augmentation categories.
 In \cite{Pan1} it is proved that $\mathcal F_{\lag}$ is injective on equivalence classes of objects when $\Lambda_\pm$ are knots by showing that the functor $(\calF_{\lag})_{+}$ induces an isomorphism on the degree $0$ cohomology of morphism spaces; i.e. 
 $H^0\Hom_+(\e^1,\e^2)\cong H^0\Hom_+(\mathcal F_{\lag}(\e^{1}),\mathcal F_{\lag}(\e^{2}))$. However, this latter statement fails for links.
Moreover, the functor $(\calF_{\lag})_{+}$ is not even well-defined for cobordisms between links. Instead,  we  employ the machinery of wrapped Floer theory for Lagrangian cobordisms developed in \cite{CDRGG} (see Section~\ref{sec:wrap}), to argue that if $\mathcal F_{\lag}(\e^{1}), \mathcal F_{\lag}(\e^{2})$ are equivalent, then $\e^{1}, \e^{2}$ are equivalent, where equivalence is with respect to $\Sim$. 
To do this, we construct   ``wrong-way'' maps, namely maps in direction opposite to those induced by $(\mathcal F_{\lag})_{+}$,
$$\iota: H^*\Hom_+(\mathcal F_{\lag} (\e^1), \mathcal F_{\lag}(\e^2)) \to H^*\Hom_+(\e^1, \e^2).$$
Combining the  work of the second author  \cite{Noe}  and wrapped Floer theory, we show that $\iota$ is unital and preserves the product structure on $H^{*}\Hom_+$

In Section~\ref{sec:obs}, we build two additional obstructions to the existence of embedded, Maslov-$0$, exact Lagrangian cobordisms in terms of linearized contact homology $LCH^\e_\ast(\Lambda)$ (see Section~\ref{sec:aug}) and the ruling polynomial $R_{\Lambda}(z)$ (see Equation~\eqref{eqn:ruling-poly}, which are Legendrian invariants that are associated to augmentations. These results are extensions of parallel results in \cite{Pan1}.

\begin{proposition}[see Proposition~\ref{prop:LCH}] \label{prop1}
	Assume $\F$ is a field, $\leg_{\pm}$ are Legendrian links in $\R^3_\std$, $\e$ is an augmentation of $\Lambda_-$, and  $\lag$ is an embedded, Maslov-$0$, exact Lagrangian cobordism from $\Lambda_-$ to $\Lambda_+$, which we further assume to be spin if $\F$ does not have characteristic $2$. 
	Then,
	\begin{equation}\label{eq:iso}
LCH^{\mathcal F_{\lag} (\e)}_{k}(\Lambda_+)\cong LCH^{\e}_{k}(\Lambda_-)
	\end{equation}
	for $k<0$ and $k>1$.
\end{proposition}

\begin{proposition}[see Corollary~\ref{cor:rul}]\label{prop2}
	Let $\lag$ be a spin, embedded, Maslov-$0$, exact Lagrangian cobordism from $\Lambda_-$ to $\Lambda_+$. Then,   $$R_{\Lambda_-}(q^{1/2}-q^{-1/2})\leq q^{-\chi(\lag)/2} R_{\Lambda_+}(q^{1/2}-q^{-1/2})$$ for any $q$ that is a power of a prime number.
\end{proposition}

\subsection{Obstructions to  reversing Lagrangian surgery}
We apply Theorem~\ref{thm:maingen1} and Theorem~\ref{thm:main}  to find examples of Legendrian knots in $\R^3_{std}$ admitting immersed, Maslov-$0$, exact Lagrangian fillings that do not arise from Lagrangian surgery.  We say that an  immersed, Maslov-$0$, exact Lagrangian filling $F_{g}^{p}$ of a Legendrian $\leg$ with 
genus $g$ and $p$ double points {\it does not arise from Lagrangian surgery} if there does not exist an immersed, Maslov-$0$, exact Lagrangian filling $F_{g-1}^{p+1}$ with 
 genus $g-1$ and $p+1$ double points where the indices and actions of $p$ of the double points agree with those of $F_{g}^{p}$ and there is an additional double point of action and index $0$ that could be surgered to produce $F_{g}^{p}$; see Definition~\ref{defn:not-from-surgery}. 

As a simple illustration of our techniques, consider the Legendrian knot $\leg_{7_4}$ in Figure \ref{fig:7_4}(a), which is the maximal-tb representative of the knot $7_4$. 
Using known construction techniques, described in Section~\ref{sec:ex},  we know that  $\leg_{7_4}$
 admits an embedded, Maslov-$0$, exact Lagrangian filling of genus $1$;  we prove this filling cannot be obtained by applying Lagrangian surgery on an immersed, Maslov-$0$, exact Lagrangian disk filling with one double point. Indeed, if it was the case, $\leg_{7_{4}}$ would admit an immersed, Maslov-$0$, exact Lagrangian disk filling with a  double point of action  $0$ and index $0$. 
 By Theorem \ref{thm:maingen1}   
 the existence of such an immersed filling is equivalent to the existence of an embedded, Maslov-$0$, exact Lagrangian cobordism from the Hopf link $\leg_{\Ho}^0$ to $\leg_{7_{4}}$. However, since we can compute  $$|Aug(\leg_{\Ho}^0;\Z_2)/\Sim| = 3, \text{ and } |Aug(\leg_{7_4};\Z_2)/\Sim| = 1,$$  
 by  Theorem \ref{thm:main} such an embedded cobordism does not exist.
In fact, for this specific example, there is an underlying smooth reason that such an immersed Lagrangian disk filling does not exist for $\leg_{7_{4}}$:  it has been shown in \cite{OS} using Heegaard Floer homology that the smooth knot $7_{4}$ does not have any smooth, immersed disk filling with $1$ double point.  The following theorem gives examples of Legendrian knots with obstructed immersed Lagrangian fillings, where there is no smooth obstruction. 
 The Legendrian knot shown in Figure~\ref{fig:7_4}(b) is an example of a Legendrian in Theorem \ref{thm:obstruct} (1), and the Legendrian shown in Figure~\ref{fig:7_4}(c) is an example of a Legendrian in Theorem \ref{thm:obstruct} (2).

\begin{figure}[!ht]
	\labellist
	\pinlabel $(a)$ at 100 -30
	\pinlabel $(b)$ at 320 -30
	\pinlabel $(c)$ at 680 -30
	\endlabellist

	\includegraphics[height=1in,width=5in]{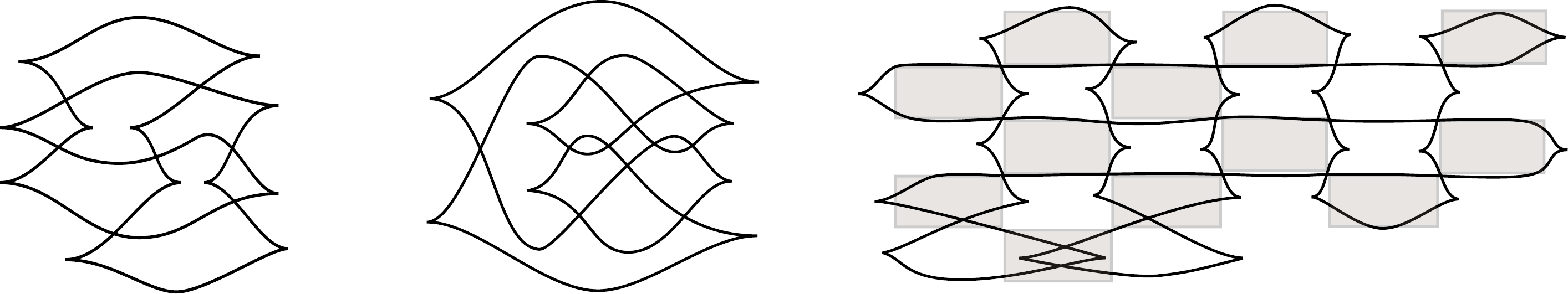}
	\vspace{0.2in}
	\caption{Legendrian knots admitting fillings that do not arise from Lagrangian surgery. (a) $\leg_{7_4}$; (b)  $\leg_{k}|_{k=1}= \leg_{9_{48}}$; (c) the clasped checkerboard $\leg_{2}^{1}$. }
	\label{fig:7_4}
\end{figure}

\begin{theorem}
	\begin{enumerate}
 \item For all $k \geq 1$, there exists a Legendrian knot $\leg_{k}$, with $\leg_{1}$ being a Legendrian $9_{48}$ knot, that admits an  immersed, Maslov-$0$, exact Lagrangian filling  $F_{k}^{k}$, which has genus $k$ and $k$ double points, that does not arise from Lagrangian surgery, even though $\leg_{k}$ admits a smooth filling of genus $(k-1)$ with $(k+1)$ double points.  		
\item Given $g \in \mathbb Z^{+}$, and $p \in \mathbb Z^{\geq 0}$, there is a Legendrian knot $\leg_{g}^{p}$ 
 that has an immersed, Maslov-$0$, exact Lagrangian filling $F_{g}^{p}$, which has genus $g$ and $p$ double points, that does not arise from Lagrangian surgery.
 	\end{enumerate}\label{thm:obstruct}  
\end{theorem}

The family $\leg_{g}^{p}$ in Theorem~\ref{thm:obstruct}(2)
generalizes $\leg_{7_{4}}$: $\leg_{1}^{0} = \leg_{7_{4}}$. Other than $\leg_{1}^{0}$, the knots in this family have crossing numbers that are at least $11$ and can be arbitrarily large:  a SnapPy calculation shows that $\leg_{1}^{1}$ is the smooth knot $11_{495}$, and, to the best of our knowledge, this and the others in the family do not have smooth obstructions.

\begin{remark}  The Poincar\'e polynomial for the Legendrian contact homology of $\leg_{9_{48}}$ is $t^{-1}+2+2t$, \cite{ChNg}. Using the techniques of generating families, this  implies that any immersed, gf-compatible (and thus Maslov-$0$ and exact) Lagrangian disk filling of $\leg_{9_{48}}$ must have at least two double points, of indices $0$ and $1$ \cite{Samthesis, sam_lisa}.   With the techniques of this paper, we obstruct the case where both the double points must satisfy the additional action-$0$ hypothesis, or the equivalent
``contractible'' formulation described in Remark~\ref{rem:contractible}. 
   \end{remark}

We end this introduction with the following observation. The fact that the immersed Lagrangian fillings in Theorem \ref{thm:obstruct} are not obtained from Lagrangian surgery on other fillings tells us about the non-existence of particular Lagrangian disks.  
As explained in Section~\ref{sec:surgery}, after a change of coordinates and the removal of a cylindrical end,
an exact Lagrangian filling $L$ of a Legendrian $\Lambda$ in the symplectization of $\R^3_{std}$ becomes a compact, exact Lagrangian filling $\overline L$ in $(\B^4, \omega_{std})$ of $\leg \subset S^3$.

We will call an essential, embedded  curve $\gamma \subset L$ is a \textbf{pre-singularity loop} if it is  obtained by the transversal intersection of a Lagrangian disk $D\subset (\B^4, \omega_{std})$ with the interior of $\overline L$. As shown in \cite{Yau}, given a pre-singularity loop, it is always possible to reverse Lagrangian surgery.  
Thus, we obtain the following corollary to Theorem~\ref{thm:obstruct}. 

\begin{corollary} \label{cor:loop}  Let $\leg$ be one of the Legendrian knots from Theorem~\ref{thm:obstruct} that admits an immersed, Maslov-$0$, exact Lagrangian filling $F_{g}^{p}$ with genus $g$ and $p$ double points that cannot be obtained by Lagrangian
surgery.  
Then the filling  $F_{g}^{p}$ does not admit a pre-singularity loop.
\end{corollary}

\begin{remark}	Given an embedded, orientable,  exact Lagrangian filling $L$ with a pre-singularity loop $\gamma \subset L$ that bounds a Lagrangian disk with interior disjoint from $L$,
 one can shrink the Lagrangian disk to a point and perform Lagrangian surgery in one of the two ways, as explained in Section~\ref{sec:surgery},  to obtain two embedded exact Lagrangian fillings $L_1$ and $L_2$. Note that $L_1$ and $L_2$ are smoothly isotopic but not Hamiltonian isotopic.
This has been employed to great effect in the construction of infinitely many orientable embedded exact Lagrangian fillings for certain Legendrian links up to Hamiltonian isotopy by~\cite[Theorem 4.21]{CZ}. Obstructing the existence of pre-singularity loops allows one to understand when such constructions are not possible. The obstruction tools that we construct however do not determine which curves in $L$ are pre-singularity loops. They also only provide an upper bound on the number of pre-singularity loops $\gamma$ in $L$. \end{remark}

\vspace{0.1in}
	\noindent
	{\bf Outline:}  In Section~\ref{sec:bg}, we define immersed, Maslov-$0$, exact Lagrangian cobordisms and the action and index of double points.  In
	Section~\ref{sec:surgery}, we review Lagrangian surgery and prove Theorem~\ref{thm:maingen1} by employing the theory of
	Liouville and Weinstein structures.  We then review concepts that are used in proving Theorem~\ref{thm:main} including the Chekanov-Eliashberg DGA, the augmentation category, and the wrapped Floer theory for cobordisms, in Sections~\ref{sec:dga}, \ref{sec:augcat},  and \ref{sec:wrap}, respectively.   In Section~\ref{sec:augcat}, the equivalence relation $\Sim$ is reviewed and the new definition of split-DGA homotopy is introduced.  
	In Section~\ref{sec:main}, we integrate everything together and prove Theorem~\ref{thm:main} as well as the other obstructions provided by Propositions \ref{prop1} and \ref{prop2}.
	Finally, in Section~\ref{sec:ex}, we apply Theorem~\ref{thm:maingen1} and Theorem~\ref{thm:main} to prove Theorem~\ref{thm:obstruct}: for one of the families we count augmentations directly while for the other family we apply the theory of rulings to count augmentations. 
	\vspace{0.1in}

	\noindent
	{\bf Acknowledgement:} 
	We would like to thank the Banff International Research Station-Casa Mathmatica Oaxaca (BIRS-CMO) for its support during the Women in Geometry Workshop (19w5115), where the work on this paper was begun in June 2019. 
	The first author was partially supported by NSF Graduate
Research Fellowship under grant no. DGE-1644868 while completing their graduate degree at Duke University, and NSF grant DMS-2103188 as a postdoc. The second author was supported by the grant KAW 2016.0198 from the Knut and Alice Wallenberg Foundation and the grant 2016-03338 from the Swedish Research Council. The first author wishes to express appreciation to ICERM where where portions of this work were completed.
	We also thank Roger Casals, Georgios Dimitroglou-Rizell, Lenny Ng, Brendan Owens, and Laura Starkston for helpful conversations.

\section{Actions and Indices of Double Points}\label{sec:bg}

In the first subsection, we define immersed, exact Lagrangian cobordisms between Legendrian links and the {\it action} of a double point.   In the second subsection, we define the {\it index} of a double point.

\subsection{Immersed Lagrangian cobordisms and the action of a double point}\label{subsec:leg} 
Let $\Lambda$ be a Legendrian knot or link in the standard contact manifold 
$\R^{2n+1}_{\text{std}}=(\R^{2n+1},\ker\alpha)$, where $\alpha=dz-\displaystyle{\sum_{i=1}^n y_i\,dx_i}$  and  $(x_1, \dots, x_n, y_1, \dots, y_n, z)$ are the coordinates of $\R^{2n+1}$.  
There are two useful projections of $\Lambda$: the \textbf {Lagrangian projection}  $\pi_{xy} (\Lambda)$ where $\pi_{xy}:\R^{2n+1}\rightarrow \R^{2n}, ({\bf x,y},z)\rightarrow ({\bf x,y})$, and the \textbf{front projection}  $\pi_{xz}(\Lambda)$ where $\pi_{xz}:\R^{2n+1}\rightarrow \R^{n+1}, ({\bf x,y},z)\rightarrow ({\bf x},z)$, where ${\bf x}$ and ${\bf y}$ are $(x_1, \dots, x_n)$ and $(y_1, \dots, y_n)$.  {We will always assume that $\leg$ is {\bf chord generic}, meaning that the self-intersection points of $\pi_{xy}(\leg)$ consists of a finite number of transverse double points.}

Now we define immersed, exact Lagrangian cobordisms between Legendrian links, which are immersed manifolds with ``cylindrical ends'' over Legendrian links; see Figure~\ref{fig:lagcob}.   This extends the definition of embedded, exact Lagrangian cobordisms of ~\cite[Definition 1.1]{EHK}.  
\begin{definition} \label{defn:cobord} 
	Let $\Lambda_\pm$  be Legendrian links in $\R^{2n-1}_{\text{std}}$.  
	An  {\bf immersed, exact Lagrangian cobordism $\lag$ from $\Lambda_{-}$ to $\Lambda_{+}$} is an immersed, Lagrangian submanifold in the symplectization, $L = i(\Sigma)$ for a Lagrangian immersion $i:\Sigma\rightarrow (\R_t\times\R^{2n-1}, d(e^t\alpha))$, such that 
for some $N>0$, 
	\begin{enumerate}
		\item  $\lag \cap ([-N,N]\times\R^{2n-1})$ is compact,
		\item  $\lag \cap ([N,\infty)\times\R^{2n-1})=[N,\infty)\times \Lambda_{+}$, 
		\item  $\lag \cap ((-\infty,-N]\times\R^{2n-1})=(-\infty,-N] \times \Lambda_{-}$, and 
		\item  there exists  a function $f: \Sigma \to \R$  and constants $\mathfrak c_\pm$ such that 
		$i^*\left(e^t\alpha \right)= df$, where $f|_{i^{-1}((-\infty, -N] \times \Lambda_{-})} = \mathfrak c_{-}$, and $f|_{i^{-1}([N, \infty) \times \Lambda_{+})} = \mathfrak c_{+}$.
	\end{enumerate}
\end{definition}

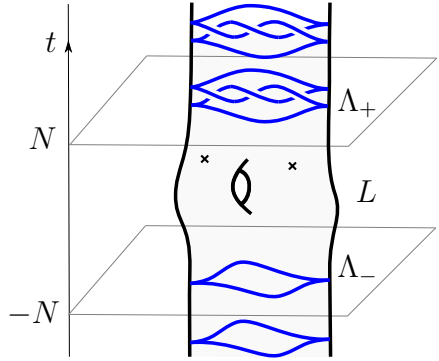
\begin{figure}[!ht]
\labellist
\pinlabel $\Lambda_+$ at 340 300
\pinlabel $\Lambda_-$ at 340 110
\pinlabel $L$ at  350 200
\pinlabel $N$ at -20 260
\pinlabel $-N$ at -30 60
\pinlabel $t$ at -10 370
\endlabellist
\includegraphics[width=2in]{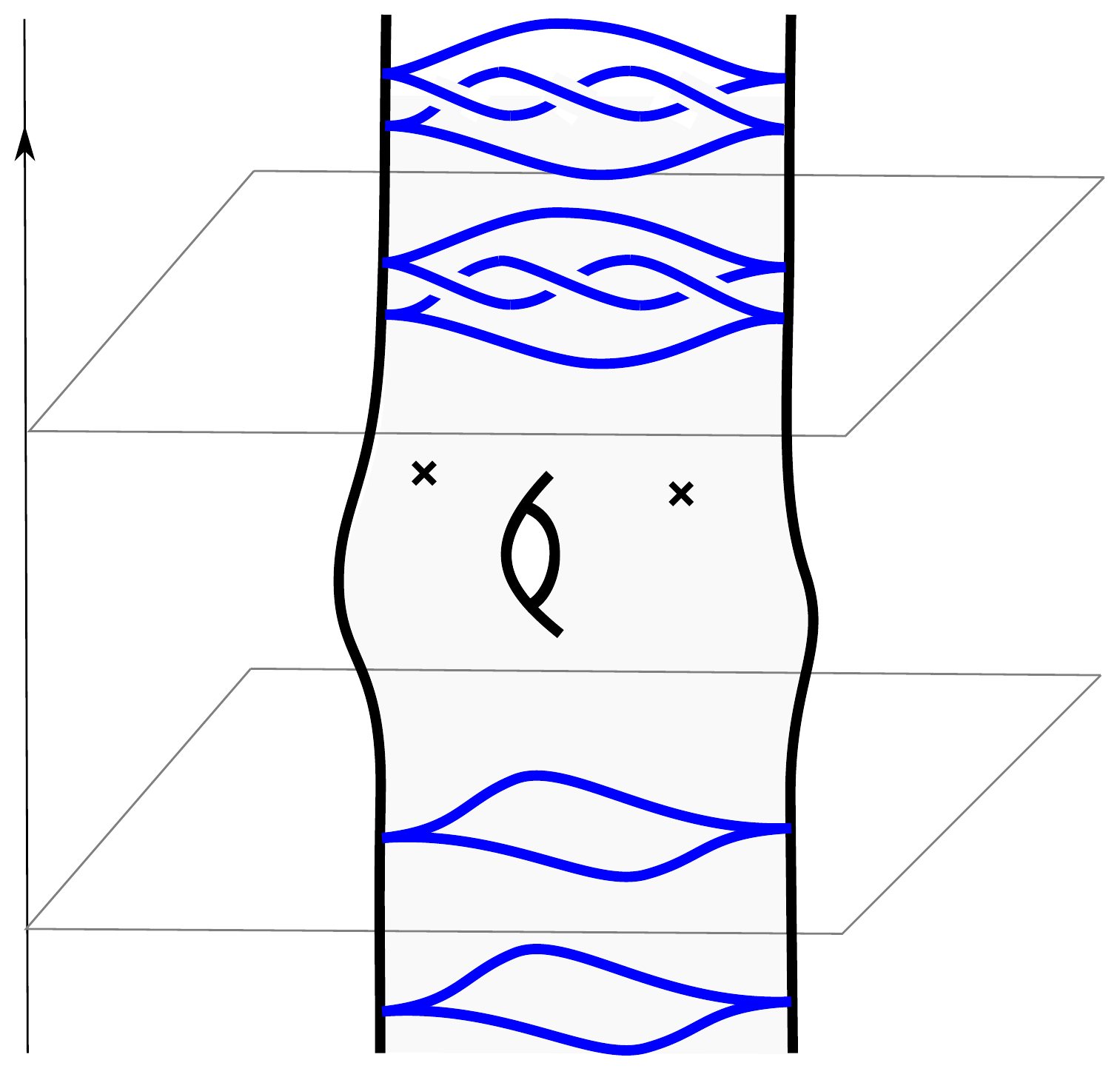}
\caption{A schematic picture of an immersed, exact Lagrangian cobordism $L$ from $\Lambda_-$ to $\Lambda_+$ with genus $1$ and two double points.}
\label{fig:lagcob}
\end{figure}

\begin{remark} \label{rem:cobord}   $\mbox{}$
\begin{enumerate}
\item  The function $f$ in  condition (4) in Definition~\ref{defn:cobord} is a {\bf primitive} of $\lag$.
Since $\Lambda_{\pm}$ are Legendrian, it follows that on the ends of $\lag$, the primitive $f$ is locally constant.
	The condition (4) enforces that when $\Lambda_{-}$ (or $\Lambda_{+}$) is not connected, the constant $\mathfrak{c}_{-}$ (or $\mathfrak{c}_{+}$) does not vary from component to component. By the
	addition of a constant, we can always assume that $\mathfrak{c}_{-} = 0$; this will be the convention that we use in Section \ref{sec:wrap}. 
\item  
 Generically all immersion points of $\lag$ are isolated, transverse double points.  In this paper, when we write $L$ for an immersed exact Lagrangian cobordism we implicitly assume that it comes as the image of an immersion $i:\Sigma\to\R\times\R^{2n-1}$ satisfying the conditions in Definition~\ref{defn:cobord} and that all the immersion points are isolated and transverse double points.
\end{enumerate}
\end{remark}

Given an immersed, exact Lagrangian cobordism $\lag\subset\R\times\R^{2n-1}$ from $\Lambda_-$ to $\Lambda_+$, the primitive $f$ guaranteed by Definition~\ref{defn:cobord}(4) allows one to
construct the {\bf Legendrian lift of} $\lag$, defined as $\widetilde{\lag}=\{(i(q),-f(q))|q\in\Sigma\}$ in the contactization of $\big(\R_t\times\R^{2n-1},d(e^t\alpha)\big)$, which is the contact manifold $\big((\R_t\times\R^{2n-1})\times\R_u,du+e^t\alpha\big)$. 
Double points of $\lag$ are in one-to-one correspondence with {\bf Reeb chords} of $\widetilde{\lag}$, which are trajectories of the Reeb vector field $\frac{\partial}{\partial u}$ that begin and end on $\widetilde \lag$.

The {\bf action of a double point} $X$ of $\lag$ is defined to  be
the {\bf length of the corresponding Reeb chord} $c_X$ of $\widetilde \lag$ starting at  $c^-\in\widetilde \lag$ and ending at $c^+\in\widetilde \lag$, which  is given by $u(c^+)-u(c^-) \geq 0$. 
From our construction of $\widetilde \lag$,  if $X$ is the image of $p_1, p_2 \in \Sigma$  the  action of a double point $X$ is the absolute value of the difference of the primitives at $p_1,$ and $p_2$: $|f(p_1) - f(p_2)|$.

\begin{remark} For an immersed, exact Lagrangian cobordism $L = i(\Sigma)$, the primitive, as defined in Remark~\ref{rem:cobord}, is defined on $\Sigma$, $f: \Sigma \to \mathbb R$.
When all the double points of $L$ have action $0$,  the primitive is a well-defined function $f: L \to \mathbb R$.
\end{remark}

\subsection{Maslov class and index of a double point}\label{subsec:maslov}

We now clarify what we mean by the {\it index} of a double point in an immersed, Maslov-$0$, exact Lagrangian cobordism.
Briefly, the index of a non-zero action double point will be defined in a standard way using the Conley-Zehnder index of the corresponding Reeb chord (of strictly positive length) in the Legendrian lift. We then define the index of an action-$0$ double point of an immersed Lagrangian.

\subsubsection{Maslov index of a loop of Lagrangians and Maslov class of a Lagrangian} \label{sssec:maslov}

First, notice that our Lagrangian cobordisms live in $\big(\R\times\R^{2n-1},d(e^t\alpha)\big)$ which is equivalent via an exact symplectic diffeomorphism to $\big(\R^{2n},\sum dq_i\wedge dp_i\big)$.
Then, there is a standard way of associating an integer, known as the {\bf Maslov index}, to a smooth loop  on an immersed, Lagrangian submanifold in $\R^{2n}$; see, for example, \cite[Section 2.2]{EESu}.  
All examples of Lagrangian cobordisms that we consider in this paper have {\it Maslov class $0$} (denoted {\bf Maslov-$\mathbf 0$}), meaning that all loops have Maslov index $0$. In particular, this implies that the Lagrangians are orientable since the Maslov class modulo $2$ is the first Stiefel--Whitney class. In general, Maslov-$n$ ensures a well-defined $\Z_n$-grading for generators of the Chekanov-Eliashberg DGA (Section \ref{sec:chekanov}) and generators of the Cthulhu complex (Section \ref{sec:cth}); all augmentations and chain maps are also $\Z_n$-graded.

\subsubsection{Index of a double point} \label{ssec:CZ}

Consider an embedded, connected Legendrian $\leg \subset \R^{2n+1}$ and its Lagrangian projection $\pi_{xy}(\Lambda)  \subset\R^{2n}$. Given a Reeb chord $c$ of $\leg$,
a \textit{capping path $\gamma$ along} $\leg$ from the point corresponding to the end of the Reeb chord $c^+$ to the start of the Reeb chord $c^-$ together with a standard closure, as defined in \cite{EESu}, 
gives rise to a smooth loop of Lagrangian subspaces.  The Maslov index of this loop defines the {\bf Conley-Zehnder index} of the Reeb chord $c$, denoted $CZ_\gamma(c)$.  When the Maslov class of the Lagrangian $\pi_{xy}(\leg)$ is $0$,  the Conley-Zehnder index does not
depend on the choice of the capping path along $\leg$, and so we denote it $CZ(c)$. Given this, if $\lag$ is an immersed, Maslov-$0$, exact  Lagrangian with {\it embedded}, Maslov-$0$, Legendrian lift $\widetilde \lag$, a double point $X$ of $ \lag$ lifts to a Reeb chord $c_X$, and we define the {\bf index of $X$} as 
\begin{equation} \ind(X) = CZ(c_X)-1. \label{eqn:dp-index}
\end{equation}

For low-dimensional Legendrians, there is a combinatorial way to compute the Conley-Zehnder index of a  Reeb chord of $\leg$  using a {\it Maslov potential} on the front projection, $\pi_{xz}(\leg)$. Let $\leg$ denote an embedded Legendrian knot  in $\R^{3}_{std}$ (resp. $\R^5_{std}$) with generic front projection, and let $\Lambda_{sing}$ be the subset of $\Lambda$ where the front projection is not an immersion, i.e. the preimage by $\pi_{xz}$ of the set of cusp points (resp. cusp edges and swallow tails). If the Lagrangian $\pi_{xy}(\Lambda)$  has Maslov class $0$, a {\bf Maslov potential} is a locally constant map
$$
\mu: \Lambda/\Lambda_{sing}\to \Z, 
$$ 
such that near a cusp point, or cusp edge, the Maslov potential of the upper sheet is $1$ more than that of the lower sheet.
The Maslov potential is well defined up to a global shift by an integer.
Now let $c$ be a Reeb chord of $\leg$ from $c^{-}$ to $c^{+}$.
In a neighborhood of $c^+$ (resp. $c^-$), $\Lambda$ is the $1$-jet of a Morse function $f_u$ (resp. $f_l$) defined on a neighborhood of $\pi_x(c)$, and $\pi_x(c)$ is a critical point of the function $f_{ul}:=f_u-f_l$.  
Given a Maslov potential $\mu$ on $\Lambda$, we have 
\begin{equation} \label{eqn:CZ}
CZ(c)=\mu(u)-\mu(l)+\ind_{f_{ul}}(\pi_{x}(c)),
\end{equation}
where $u$ and $l$ are the sheets of $\Lambda$ containing $c^+$ and $c^-$ respectively, see \cite[Lemma 3.4]{EESu}.  

In the case when $\Lambda$ is not connected, there is no capping path for Reeb chords between two different components, so we need to make additional choices, as explained in, for example, \cite[Section 3.1]{EHK}.  In particular, 
the capping paths involve the choice of  points in each component of $\leg$ as well as paths between the corresponding Lagrangian tangent spaces at these points.  The Conley-Zehnder index of a particular Reeb chord between components depends on these choices, but for two such Reeb chords, the difference is independent of the choices.  One can again compute the index of a Reeb chord combinatorially using Equation~\ref{eqn:CZ}; the paths determine ``the jump'' of Maslov potential between the two components $\Lambda_i$ and $\Lambda_j$.

The above definition of the index of a Reeb chord applies to the case where the Legendrian $\leg$ is embedded, and so $c^{\pm}$ are distinct points of $\leg$ for each Reeb chord $c$. 
In other words, the double points of the Lagrangian projection $\pi_{xy}(\leg)$ are all of strictly positive action. 
When $\leg$ is immersed and $c$ is a Reeb chord of length $0$, meaning
$c^{+} = c^{-}$ (by assumption this Reeb chord still corresponds to a transverse double point in the Lagrangian projection), the Conley-Zehnder index may depend on the choice of capping path even if $\pi_{xy}(\leg)$ has Maslov class $0$. 
Indeed, for any non-trivial path $\gamma:[0,1]\to\Lambda$ from $c=c^\pm$ to itself starting on one sheet of $\Lambda$ and coming back to $c$ along the other sheet, both $\gamma$ and its reverse $-\gamma$ are capping paths for the  Reeb chord $c$. 
Since in a neighborhood of $c$, $\Lambda$ consists of two sheets meeting tangentially at $c$, using Equation~\ref{eqn:CZ}, we find that 
$$CZ_{-\gamma}(c)=n-CZ_\gamma(c),$$
where $n$ is the dimension of the Legendrian.
Thus if $X$ is an action-$0$ double point of an $n$-dimensional, exact Lagrangian $L$, and $c_X$ denotes the associated length $0$ Reeb chord in the Legendrian lift, then comparing a capping path $\gamma$ and its reverse,
 we have that 
$$\ind_{\gamma}(X)= CZ_\gamma(c_X) - 1 =  (n - CZ_{-\gamma}(c_X) ) - 1 =  n-1 - CZ_{-\gamma}(c_X)  = n-2-\ind_{-\gamma}(X).$$
In particular, when $n=2$,
the index of $X$ using a capping path $\gamma$ or its reverse differs by a sign:
$$\ind_{\gamma}(X)= -\ind_{-\gamma}(X).$$

\begin{definition}\label{defn:ind}
Suppose $X$ is an action-$0$ double point in an $n$-dimensional,  immersed,  Maslov-$0$,  exact Lagrangian. 
The index of $X$ is defined to be the greater of $\ind_{\gamma}(X)$ and $\ind_{-\gamma}(X)$,
 for any capping path $\gamma$ for $c_X$. 
When $n  =2$, we have that $\ind(X)=|\ind_{\gamma}(X)|$.
\end{definition}

The index of a double point arises when considering 
Legendrian Hopf links.
 \begin{definition} \label{defn:Hopf} 
The $(n-1)$-dimensional {\bf Legendrian Hopf link} $\Lambda_{\Ho}^k$  is given by the intersection of the standard local model of an index-$k$ double point of an $n$-dimensional Lagrangian submanifold (namely, $\R^n\cup i \R^n\subset \C^n$) 
and the unit sphere $S^{2n-1}$ with its standard contact structure.
\end{definition}
For $n=2$, we can give a more specific description of the $1$-dimensional Legendrian Hopf link $\Lambda_{\Ho}^k$.

\begin{example}[Hopf links]\label{ex:Hopf}  

	When $n=2$, consider the Hopf link $\leg_{\Ho}^k$ given by the intersection of the local model for an index-$k$ double point of a Lagrangian surface ($\R^2\cup i\R^2\subset \C^2$) and $S^3$. We claim that, potentially after a Legendrian isotopy, there is a front projection of $\leg_{\Ho}^k$ as shown in the leftmost  diagram in Figure~\ref{fig:Hopf}, 
	where the Maslov potential, near the right cusps, from bottom to top, on the four strands is given by $0,1,k+1$ and $k+2$ (up to a global addition of an integer). 
To see this correspondence for $\leg_{\Ho}^0$, we will observe in Lemma \ref{lem:index-0-surgery} that in order to get a {\it Maslov-0} exact Lagrangian cobordism from another {\it Maslov-0}, immersed, exact Lagrangian cobordism on which we perform Lagrangian surgery, the index of the double point we surgered must be $0$. The Hopf link corresponding to this double point (link of the singularity) will thus admit an embedded, Malsov-0, exact Lagrangian filling. From consideration on augmentations and using the Seidel's isomorphism, see Example \ref{ex:aug-hopf}, one can check that $\leg_{\Ho}^0$ is the only Hopf link that bounds an embedded, Maslov-$0$, exact Lagrangian filling. Then, if the double point is of index $k$, the difference in Maslov potential of the two components of $\R^2\cup i\R^2$ must be $k$. Therefore, the boundary $\Lambda_{\Ho}^k$ inherits the required Maslov potential from that of the surface $\R^2\cup i\R^2$.
An explicit Legendrian isotopy via Legendrian Reidemeister moves shows that $\Lambda_{\Ho}^k$ and $\Lambda_{\Ho}^{-k}$ are Legendrian isotopic.
\end{example}

\section{Lagrangian surgery}\label{sec:surgery}

We start this section by reviewing the \emph{Lagrangian surgery} operation on immersed Lagrangian submanifolds, which was first defined for Lagrangian surfaces by Lalonde and Sikorav in \cite{LS} and then generalized to higher dimensions by Polterovich \cite{Pol91}. 
We then prove Theorem~\ref{thm:maingen1}, which translates the existence of immersed fillings into the existence of  embedded cobordisms with the double points
of action 
$0$ being replaced by Hopf links.

\subsection{Lagrangian surgery construction} 

In this subsection, our goal is to prove the following:

\begin{proposition}\label{prop:dp-g} If a Legendrian link $\Lambda\subset \R^3_{std}$ admits an immersed, Maslov-$0$,  exact Lagrangian filling $L$ of genus $g$ with $p$ double points such that one of the double points has index $0$ and action $0$, then $\leg$ also admits an
 immersed,  Maslov-$0$,  exact Lagrangian filling $L'$ of genus $g+1$ with $p-1$ double points.
\end{proposition}

To resolve a double point $X$ of a Lagrangian, we remove a small neighborhood of $X$ and glue back in a {Lagrangian handle}.  
In the setting where the Lagrangian $\lag$ is exact, 
we can 
understand Lagrangian surgery in terms of the Legendrian lift $\widetilde \lag$ of  $\lag$.  This is the approach taken in \cite[Section 6.2]{CMP} where
Casals-Murphy-Presas  give explicit parametrizations of two Lagrangian handles that can be constructed to replace an action-$0$ double point.  The Legendrian lift of one of these
handles can be seen as a ``cusp-sum'', and the Legendrian lift of the other can be seen as a ``cone-sum''; see Figure~\ref{surgery1}. 
These two Lagrangian surgeries are smoothly the same \cite[Proposition 2]{Pol91}.  Observe that $\lag'$ obtained
from either of these surgeries is necessarily exact since it is constructed through its Legendrian lift. The proof of Proposition~\ref{prop:dp-g} then follows immediately from the next lemma that tells us that if the double point has index $0$, the Maslov-$0$ condition is preserved under surgery.
\begin{figure}[!ht]
	\includegraphics[width=4in]{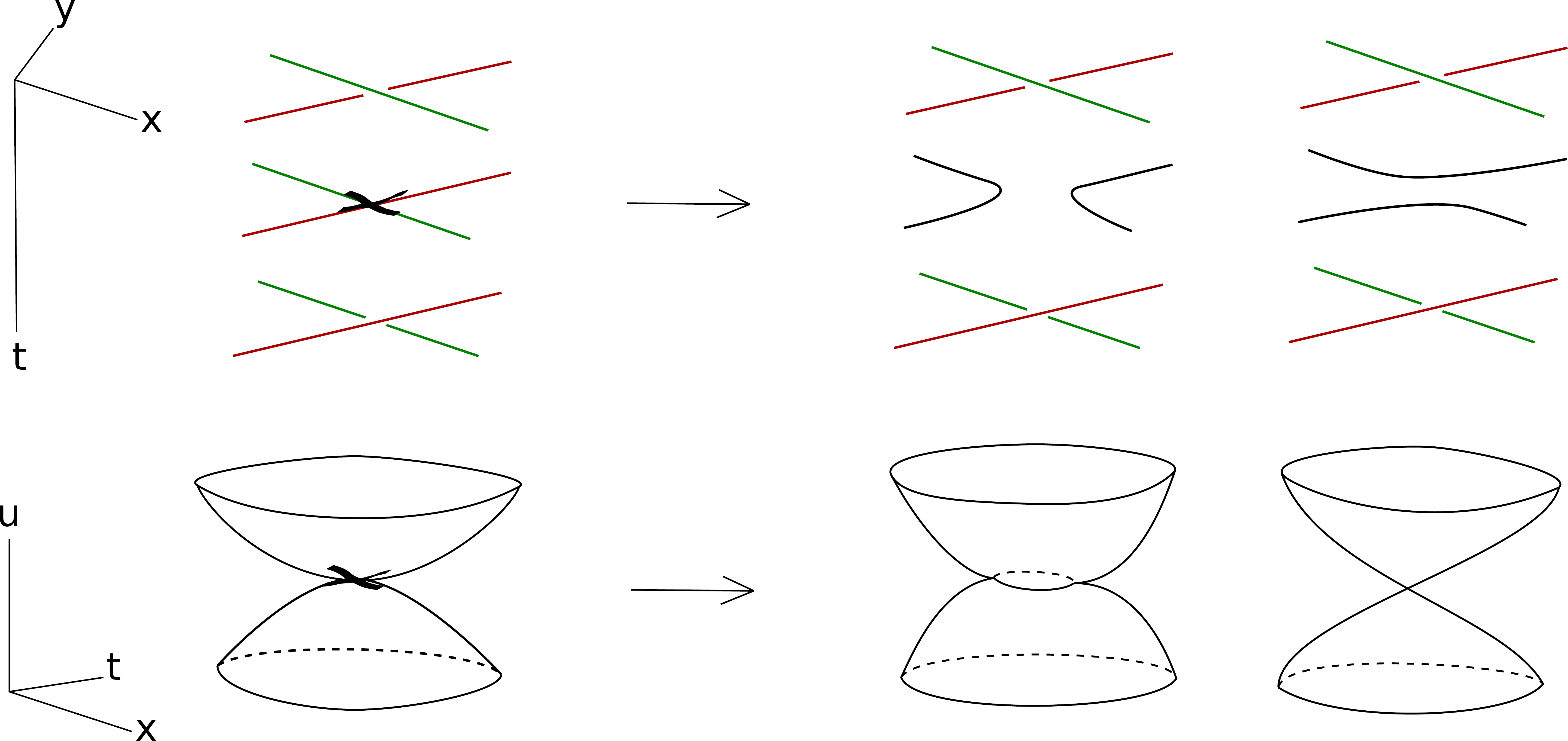}
	\caption{On the top row, left side, are schematized $\pi_{xy}$ slices of the Lagrangian $\lag\subset\R_t\times\R^3$ in a neighborhood of a double point, 
	and on the  top right are slices of the Lagrangian obtained after the two possible handle attachments. The bottom row, left side, schematizes the Legendrian lift $\widetilde{\lag}\subset\R_t\times\R^3\times\R_u$ of $\lag$ and on the right the Legendrian lifts  
	of each handle attachment.}
	\label{surgery1}
\end{figure}

\begin{lemma} (cf. \cite{Pol91, CMP}) \label{lem:index-0-surgery}
Suppose $\lag$ is an immersed,  Maslov-$0$, exact Lagrangian surface that contains an action-$0$  double point  $X$; let  $\lag'$ denote an exact Lagrangian obtained from one of the two  Lagrangian surgeries that correspond to the Legendrian
 ``cusp-sum'' or ``cone-sum'' resolutions of the lift
described above. 
If the  index of $X$ is $0$, then $\lag'$  has Maslov class $0$.
\end{lemma}

\begin{proof} 
The Maslov class of $\lag'$ is $0$ if and only if its Legendrian lift $\widetilde{\lag'}$ admits a ($\mathbb Z$-valued) Maslov potential.
Before surgery, $\lag$ has Maslov class $0$ so its lift $\widetilde{\lag}$ admits a Maslov potential $\mu$.
In the lower left model shown in Figure~\ref{surgery1}, denote the upper and lower sheets  of $\widetilde{\lag}$ by $u$ and $\ell$ respectively.
For both the cusp edge and the cone singularity cases, the Maslov potential $\mu$ can be ``extended'' after surgery to $\widetilde{\lag'}$ if and only if $\mu(u)-\mu(\ell)=1$, 
(see also \cite[Figure 3]{DR:knotted} for the cusp edges arising after perturbing the cone).
The condition $\mu(u)-\mu(\ell)=1$ is equivalent to the condition $\ind(X)=0$ according to Definition~\ref{defn:ind} and Formulas \eqref{eqn:CZ} and \eqref{eqn:dp-index}.
\end{proof}

 \begin{definition}\label{defn:not-from-surgery}  Let $F_{g}^{p}$ denote an immersed, Maslov-$0$, exact Lagrangian filling $F_{g}^{p}$ of a Legendrian $\leg$ with 
genus $g$ and $p$ double points of indices $i_{1}, \dots, i_{p}$ and actions $a_{1}, \dots, a_{p}$. We say that $F_{g}^{p}$  
{\bf arises from Lagrangian surgery} if there exists an immersed, Maslov-$0$, exact Lagrangian filling $F_{g-1}^{p+1}$ of $\leg$ with 
 genus $g-1$ and $p+1$ double points such that 
 \begin{enumerate}
 \item $p$ of the double points have indices $i_{1}, \dots, i_{p}$ and actions $a_{1}, \dots, a_{p}$,
 \item there exists a double point $x_{0}$ of index $0$ and action $0$, and
 \item the Lagrangian surgery corresponding to the Legendrian cusp-sum or cone-sum at $x_{0}$ produces $F_{g}^{p}$. 
\end{enumerate}
If there is no such Lagrangian filling $F_{g-1}^{p+1}$, then we say that $F_{g}^{p}$  
{\bf does not arise from Lagrangian surgery}.
  \end{definition}

\subsection{Proof of Theorem~\ref{thm:maingen1}}

In \cite[Theorem 1.3]{Cha15} Chantraine showed that the existence of an immersed, exact Lagrangian filling of $\leg$ with a single action-$0$ double point  implies the existence of an embedded exact Lagrangian cobordism from a Hopf link  to $\leg$.  In this section, we prove
Theorem~\ref{thm:maingen1}, which generalizes this result to more general cobordisms, more double points, and higher dimensions.

The proof of Theorem~\ref{thm:maingen1} will use the theory of Liouville  structures. Below we briefly describe some of the key terms.  See, for example, \cite[Chapters 11 and 12]{CE} for more details.
 A $1$-form $\lambda$ on a manifold $M$ such that $\omega = d\lambda$ is symplectic is called a 
{\bf Liouville form}; the associated $\omega$-dual vector field $V$, defined by $i_{V}\omega = \lambda$, is the {\bf Liouville vector field} of $\lambda$.  
A {\bf Liouville domain}, $(W, \omega, V)$, is a compact manifold with boundary, $W$, equipped with an exact symplectic structure $\omega=d\lambda$ such that the associated Liouville vector field $V$ points outward along $\partial W$. The boundary $\partial W$ is a contact manifold with contact form $\alpha:=\lambda_{|\partial W}$.
A {\bf Liouville manifold} is a manifold $M$ together with 
a Liouville form $\lambda$, equivalently a triple $(M, \omega=d\lambda, V)$, such that $V$ is complete and $M$ admits an exhaustion $M = \cup_{k}W^{k}$ where $(W^{k}, \omega, V)$ are Liouville domains.
The \textbf{skeleton} of a Liouville manifold $(M, \omega=d\lambda, V)$ is the isotropic set of points that do not escape to infinity under the Liouville flow. More concretely, 
$Skel(M, \omega, V)=\cup_{k=1}^{\infty} \cap_{t>0} \phi^{-1}(W^k)$, where $\cup_k W^k$ is an exhaustion of $M$, and $\phi^t:M\rightarrow M$ is the flow along $V$ for time $t$. A Liouville manifold is obtained from a Liouville domain $W$ by attaching the semi-infinite cylinder $([0, \infty) \times \partial W)$ to $W$ and extend the Liouville form by $e^{t}\alpha$.
For example, 
\begin{equation}\label{eq:LW}
\left(\R^{2n}, \omega_{std}=\sum dq_i\wedge dp_i, V_{rad} = \frac12\sum_{i=1}^{n} \left(q_{i} \frac{\partial}{\partial q_{i}} + p_{i}\frac{\partial}{\partial p_{i}} \right)\right)
\end{equation}
 is a Liouville manifold. 
In a Liouville manifold $(M, \omega, V)$, any hypersurface $\Sigma\stackrel{i}\hookrightarrow M$ transverse to $V$ is a contact manifold, with contact form given by $\alpha = i^{*}\lambda$.
 For any Legendrian  $\leg \subset \Sigma$, flowing $\leg$ along $V$ defines a {\bf Lagrangian} that is cylindrical over $\leg$.
Weinstein domains are Liouville domains with a compatible Morse handlebody decomposition. For $k\leq n$, a $2n$-dimensional {\bf Weinstein handle of index $k$}  has underlying Liouville domain given as  
$\left(\B^{k} \times \B^{2n-k}, \omega_{std}, V_{k} \right),$
where  
$$ 
\begin{aligned}
\omega_{std}=\sum_{i=1}^{n} dq_i\wedge dp_i,  \quad V_{k} &= \sum_{i=1}^{k} \left( -q_{i} \frac{\partial}{\partial q_{i}} + 2p_{i}\frac{\partial}{\partial p_{i}}\right) + \frac12 \sum_{i=k+1}^{n} \left( q_{i} \frac{\partial}{\partial q_{i}} + p_{i}\frac{\partial}{\partial p_{i}} \right). \\
 \end{aligned}
$$
The {\bf core} (respectively, {\bf cocore)} of the $k$-handle is $\B^{k} \times \{0\}$ (respectively, $\{0\} \times \B^{2n-k}$)  and the handle has {\bf attaching sphere} given by the boundary of the core, $S^{k-1} \times \{0\}$.
It is possible to build Weinstein cobordisms via attaching handles by gluing the isotropic attaching sphere to isotropic spheres in the contact level sets, \cite[Proposition 11.13]{CE}.
 
\begin{proof}[Proof of Theorem~\ref{thm:maingen1}] Let $L^{\times}$ be an immersed,  Maslov-$0$,
exact Lagrangian cobordism from $\Lambda_-$ to $\Lambda_+$ with $p$ double points,  $m$ of which, $x_{1}, \dots, x_{m}$, have action $0$. By Definition~\ref{defn:cobord}, we know that the value of the primitive is constant along all components of $\leg_-$. For the reader's convenience, we outline the argument.

\begin{enumerate}
\item Map $(\R_t\times \R^{2n-1}, d(e^t \alpha))$ to $(\R^{2n}- \{\text{ray}\}, \omega_{std}=\sum dq_i\wedge dp_i)\subset (\R^{2n}, \omega_{std})$ with an exact symplectomorphism
so that $L^{\times}$ is sent to an exact Lagrangian $\widetilde L^{\times}$ that is cylindrical outside of $\B_{0}(\rho_{+})$ and inside $\B_{0}(\rho_{-})$, where $\B_{0}(\rho)$ is the standard Euclidean ball centered at $0$ of radius $\rho$. 
\item Change the Liouville structure on  $\R^{2n}$ from $(\omega_{std}, V_{rad})$ to a Liouville structure $(\omega_{std}, V^{rad}_{\#})$ so that a ``\emph{multi-dumbbell region}'' $\D_{\#} \subset \B_{0}(\rho_{-})$  has a
 Liouville structure obtained from attaching $m$ ``exterior'' Weinstein $0$-handles to a ``center'' Weinstein $0$-handle via $m$ Weinstein $1$-handles.  
\item  Apply a Hamiltonian isotopy to drag the double points of $\widetilde L^{\times}$ to the center of the exterior $0$-handles of $\D_{\#}$ 
and move $\widetilde L^{\times}$
to agree with standard intersecting Lagrangian disks near each double point. 
Now $\widetilde L^{\times} \cap \partial  \D_{\#} = \widetilde \leg_{-}$ consists of the disjoint union of $m$ {\it Legendrian} Hopf links and the Legendrian link corresponding to $\leg_{-}$.  Since, by hypothesis, the $m$ double points all have action $0$,
we can guarantee that the primitive evaluates to the same constant  on the two components of each Hopf link.  Thus  on the components of $\widetilde \leg_{-}$, the primitive agrees with the constants 
$c_{0}, c_{1}, \dots, c_m$. 
\item By modifying $V_{\#}^{rad}$ inside $\D_{\#}$, we change the Liouville structure from $(\omega_{std}, V_{\#}^{rad})$ to $(\omega_{std}, V_{0})$ so that $V_0$ only vanishes at the origin. Furthermore, we also ensure that on a small ball $\B_{0}(\epsilon)\subset  \operatorname{Int}(\D_{\#})$, $V_{0}$ agrees with the radial Liouville structure. 
The flow of the Liouville vector field $V_{0}$ defines an exact Lagrangian cylinder $L_{V_{0}}$ over the Legendrian $\widetilde \leg_{-}  \subset \partial \D_{\#}$.  We construct a new, immersed, Maslov-$0$, exact Lagrangian cobordism $\widehat L$ with only $(p-m)$ double points by replacing 
$\widetilde L^{\times} \cap \operatorname{Int} \D_{\#}$ with the Lagrangian cylindrical end, $L_{V_{0}} \cap (\D_{\#} \setminus \operatorname{Int}(\B_0(\epsilon))$.  
Since $\partial(\B_0(\epsilon))$ is transverse to $V_{0}$,  
$\widehat \leg_{-} = \widehat L \cap \partial (\psi_{1}(B_0(\epsilon))$ is Legendrian, and by construction the primitive continues to agree with the constants $c_{0}, c_{1}, \dots, c_{m}$ on the components of $\widetilde \leg_{-}$.
 \item  Sard's Theorem guarantees the existence
  of a ray that avoids $\widehat L$.  This allows us to map the Lagrangian cobordism $\widehat L$ back to an immersed, Maslov-$0$, exact Lagrangian cobordism  $L \subset (\R_t\times \R^{2n-1}, d(e^t \alpha))$ with only $(p-m)$ double points.  By applying another
  Hamiltonian isotopy, we can guarantee that the primitive agrees with the same constant on all components of the negative end.
 \end{enumerate}

We now give more details for these steps.

\textit{Step 1:} As shown in, for example, \cite[Proposition 2.1.8]{Geiges} there is a contactomorphism $$\kappa: (\R^{2n-1}, \ker\alpha) \to \left(S^{2n-1}-\{pt\}, \ker \left(\frac{1}{2}\left(\sum q_i dp_i- p_i dq_i\right)\right) \right).$$ 
This contactomorphism lifts to an exact symplectomorphism between the symplectizations:
$$
\begin{aligned}
\widetilde \kappa: (\R_t\times \R^{2n-1}, d(e^t \alpha))&\to \left(\R^{2n}- \{ray\}, \omega_{std}=\sum dq_i\wedge dp_i\right) \\
\widetilde \kappa(t, p)&\mapsto  t\kappa(p).
\end{aligned}
$$
We can view the image of $\widetilde \kappa$ as a subset of the Liouville manifold $(\R^{2n}, \omega_{std}, V_{rad})$, as defined in Equation~\eqref{eq:LW}. Then, $\widetilde L^{\times} := \widetilde \kappa(L^{\times})$ is  an immersed, Maslov-$0$, exact Lagrangian surface that is cylindrical over the Legendrians $\kappa(\leg_{\pm})$ with respect to the radial Liouville vector field $V_{rad}$.
 In particular, if $L^{\times}$ is cylindrical outside $t_{\pm}$, there exist $\rho_{\pm}$ such that   $\widetilde L^{\times}$ is cylindrical outside $\B_0\left(\rho_{\pm}\right)$, which are balls with respect to the standard Euclidean metric of radius $\rho_{\pm}$ centered at the origin.

In the next steps, we will work in $(\R^{2n}, \omega_{std})$.  In Step 5, we will guarantee the existence of a ``ray'' that will allow us to transfer our Lagrangian back to 
$(\R_t\times \R^{2n-1}, d(e^t \alpha))$.

\textit{Step 2:}  
Choose $y_{1}, \dots, y_m \in \B_{0}(\rho_{-})$, and consider  balls $\B_{y_1}, \dots,  \B_{y_m} \subset \B_{0}(\rho_{-})$ centered at $y_{1}, \dots, y_{m}$, and  attach each of  these balls via radial paths $\delta_{1}, \dots, \delta_m$ to a disjoint center ball 
$\B_{0} \subset  \B_{0}(\rho_{-})$ centered at the origin.  
 View the balls $\B_{0}$ and $\B_{y_{k}}$, $k=1, \dots, m$, as Weinstein $0$-handles and construct $m$ Weinstein $1$-handles with core  $\delta_k$.  Thus, it is possible to glue these Weinstein structures together to obtain a Weinstein structure on a  neighborhood of a dumbbell region  $\D_{\#}$, \cite[Proposition 11.13]{CE};
  see Figure~\ref{fig:dumbbell} for a schematic. 
 Let  $(\D_{\#}, \omega_{std}, V_{\#})$ denote the resulting Liouville domain.
 Now we define a new Liouville structure  $(\R^{2n}, \omega_{std}, V_{\#}^{rad})$ that agrees with $(\omega_{std}, V_{rad})$ outside a neighborhood  of $\D_{\#}$ and with the Liouville structure $(\omega_{std}, V_{\#})$ on $\D_{\#}$. Let $N(\D_{\#})$ denote a contractible neighborhood of $\D_{\#}$ where $V_{\#}$ is defined.
 Let $\lambda_{rad}$ and $\lambda_{\#}$ denote the  Liouville $1$-forms for $V_{rad}$  and $V_{\#}$ in $(\R^{2n}, \omega_{std})$. Since  $d(\lambda_{\#}-\lambda_{rad})=\omega_{std}-\omega_{std}=0$, and all closed $1$-forms on $N( \D_{\#})$ are exact,  we know $\lambda_{\#}-\lambda_{rad}=dH$ for some function $H: N(\D_{\#}) \rightarrow \R$.
 Let $\sigma$ be a smooth bump function for $\D_{\#}$ supported on $N(\D_{\#})$:   $\sigma(p)=1$ for all $p \in \D_{\#}$,
  and $\operatorname{supp}\sigma \subset N(\D_{\#})$. 
Then consider $\lambda_{\#}^{rad}=\lambda_{rad}+d(\sigma H)$. On $\D_{\#}$, $\lambda_{\#}^{rad}=\lambda_{\#}$, while on the complement of $N(\D_{\#})$, $\lambda_{\#}^{rad}=\lambda_{rad}$. By construction, 
$\lambda_{\#}^{rad}$ is a Liouville $1$-form of $(\R^{2n},\omega_{std})$, so it provides a uniquely defined Liouville vector field $V_{\#}^{rad}$ on $(\R^{2n}, \omega_{std})$.
By construction of $\lambda_{\#}^{rad}$,  $\widetilde L$ is still exact in the new Liouville manifold $(\R^{2n}, \omega_{std}, V_{\#}^{rad})$.

\textit{Step 3:}  
By the $n$-transitivity of Hamiltonian isotopies, see for example \cite[Theorem A]{Boothby}, we can assume that after applying a compactly supported Hamiltonian isotopy
     the double points $x_{k}$ are at the point $y_k$ for $k=1, \dots, m$.  
By Moser arguments (as in, for example, \cite[Section 3.3]{MS95}), we can further assume that, after applying a Hamiltonian isotopy,
the immersed $\widetilde L^{\times} $ agrees with standard intersecting Lagrangian disks passing through $y_{k}$ parallel to $\mathbb R^{n}$ and $i\mathbb R^{n}$.
  Then   $\widetilde \leg_{-} := \widetilde L^{\times} \cap \partial \D_{\#}$ consists of $m$ Legendrian Hopf link and the Legendrian $\kappa(\leg_{-})$, 
 and the immersed $\widetilde L^{\times}$ is cylindrical over the Legendrians $\kappa(\Lambda_\pm)$.  
 By exactness of $\widetilde L^{\times}$,
  $\lambda_{\#}^{std}|_{\widetilde L^{\times}} = d \widetilde f$, for $\widetilde f: \Sigma  \to \mathbb R$, where
 $\widetilde L^{\times}$ is the immersed image of $\Sigma$. 
 Observe that on the intersecting Lagrangian disks at $y_k$, $\lambda_{\#}^{std} = 0$.  Thus $\widetilde f$ is constant
 on each of these disks, and this constant must agree with $\widetilde f(y_k)$.   Letting $\widetilde f(y_{k}) = c_{k}$, $k=1,\dots, m$, we then know that the primitive restricts to the constant $c_{k}$ on the $k$-th Hopf link in $\widetilde \leg_{-}$.  By hypothesis, $\widetilde f$ is constant on the Legendrian $\kappa(\leg_{-}) \subset \partial \B_{0} \cap \partial \D_{\#}$; we denote this constant by $c_{0}$. 

 \begin{figure}[!ht]
\labellist
\small
\pinlabel $\B_0(\rho_{-})$ at  18 90
\pinlabel ${\color{red} L}$ at  78 385
\pinlabel $\kappa(\Lambda_{-})$ at  334 431
\pinlabel $N(\D_{\#}^2)$ at  370 150
\pinlabel $\D_{\#}^2$ at  200 180
\pinlabel $y_1$ at  100 275
\pinlabel $y_2$ at  380 295
\endlabellist
\includegraphics[width=3in]{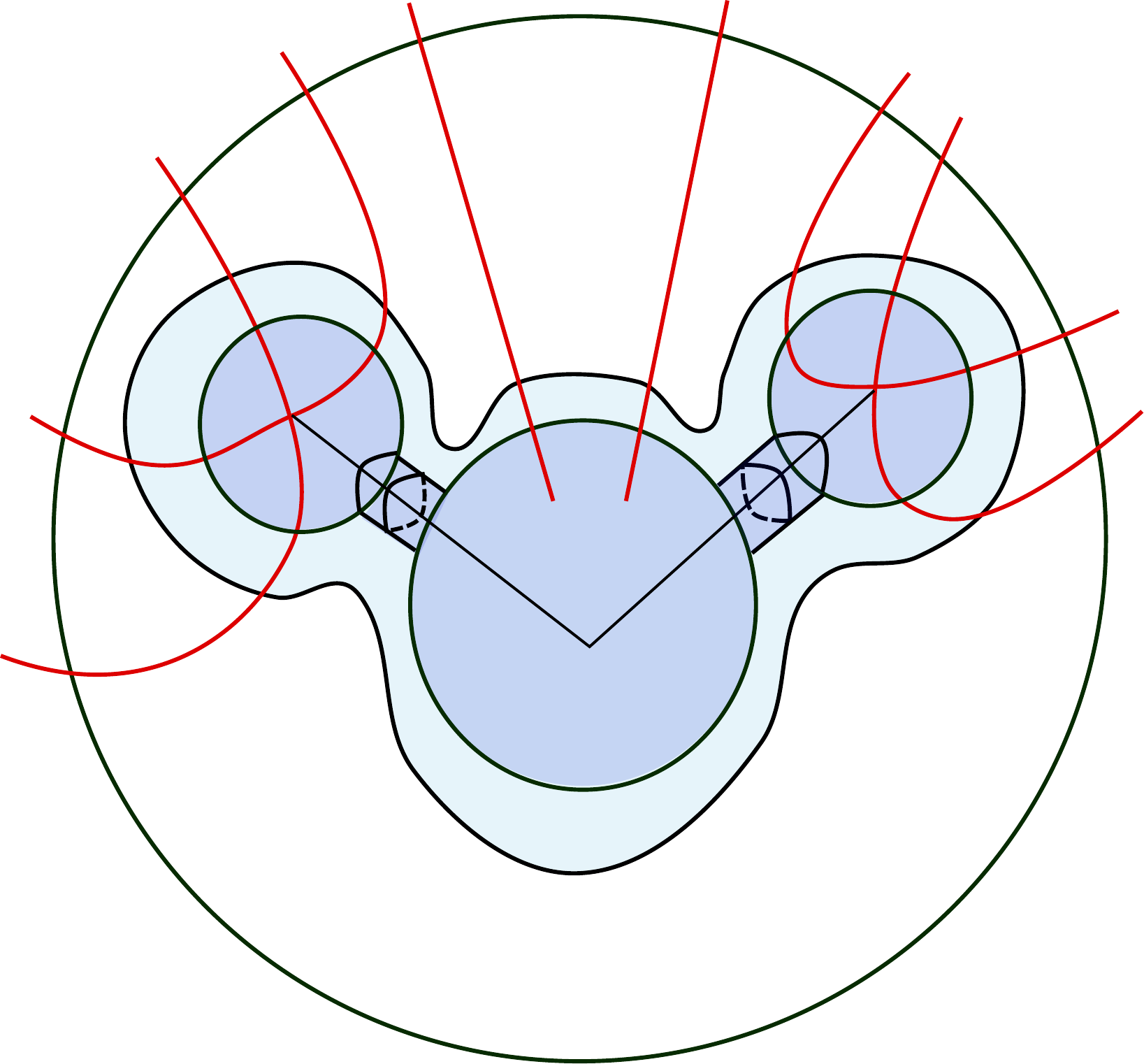}
\caption{A schematic of the dumbbell region $\D_{\#} \subset \operatorname{Int} \B_0(\rho_{-})$.}
\label{fig:dumbbell}
\end{figure}

\textit{Step 4:}   By construction, the skeleton of $V_{\#}^{rad}$, $\skel(V_{\#}^{rad})$, consists of the origin, the points $y_k$, and the paths $\delta_k$ between the origin and $y_k$ for $k=1, \dots, m$.  
Choose $\epsilon>0$ such that $\B_0(\epsilon) \subset \operatorname{Int}\D_{\#}$, and fix a neighborhood $N_{0} \subset  \operatorname{Int}\D_{\#}$ containing $\skel(V_{\#}^{rad}) \cup \B_{0}(\epsilon)$.
We will change the Liouville structure from  $(\R^{2n}, \omega_{std} = d\lambda_{std}, V_{\#}^{rad})$ to $(\R^{2n}, \omega_{std}=d\lambda_0, V_0)$ where  $V_0$ agrees with the radial Liouville vector field $V_{rad}$ on $\B_{0}(\epsilon) \cup \skel(V_{\#}^{rad})$, and  $V_0$ agrees with  $V_{\#}^{rad}$ on $\R^{2n}\setminus \operatorname{Int} \D_{\#}$.
Since both $\lambda_{rad}$ and $\lambda^{rad}_{\#}$ are Liouville $1$-forms for $(\R^{2n}, \omega_{std})$, then, as argued in Step 3, $\lambda_{rad}-\lambda^{rad}_{\#}=dH_{0}$ for some function $H_{0}:  
 \R^{2n} \to \R$. 
 Let $\sigma_{0}: \R^{2n}\rightarrow [0,1]$ be a smooth bump function for $\B_{0}(\epsilon) \cup \skel(V_{\#}^{rad})$ supported in $N_{0}$: $\sigma_0(p) \equiv 1$ for all $p \in \B_{0}(\epsilon) \cup \skel(V_{\#}^{rad})$,   $\operatorname{supp} \sigma_{0} \subset N_{0}$. We will also choose $\sigma_{0}$ such that $\nabla \sigma_0$ is parallel to $-V_{\#}^{rad}$.

 Now consider $\lambda_0=\lambda_{\#}^{rad}+d(\sigma_{0}H_{0})$. On $\B_0(\epsilon)\cup \skel(V_{\#}^{rad})$, $\lambda_0=\lambda_{rad}$, while on the complement of $N_{0}$, we have that $\lambda_0=\lambda_{\#}^{rad}$. By construction $\lambda_0$ is a Liouville $1$-form of $(\R^{2n}, \omega_{std})$ so it provides a uniquely defined Liouville vector field $V_0$ on $(\R^{2n}, \omega_{std})$.

We now show that on $\D_{\#}$, $V_{0}$ vanishes only at  the origin. 
First observe that on $\B_{0}(\epsilon) \cup \skel(V_{\#}^{rad})$, $V_{0} = V_{rad}$ and thus $V_0$ only vanishes at the origin within this subset.
Let $F_0=\sigma_0 H_0$.   Observe that, with respect to the standard almost complex structure $J$,  $\iota_{-J\nabla F_0}\omega_{std}=dF_0$. On $\D_{\#}$, $V_{\#}^{rad}=V_{\#}$ and $\lambda_{\#}^{rad}=\lambda_{\#}$,  
 and thus, on $\D_{\#}$, $V_{0} = V_{\#} - J\nabla F_{0}$.
  To ensure that  
  $V_0\neq 0$ on 
   $\D_{\#} \backslash (\B_{0}(\epsilon) \cup \skel(V_{\#}^{rad}))$,  
 we need to show that  
 $$ -J\nabla F_0\neq -V_{\#}, \qquad \forall p \in \D_{\#}\backslash  (\B_{0}(\epsilon) \cup \skel(V_{\#}^{rad})).
 $$
   Since $\iota_{-J\nabla H_0}\omega_{std}=dH_0 = \lambda_{rad} - \lambda_{\#}^{rad}$, we see that on $\D_{\#}$, $-J\nabla H_{0} = V_{rad} - V_{\#}^{rad} = V_{rad} - V_{\#}$. Thus
$$-J\nabla F_0= -J\nabla (\sigma_{0}H_{0})= -\sigma_0 J \nabla H_0- H_0J\nabla \sigma_0= \sigma_0(V_{rad}-V_{\#})+H_0(-J\nabla \sigma_0), \text{ on } \D_{\#}.$$
By construction, $\sigma_0$ is a smooth bump function such that $\nabla \sigma_0$ is parallel to $-V_{\#}^{rad} = -V_{\#}$, and thus $-J\nabla \sigma_0$ is perpendicular to $-V_{\#}$. 
Thus we see that for any  
$p \in \D_{\#} \backslash  (\B_{0}(\epsilon) \cup \skel(V_{\#}^{rad}))$ such that 
$-J\nabla F_0 = \sigma_0(V_{rad}-V_{\#})+H_0(-J\nabla\sigma_0)=-V_{\#}$, we know that the following properties hold:
\begin{enumerate}
\item  
 $V_{rad} - V_{\#}$ is in the $2$-plane spanned by $V_{\#}$ and $JV_{\#}$, and thus
$V_{rad}$ is contained in the $2$-plane spanned by $V_{\#}$ and $JV_{\#}$,  
 \item  $\langle \sigma_0(V_{rad}-V_{\#}), V_{\#} \rangle = -\| V_{\#} \|^{2}$,  and  
\item $\langle \sigma_0(V_{rad}-V_{\#}), JV_{\#} \rangle JV_{\#}+H_0(-J\nabla \sigma_0)=0.$
\end{enumerate}
The solution space $S$ for conditions (1) and (2) is a closed subset of $\D_{\#}$ and is thus a bounded set. By globally shifting $H_0$ by some constant, we can ensure that the third property is not satisfied for any $p\in S$. Although $\lambda_0$ is determined by $\sigma_0$ and $H_0$, the shift of $H_0$  does not affect our previous computations.

Let $L_{V_{0}}$ denote the Lagrangian cylinder formed by the trajectories of $V_{0}$ through
 $\widetilde \leg_{-} \subset \partial \D_{\#}$. 
Let us justify that $L_{V_0}$ is an exact Lagrangian. Let $S_{0}(\epsilon) = \partial (\B_{0}(\epsilon))$.
By construction, $V_0$ vanishes only at the origin in $ \D_{\#}$ and  is outwardly transverse to $S_{0}(\epsilon)$ and $\partial \D_{\#}$. Thus there is a canonical map
$$\begin{aligned}
G:\big([0,+\infty)\times S_0(\epsilon), e^t\alpha_0\big)&\rightarrow\big(\R^{2n}\setminus\operatorname{Int}(\B_{0}(\epsilon)),\lambda_{0}\big) \\
(t,x) &\mapsto \psi^{V_0}_t(x),
\end{aligned}$$
 where $\psi^{V_{0}}_{t}$ is the time-$t$ flow of $V_{0}$ and  $\alpha_0$ is the pullback of $\lambda_0$ to $S_0(\epsilon)$.
By the nonvanishing of $V_{0}$ on $\D_{\#}\backslash \{0\}$, we see that $G$ 
is surjective onto $\D_{\#} \setminus\operatorname{Int} \B_{0}(\epsilon)$, and under this map $G$, the boundary of the dumbbell, $\partial\D_\#$, is graphical over $S_0(\epsilon)$, i.e.  there exists a smooth function $F:S_0(\epsilon)\to\R^{+}$  such that $$\partial\D_\#=G(\Gamma_{F}), \quad \text{ where } \Gamma_{F}=\{(t,x)~|~t=F(x)\}\subset[0,+\infty)\times S_0(\epsilon).$$ Note that $\Gamma_{F}$ is a contact type hypersurface with induced contact form $e^F\alpha_0$.
Suppose that $\widetilde\leg_{-} \subset \partial \D_{\#}$  flows backwards along $V_{0}$ to $\leg_{S_0(\epsilon)} \subset L_{V_{0}} \cap S_{0}(\epsilon)$.  Equivalently, we can write
$$\widetilde\leg_-=\left\{\psi^{V_0}_{F(x)}(x)~|~x\in\leg_{S_0(\epsilon)}\right\} = L_{V_{0}} \cap \partial \D_{\#}.$$
Letting
$$\leg_{\Gamma_{F}}=\{(t,x)~|~x\in\leg_{S_0(\epsilon)},t=F(x)\} \subset \Gamma_{F},$$ 
we see that $\widetilde\leg_-=G(\leg_{\Gamma_{F}}) \subset \partial \D_{\#}$.
If we denote the inclusion $\partial\D_\#\stackrel{j}\hookrightarrow\R^{2n}$ then we have
$G^*j^*\lambda_0=e^F{\alpha_0}|_{\Gamma_{F}},$
and thus
$e^F{\alpha_0}|_{\leg_{\Gamma_{F}}}=0$ as $\widetilde \leg_{-} \subset \partial \D_{\#}$ is Legendrian. So we get $\leg_{\Gamma_{F}}\subset (\Gamma_{F},e^F\alpha_0)$ is Legendrian, which implies by construction of $\leg_{\Gamma_{F}}$ that $\alpha_0$ vanishes on $\leg_{S_0(\epsilon)}$, and thus $\leg_{S_0(\epsilon)}\subset(S_0(\epsilon),\alpha_0)$ is Legendrian.
The cylinder $L_{V_0}$ is the image by $G$ of 
$$\mathcal{L}_0=\{(sF(y),y)~|~y\in\leg_{S_0(\epsilon)},s\in[0,1]\}\subset[0,+\infty)\times S_0(\epsilon)$$
which is an exact Lagrangian (codimension 0 submanifold of $[0,+\infty)\times\leg_{S_0(\epsilon)}$) on which the primitive is constant. Thus $L_{V_0}=\{\psi^{V_0}_{sF(x)}(x)~|~x\in\leg_{S_0(\epsilon)},s\in[0,1]\}$ is also an exact Lagrangian with constant primitive.
 
We can now construct a new, immersed, Maslov-$0$, Lagrangian cobordism $\widehat L$ with only $(p-m)$ double points by replacing 
$\widetilde L^{\times} \cap \operatorname{Int} \D_{\#}$ with  $L_{V_{0}} \cap (\D_{\#} \setminus \operatorname{Int}( \B_0(\epsilon))$.  
Since $\partial(\B_{0}(\epsilon))$ is transverse to $V_{0}$,  
$\widehat \leg_{-} = \widehat L \cap \partial( \B_{0}(\epsilon))$  
is Legendrian.  A direct calculation shows that $\lambda_{0}|_{L_{V_{0}}}= 0$, and thus the primitive $\widehat f$ of the exact Lagrangian $\widehat L$ evaluates
to the same constants $c_0, c_1, \dots, c_m$ on the components of $\widehat \leg_{-}$ (as was the case for the evaluation of the primitive $\widetilde f$ on all components of the Legendrian $\widetilde \leg_{-}$ of $\widetilde L^{\times}$).  Moreover, since the Maslov potential of the $k$-th Hopf link $\Lambda_{\Ho}^{i_k}$ is  inherited from the Maslov potential of $\widetilde L^{\times}$, replacing part of the surface does not affect the Maslov-$0$ condition.

\textit{Step 5:}  
To send $\widehat L$ back to $(\R_t\x \R^{2n-1}, d(e^t\alpha))$, all we need to do is 
find a ray that does not intersect $\widehat L$ and then apply the exact symplectomorphism $\widetilde{\kappa}^{-1}: (\R^{2n}- \{ray\}, \omega_{std})\rightarrow (\R_t\x \R^{2n-1}, d(e^t\alpha))$.
We can ensure the existence of the ray for the following reason.
Note that $\widehat L$ is an Lagrangian immersion $i(\Sigma)$ for  $i: \Sigma\to \R^{2n}-\{0\}$ to $\widehat L$, where $\Sigma$ is an $n$-dimensional embedded surface. 
We can project $\widehat L$ to the unit sphere $S^{2n-1}$ and get a smooth map from $\Sigma$ to $S^{2n-1}$.
By Sard's Theorem, this map cannot be surjective for  $n\ge 1$, and therefore   we can always find a ray that does not intersect $\widehat L$. Once back in $(\R_t\x \R^{2n-1}, d(e^t\alpha))$,
by a Hamiltonian isotopy we can adjust the primitives to be the same constant on all components at the negative end (see, for example, \cite[Section $10.1$]{CDRGG}).
\end{proof}

\section{Legendrian Contact Homology}\label{sec:dga}
 
In this section we recall the definition of Legendrian contact homology, which was originally formulated by  Chekanov \cite{Che} and Eliashberg \cite{Eli}. We recall also the definition of augmentations and of linearized and bilinearized Legendrian contact homology. Throughout this section, we follow notations and conventions of \cite{CDRGG} and refer to this paper for more details.  More details about the situation when coefficients are taken in a field can be found, for example, in \cite{EESorientation} or \cite{NE}.

\subsection{Chekanov-Eliashberg DGA}\label{sec:chekanov}  
{Here we give the key definitions and set the notation that we will use.  A careful description of the Chekanov-Eliashberg DGA can be found, for example, in~\cite{NE, CDRGG}.  

The \textbf{Chekanov-Eliashberg differential graded algebra (DGA)} of $\Lambda$, $(\mathcal{A}(\Lambda), \partial)$ is the unital, graded algebra over a commutative ring $\F$ generated by Reeb chords of $\Lambda$. Let $R(\Lambda)$ denote the set of Reeb chords of $\Lambda$.
The grading on $\mathcal{A}(\Lambda)$ is defined on the Reeb chord generators by
\begin{equation} \label{eqn:CZ-grade}
|c|=CZ(c)-1,
\end{equation}
where $CZ(c)$ is as described in Section~\ref{ssec:CZ}.  
The differential $\partial$ on $\alg(\Lambda)$ is defined by a count of rigid pseudo-holomorphic disks in the symplectization $(\R_t\times\R^3,d(e^t\alpha))$, with boundary on $\R\times\Lambda$.  
For any Reeb chords $a, b_1, \ldots, b_m \in R(\Lambda)$, 
and   any almost complex structure $J$ which is a cylindrical lift  of an admissible almost complex structure on $\R^2$ (see \cite[Section 2.2]{CDRGG}),
define the {\bf LCH moduli space} $\widetilde{\mathcal{M}}_J^{\R\times\Lambda}(a;b_1,\ldots,b_m)$ to be the space of $J$-holomorphic maps
$u:(D^2_{m+1}, \partial D^2_{m+1})\rightarrow (\R\times\R^3, \R\times\Lambda),$
with a \textit{positive} asymptotic to the Reeb chord $a$
and \textit{negative} asymptotics to the Reeb chords $b_1,\dots,b_m$, up to conformal reparametrization of the domain; see \cite[\S3.2.3]{CDRGG}. This moduli space admits an $\R$-action by translation along the symplectization direction; we let $$\mathcal{M}_J^{\R\times\Lambda}(a;b_1,\ldots,b_m)$$ denote the quotient of $\widetilde{\mathcal{M}}_J^{\R\times\Lambda}(a;b_1,\ldots,b_m)$ by $\R$. 
 A disk $u\in\mathcal{M}_J^{\R\times\Lambda}(a;b_1,\ldots, b_m)$ is called {\it rigid} if $\dim \mathcal{M}_J^{\R\times\Lambda}(a;b_1,\ldots, b_m)=0$. 
  Compactness results ensure that there are finitely many rigid holomorphic disks, which are used to define the differential $\partial$:
$$\partial(a)=\displaystyle{\sum_{\dim\left(\mathcal{M}_J^{\R\times\Lambda}(a;b_1,\ldots, b_m)\right)=0}} 
|\mathcal{M}_J^{\R\times\Lambda}(a;b_1,\ldots, b_m)|b_1\dots b_m.$$
The \textbf{Legendrian contact homology of $\Lambda$}, denoted $LCH_*(\Lambda)$, is the homology of $(\alg(\Lambda),\partial)$.

\begin{example}[DGA of Hopf links] \label{ex:DGA-hopf}
	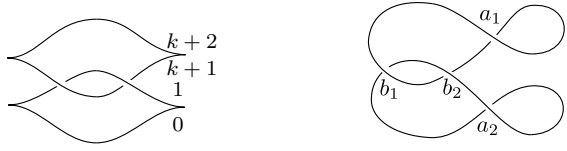
\begin{figure}[!ht]
	\labellist
	\tiny
	\pinlabel $0$ at 130 20
	\pinlabel $1$ at  130 45
	\pinlabel $k+1$ at 140 60
	\pinlabel $k+2$ at 140 80
	\endlabellist
		\includegraphics[width=3in]{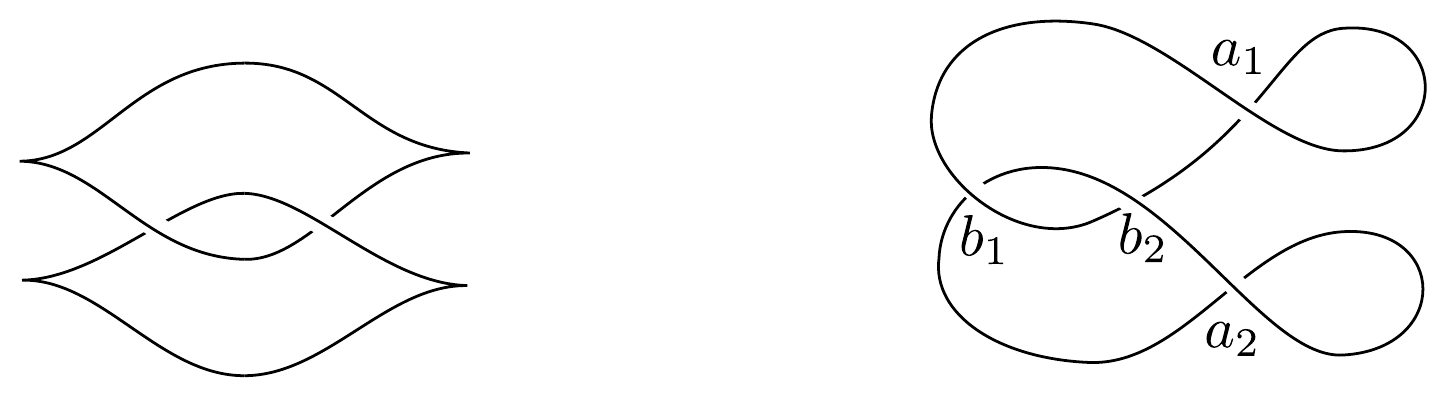}
		\caption{The front and Lagrangian projections of $\Lambda_{\Ho}^k$.}
		\label{fig:Hopf}
	\end{figure}
 Consider the Hopf link $\Lambda_{\Ho}^k$ whose front and Lagrangian projections as well as Maslov potential are depicted in Figure~\ref{fig:Hopf}.
The algebra $\mathcal A(\Lambda_{\Ho}^k)$ is generated by four Reeb chords $a_1$, $a_2$, $b_1$ and $b_2$ with $|a_i|=1$ and $|b_1|=-|b_2|=k$. Using results of \cite{DR}, the differential as described above can be computed in a combinatorial way (see for example in \cite{Che, NE}), and the non-trivial part of the differential is given by $\partial a_1=b_1b_2$ and $\partial a_2=b_2b_1$.
\end{example}

\subsection{Augmentations} \label{sec:aug}

In this section, we review how augmentations, first used in \cite{Che}, can be used to construct a variety of ``linearizations'' of Legendrian contact homology.   

First observe that a commutative ring $\F$ can be considered as a DGA, where all elements of $\F$ have degree $0$ and the differential is identically $0$.  Then 
an \textbf{augmentation} of $\mathcal{A}(\leg)$ to $\F$ is a DGA-morphism, which is a graded algebra homomorphism that preserves the differential.  In particular, 
  $\epsilon: (\mathcal{A}(\Lambda),\partial)\rightarrow (\F,0)$ is a chain map such that $\epsilon(1)=1$, and 
for any element $a$ of nonzero degree, $\epsilon(a)=0$.

  \begin{definition}\label{defn:aug} $Aug(\leg; \mathbb F)$ will denote the set of augmentations of $\mathcal A(\leg)$ to $\mathbb F$.   As shown in \cite{EHK}, an embedded, Maslov-$0$, exact Lagrangian cobordism $\lag$ from $\Lambda_-$ to $\Lambda_+$ induces a DGA map $\Phi_\lag:\alg(\Lambda_+)\to\alg(\Lambda_-)$ and thus a map:
$$
  \begin{aligned} 
\mathcal{F}_{\lag}: Aug(\Lambda_-; \F) &\to Aug(\Lambda_+;\F), \\
\e_{-} &\mapsto  \e_{-} \circ \Phi_{\lag}.
\end{aligned}
$$
  \end{definition}

As above, let $R(\leg)$ denote the set of Reeb chords of $\leg$, and then let 
 $C(\leg)$ denote the graded $\F$-module generated by elements in $R(\leg)$, where the grading is as in Equation~(\ref{eqn:CZ-grade}).
Given an augmentation $\e$ of $\alg(\Lambda)$, the \textbf{linearized Legendrian contact homology} of $\Lambda$, denoted $LCH^{\epsilon}_{*}(\Lambda)$,
 is the homology of the chain complex $(C(\Lambda), \partial^{\epsilon})$  with
$$ \partial^\e(a)=\sum_{\dim\left(\mathcal{M}_J^{\R\times\Lambda}(a;{\bf p}b{\bf q})\right)=0}|\mathcal{M}_J^{\R\times\Lambda}(a;{\bf p}b{\bf q})|\e({\bf p})\e({\bf q})b,$$
where $a, b \in R(\leg)$, and ${\bf p},{\bf q}$ are words of Reeb chords.

\begin{example}[Augmentations of Hopf Links] \label{ex:aug-hopf}
Continuing with Example~\ref{ex:DGA-hopf}, one computes that the Hopf link $\Lambda_{\Ho}^0$ admits three augmentations to $\Z_2$ defined by sending the pair of chords $(b_1,b_2)$ to $(0,0)$, $(1,0)$ and $(0,1)$, while the Hopf links $\Lambda_{\Ho}^k$ for $k\neq0$ admit only the augmentation sending all chords to $0$.
One can now complete the explanation of the claim in Example \ref{ex:Hopf}, namely that $\Lambda_{\Ho}^0$ is the only Hopf link admitting an embedded, Maslov-0, exact Lagrangian filling. 
		If a Hopf link $\Lambda_{\Ho}$ bounds a connected, embedded, Maslov-$0$, exact Lagrangian filling $L$, then by Seidel's isomorphism \cite{EkhSFT,DR} the Poincaré polynomial of the Legendrian contact homology linearized by the augmentation induced by $L$ must be of the form $t+2g(L)+1$, where $g(L)$ is the genus of $L$. The LCH polynomial of $\leg_{\Ho}^k$ is $2t+t^k+t^{-k}$ when $k\neq 0$ and is $2t+2$ or $t+1$ when $k=0$ (depending on the choice of augmentation). Thus, when $k\neq 0$, Seidel's isomorphism obstructs the existence of a connected, embedded, exact, Maslov-$0$ Lagrangian filling of $\leg_{\Ho}^k$.
\end{example}

In fact, one can use two augmentations to linearize: given augmentations $\e^1,\e^2$ of $\Lambda$, the \textbf{bilinearized Legendrian contact homology} $LCH^{\e^1,\e^2}_*(\Lambda)$, defined first in \cite{BCh}, is the homology of $(C(\Lambda), \partial^{\e^1,\e^2})$, where
$$\partial^{\e^1,\e^2}(a)=\sum_{\dim\left(\mathcal{M}_J^{\R\times\Lambda}(a;{\bf p}b{\bf q})\right)=0}|\mathcal{M}_J^{\R\times\Lambda}(a;{\bf p}b{\bf q})|\e^1({\bf p})\e^2({\bf q})b.$$

In Section~\ref{sec:augcat} and \ref{sec:wrap}, we will be using moduli spaces that are defined using a partition of a Legendrian link into components. 
In the case $\Lambda=\Lambda^1\cup\Lambda^2$, where $\leg^1, \leg^2$ are Legendrian links, denote $ {R}(\Lambda^i,\Lambda^j)$ the set of Reeb chords from $\Lambda^j$ to $\Lambda^i$. If $c \in R(\Lambda^i,\Lambda^j)$ with $i\neq j$, we call $c$ a \textbf{mixed Reeb chord}, otherwise we call $c$ a \textbf{pure Reeb chord}.
Denote $C(\Lambda^1,\Lambda^2)$ the graded $\F$-module generated by elements in $R(\Lambda^1,\Lambda^2)$. Augmentations $\e^1$ of $\Lambda^1$ and $\e^2$ of $\Lambda^2$ induce an augmentation $\e=(\e^1,\e^2)$ of $\Lambda^1\cup\Lambda^2$ that agrees with $\e^i$ on pure chords of $\Lambda^i$, for $i=1,2$, and vanishes on mixed Reeb chords. Then, the differential of the Legendrian contact homology of $\Lambda$ linearized by $\e$, and restricted to mixed Reeb chords in $R(\Lambda^1,\Lambda^2)$ is defined via a count of $J$-holomorphic disks in {\bf mixed LCH moduli spaces}:
$$\partial^{\e}(a^{12})=\sum_{\dim\left(\mathcal{M}_J^{\R\times(\Lambda^1\cup\Lambda^2)}(a^{12};{\bf p}^{11}b^{12}{\bf q}^{22})\right)=0}|\mathcal{M}_J^{\R\times(\Lambda^1\cup\Lambda^2)}(a^{12};{\bf p}^{11}b^{12}{\bf q}^{22})|\e^1({\bf p}^{11})\e^2({\bf q}^{22})b^{12},$$
where $a^{12}, b^{12} \in R(\Lambda^1,\Lambda^2)$, $\mathbf{p}^{11}$ is a word of Reeb chords of $\Lambda^1$, and $\mathbf{q}^{22}$ is a word of Reeb chords of $\Lambda^2$.
Note that $\big(C(\Lambda^1,\Lambda^2),\partial^{\e}_{|C(\Lambda^1,\Lambda^2)}\big)$ is a subcomplex of $(C(\Lambda),\partial^{\e})$.

\section{The augmentation category}\label{sec:augcat}

In this section, we give a brief summary of the augmentation category mainly following \cite{NRSSZ}. We then define a  new notion of \emph{split-DGA homotopy} for augmentations of multicomponent links. This gives rise to a simple criterion, Corollary~\ref{cor:distinct-not-equiv}, to determine when two augmentations are not equivalent in $\aug_{+}$, which will be used frequently in Section~\ref{sec:ex}  when applying
Theorem~\ref{thm:main}.

\subsection{Definitions} \label{ssec:aug-defns}
Let $\Lambda$ be a Legendrian knot or link in $\R^3_{std}$. Assume that the Lagrangian projection $\pi_{xy}(\leg)$ has Maslov class $0$ 
and that  each connected component of $\Lambda$ is decorated with a base point.

The augmentation category $\aug_+(\Lambda)$  is an $A_{\infty}$-category whose {\it objects} are elements of $Aug(\leg; \F)$, namely augmentations of the Chekanov-Eliashberg DGA $\alg(\Lambda)$ to $\F$.
 In order to define the morphisms in $\aug_+(\Lambda)$, we use the DGA of a $2$-copy of $\Lambda$, denoted by $2\Lambda=\Lambda^1 \cup \Lambda^2$.
The copy $\Lambda^1$ is a perturbed push-off of $\Lambda^2$ in the $z$-direction, perturbed via a positive Morse function $f: \Lambda \to \R^+$ having one maximum and one minimum on each component of $\Lambda$, located near its base point as in Figure \ref{base}. Both $\Lambda^1$ and $\Lambda^2$ have the same Maslov potential. The Lagrangian projection of $2\Lambda$ for the max $tb$ right-handed (positive) trefoil is shown in Figure \ref{trefoil}.

For any two objects $\e^1, \e^2 \in Aug(\leg; \F)$ of $\aug_{+}(\leg)$, the {\it morphism space from $\e^1$ to $\e^2$}, denoted $Hom_+(\e^1,\e^2)$, is the graded $\F$-module generated by the Reeb chords in $R(\Lambda^1,\Lambda^2)$ 
of $2\Lambda$,   
with the grading of generators shifted up by $1$,  
commonly denoted as
$$Hom_+(\e^1,\e^2):=C(\Lambda^1,\Lambda^2)[1].$$ {We use $| \cdot |$ to denote the gradings in the $\F$-module $C(\leg)$, as given in Equation~(\ref{eqn:CZ-grade}),
and $| \cdot |_{+}$ to denote the shifted gradings in 
$Hom_{+}(\e^1,\e^2)$.  
}

\begin{figure}[!ht]
\begin{minipage}{2.9in}
\captionsetup{width=2.8in}
\labellist
\pinlabel $\ast$ at 50 8
\pinlabel $x$ at 80 17
\pinlabel $y$ at 105 25
\endlabellist
\begin{center}
\includegraphics[width=2in]{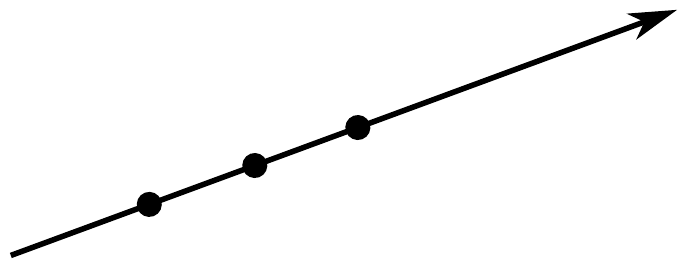}
\end{center}
\vspace{0.1in}
\caption{The local model near the base point $\ast$ of each component of $\Lambda$, with the arrow representing the orientation of $\Lambda$,   and $x$ ($y$) denoting the  maximum (minimum) of the function $f$.
\label{base}
}

\end{minipage}
\begin{minipage}{3in}
\begin{center}
\labellist
{\color{blue}
\pinlabel $\Lambda^2$ at 50 140
}
{\color{red}
\pinlabel $\Lambda^1$ at 50 0
}
\endlabellist
\includegraphics[width=2in]{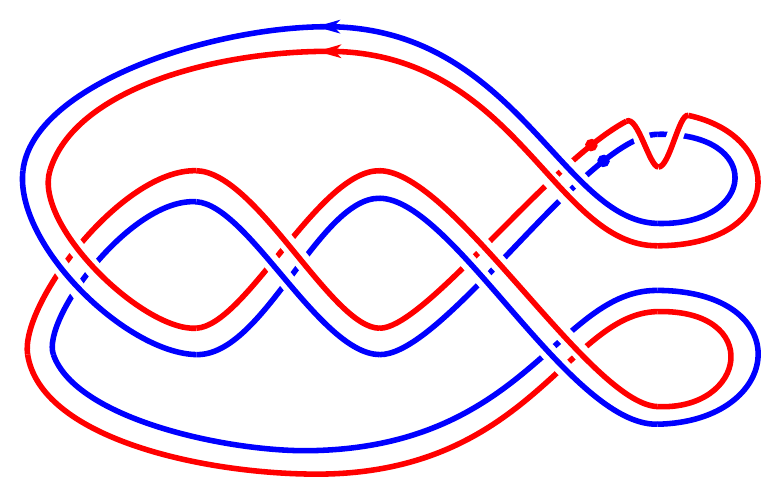}
\captionsetup{width=2.6in}
\caption{The Lagrangian projection of the $2$-copy $2\Lambda$, where $\Lambda$ is a max $tb$ positive trefoil.}
\label{trefoil}
\end{center}
\end{minipage}
\end{figure}

Taking a closer look at the generator set of $Hom_+(\e^1,\e^2)$, we
note that for each Reeb chord $a$ of $\Lambda$, there is a corresponding mixed Reeb chord 
$a^{12} \in R(\Lambda^1,\Lambda^2)$ 
of $2\Lambda$ with grading given by 
 $$|{a^{12}}|_{+}=|a|+1. $$
The other generators of $Hom_+(\e^1,\e^2)$ are the Morse Reeb chords, corresponding to the critical points of the Morse function $f$. Assume $\leg$ has $m$ components, and denote the Morse Reeb chords corresponding to the maxima of $f$ by $x_{i}^{12}$   and the ones corresponding to the minima of $f$ by 
 ${y}^{12}_i$, for $i=1,\dots, m$.  
By Equations \eqref{eqn:CZ-grade} and \eqref{eqn:CZ}, we find
\begin{equation} \label{eqn:max-min-generators}
|{x}^{12}_i|_{+}=1,~\text{and}~ \qquad  |{y}^{12}_i|_{+} =0.
\end{equation}

As a graded module, the morphism space $Hom_+(\e^1,\e^2)$  does not depend on $\e^1$ or $\e^2$, but the $A_{\infty}$ operators, called  {\it compositions},
$$m_n: Hom_+(\e^n,\e^{n+1})\otimes Hom_+(\e^{n-1}, \e^n) \otimes \cdots \otimes Hom_+(\e^1, \e^2)\to Hom_+(\e^1, \e^{n+1})$$
do depend on the choice of augmentations $\e^1,\dots,\e^{n+1}$. 
These  $A_{\infty}$ operators $m_n$ can be defined using the DGA of an $(n+1)$-copy of $\Lambda$.
{The $(n+1)$-copy is perturbed in such a way that 
every Reeb chord generator of $Hom_{+}(\e^{1}, \e^{2})$ has corresponding versions on consecutive pairs of the $(n+1)$-copy; 
see  \cite[Figure 6]{NRSSZ}.}  
We recall below the definitions of the operators $m_1$ and $m_2$ that will be used in this paper;  see \cite[Section $4$]{NRSSZ} for  details of this construction.

\begin{itemize}
\item
The operator $ m_1: Hom_+(\e^1,\e^2)\to Hom_+(\e^1,\e^2)$ is defined by a count of rigid holomorphic disks with boundary on $\R\times2\Lambda=\R\times(\Lambda^{1}\cup\Lambda^{2})$ with one positive asymptotic and one negative asymptotic to Reeb chords in  $R(\Lambda^1,\Lambda^2)$ and possibly some other negative asymptotics to {\it pure} Reeb chords.
Indeed,
$$\displaystyle{
m_1( b^{12})= \sum_{\dim \left(\M(a^{12}; {\bf p}^{11} b^{12} {\bf q}^{22})\right)=0} |\M^{\R\times(\Lambda^{1}\cup\Lambda^{2})}(a^{12}; {\bf p}^{11} b^{12} {\bf q}^{22})|\e^1({\bf p}^{11}) \e^2({\bf q}^{22})  a^{12}},$$ 
where $b^{12}, a^{12} \in R(\leg^{1}, \leg^{2})$, and ${\bf p}^{11}$, ${\bf q}^{22}$ are words of pure Reeb chords in $R(\Lambda^1)$ and $R(\Lambda^2)$, respectively.
The operator $m_1$ is a degree $1$ map that satisfies $m_1^2=0$, and we denote $H^*Hom_+(\e^1,\e^2)$ the cohomology of the complex $(Hom_+(\e^1,\e^2), m_1)$.
In addition, one has the following isomorphism from \cite[Corollary 5.6]{NRSSZ}:
$$
H^*Hom_+(\e^1,\e^2) \cong LCH^{\e^1,\e^2}_{1-*} (\Lambda).
$$
 
\item To define the operator $ m_2: Hom_+(\e^2,\e^3) \otimes Hom_+(\e^1,\e^2)\to Hom_+(\e^1,\e^3)$, we first consider the $3$-copy $3\Lambda=\Lambda^1\cup \Lambda^2\cup \Lambda^3$, where
$3\Lambda$ is constructed  such  that for any $i < j$, the DGA of $\Lambda^i\cup \Lambda^j$ is canonically identified with the DGA of $2\Lambda$; see \cite[Figure 6]{NRSSZ}.
The operator $m_2$ counts rigid holomorphic disks with boundary on $\R\times3\Lambda$, with a positive asymptotic to a Reeb chord in $R(\Lambda^1,\Lambda^3)$, two negative 
asymptotics to Reeb chords in $R(\Lambda^1,\Lambda^2)$ and $R(\Lambda^2,\Lambda^3)$ respectively, and possibly additional negative asymptotics to pure Reeb chords.
More precisely, 
 $$\displaystyle{m_2(c^{23}, b^{12})=\sum_{\dim( \M(a^{13}; {\bf p}^{11} b^{12} {\bf q}^{22} c^{23} {\bf r}^{33}))=0} |\M^{\R\times3\Lambda}(a^{13}; {\bf p}^{11} b^{12} {\bf q}^{22} c^{23} {\bf r}^{33})| \e^1({\bf p}^{11}) \e^2 ({\bf q}^{22})  \e^3( {\bf r}^{33})\  {a}^{13}},
$$ 
where $a^{13}\in R(\Lambda^1,\Lambda^3), b^{12}\in R(\Lambda^1,\Lambda^2), c^{23}\in R(\Lambda^2,\Lambda^3)$, and ${\bf p}^{11},{\bf q}^{22},{\bf r}^{33}$ are words of pure chords.   The operator $m_2$ is of degree $0$ and induces a product structure on the cohomology $H^*Hom_+$:
$$ m_2: H^i Hom_+(\e^2,\e^3) \otimes H^j Hom_+(\e^1,\e^2)\to H^{i+j}Hom_+(\e^1,\e^3).$$
\end{itemize}

\subsection{Unital $A_{\infty}$ category}\label{sec:unital}
  A key property of $\aug_+(\Lambda)$ is that it is a {\bf strictly unital $A_{\infty}$ category}:
for any $\e$, there is an element   { $e_{\e}\in Hom_+(\e,\e)$ with $| e_{\e}|_{+} = 0$} such that
 \begin{itemize}
\item $m_1(e_{\e})=0$;
\item for all $a\in Hom_+(\e,\e')$ and $b\in Hom_+(\e',\e)$,
$$m_2(e_{\e},b)=b, \quad m_2(a,e_{\e})=a; \qquad \text{and}$$
\item any higher order composition, {$m_n$ for $n \geq 3$} vanishes when  $e_{\e}$ is one of the inputs.
\end{itemize}
In fact, if $\leg$ has $m$ components, the unit is given by  
$$e_\e=-\sum_{i=1}^m {y}^{12}_i\in Hom_+(\e,\e),$$
for ${y}^{12}_i$ as defined in Equation~\eqref{eqn:max-min-generators}.
It follows that the induced cohomology category $H^*\aug_+(\Lambda)$ is a unital category.
This allows us to define a notion of equivalence of two objects.
 
\begin{definition}\label{defn:equal}
Two augmentations $\e^1$ and $\e^2$ of $\alg(\Lambda)$ are {\bf equivalent} in the augmentation category $\aug_+(\Lambda)$, denoted by $\e^1\Sim\e^2$, if they are isomorphic in $H^*\aug_+(\Lambda)$, that is, if there exist $[\alpha]\in H^0Hom_+(\e^1, \e^2)$ and $[\beta]\in H^0Hom_+(\e^2, \e^1)$ such that 
$$m_2([\alpha],[\beta])= [e_{\e^2}]\in H^0Hom_+(\e^2, \e^2), \mbox{ and } m_2([\beta],[\alpha])= [e_{\e^1}]\in H^0Hom_+(\e^1, \e^1),$$
where $[e_{\e^i}]$ is the unit in $H^0Hom_+(\e^i, \e^i)$ for $i=1,2$.
\end{definition}

It can be difficult to show that two augmentations of $\leg$ are not equivalent using Definition~\ref{defn:equal}.  However, by relating this definition of equivalence to the notion of DGA-homotopic augmentations, there is an
easier criterion for distinguishing non-equivalent augmentations; see Corollary~\ref{cor:distinct-not-equiv}.

An augmentation is a DGA morphism, and there is an established notion of a homotopy between  DGA morphisms; see, for example, \cite[Section 2.3]{kalman:one-parameter} and \cite[Definition 5.15]{NRSSZ}.
It is proved in \cite[Proposition 5.19]{NRSSZ} that if $\Lambda$ is a Legendrian {\it knot}, then two augmentations are DGA-homotopic if and only if they are equivalent in $\aug_+(\Lambda)$. In order to obtain a similar result in the case where $\Lambda$ is a Legendrian \textit{link}, we use the fact that the DGA of a Legendrian link has a  ``homotopy splitting,''  which was first defined by Mishachev \cite{Mishachev03}.  Before we explain this splitting of the DGA for a Legendrian link,  we give the general definition of a split DGA and morphisms of split DGAs.  
\begin{definition} A (unitary) {\bf  split DGA} $(\mathcal A_{**}, \partial_{**})$ over $\F$ is  an algebra $\mathcal A_{**}$ over $\F$ such that  $\mathcal A_{**} = \oplus_{j_{1}, j_{2}} \mathcal A_{j_{1}j_{2}}$, where
\begin{enumerate}
\item  each $\mathcal A_{j_{1}j_{2}}$ is a module over $\F$, 
\item there are bilinear multiplication maps
$\mathcal A_{j_{1}j_{2}} \times \mathcal A_{j_{3}j_{4}} \to \mathcal A_{j_{1}j_{4}}$
that are $0$ unless $j_{2}=j_{3}$, 
\item for all $j$, $\mathcal A_{jj}$ contains an element $e_{j}$ that acts as the identity under multiplication, and 
\item $\partial_{**}$ respects the splitting, namely $\partial_{**}: \mathcal A_{j_{1}j_{2}} \to \mathcal A_{j_{1}j_{2}}$, for all $j_{1}, j_{2}$.
\end{enumerate}
Given two split DGAs, 
$\left( \oplus_{i,j=1}^n \mathcal{A}_{ij}, \partial \right) $ and 
$\left( \oplus_{i,j=1}^m \mathcal{A}'_{ij}, \partial' \right) $, a {\bf split-DGA morphism} $f: (\oplus_{i,j=1}^n \mathcal{A}_{ij}, \partial)\rightarrow (\oplus_{i,j=1}^m \mathcal{A}'_{ij}, \partial')$ is a DGA morphism such that for all $i, j$, there exist $i', j'$ such that 
 $f(\mathcal{A}_{ij}) \subset \mathcal{A}'_{i'j'}$. Observe that $(\F, 0)$ can be viewed as a split DGA with no splitting.
 \end{definition}
 
 The following  is a new definition, which extends the definition of DGA homotopy given, for example, in \cite[Definition 5.15]{NRSSZ}.
 \begin{definition}\label{defn:split-DGA-homotopy} 
 Given a unital, commutative ring $\mathbb F$, let $\mathbb F^*$ denote the set of units.
Two split-DGA morphisms $f_1, f_2: (\oplus_{i,j=1}^n \mathcal{A}_{ij}, \partial)\rightarrow (\oplus_{i,j=1}^m \mathcal{A}'_{ij}, \partial')$ are \textbf{split-DGA homotopic} 
if there exists 
$K:\oplus_{i,j=1}^n \mathcal{A}_{ij},\rightarrow \oplus_{i,j=1}^m \mathcal{A}'_{ij}$ such that:
\begin{enumerate}
\item $K$ is split, $\F$-linear, and degree $1$, 
\item for all $i,j$ there exists $\alpha_i,\alpha_j\in \F^*$
 such that for all $a\in \mathcal A_{ij}$, $$\alpha_i f_1(a)-\alpha_j f_2(a)=\partial' K(a)+K\partial(a), \quad \text{and}$$ 
\item $K(x\cdot y)=K(x)\cdot f_2(y)+(-1)^{|x|}f_1(x)\cdot K(y)$, for all $x,y \in \oplus_{i,j=1}^n \mathcal{A}_{ij}$. 

\end{enumerate}
\end{definition}
 
 \begin{remark}
 \begin{enumerate}
 \item If $\alpha_i=1$ for all $i$ or $\F = \Z_2$, then Definition~\ref{defn:split-DGA-homotopy} agrees the usual definition of DGA homotopy. 
 \item If   $\leg$ has a single component, and $\e^{1}, \e^{2}$ are two augmentations of $\mathcal{A}(\leg)$,  then the existence of a DGA homotopy between  $\e^{1}, \e^{2}$
 is equivalent to the existence of a split-DGA homotopy  between $\e^{1}, \e^{2}$.  It is immediate to see that a DGA homotopy implies the existence of a split DGA homotopy.
 In the other direction, a split DGA homotopy implies the existence of $K$ and $\alpha \in \F^*$ satisfying
 $$\alpha f_1(a)-\alpha f_2(a)=\partial' K+K\partial, \quad \text{for all}~ a\in \mathcal A.$$
 Then $K' = \alpha^{-1} K$ is the desired DGA homotopy.
 \end{enumerate} \label{rem:equiv-DGA-homotopy}
\end{remark}

 Given a Legendrian link $\leg = (\leg^{1}, \dots, \leg^{m})$, we can split what is essentially a submodule of the Chekanov-Eliashberg DGA into $m^2$ pieces that are invariant under Legendrian isotopy as has been shown in, for example, \cite[Definition 2.18]{Ng:computable} and \cite[Section 2.4]{NgTraynor04}).  Let $\mathcal{A}_{ij}$ be the module generated by words of Reeb chords that begin on  $\leg^{i}$ and end on  $\leg^{j}$, i.e. Reeb chords in $R(\Lambda^j, \Lambda^i)$. If $i=j$ we also add in an indeterminate $e_j$. The differential $\partial_{**}$ is defined on the generators $a$  as follows:
if the Reeb chord $a$ begins and ends on distinct components of $\leg$, then $\partial_{**}(a) = \partial(a)$; if $a$ is a Reeb chord that begins and ends on the same component $\leg^{j}$ of $\leg$,
 then replace any occurrence of $1$ in $\partial(a)$ by $e_{j}$, that is, every holomorphic disk with boundary on $\Lambda_j$ with positive asymptotic to $a$ and no negative asymptotics contributes $e_j$ to $\partial_{**}(a)$.  Then $\partial_{**}$ extends to $\mathcal A_{**}$ by applying the Leibniz rule and setting $\partial_{**}(e_{j}) = 0$, for all $j$. Augmentations $\epsilon: (\mathcal A, \partial) \to (\F, 0)$ are in bijective correspondence with {\bf split augmentations} $\epsilon_{**}: (\mathcal A_{**}, \partial) \to (\F, 0)$: on any Reeb chord generator $a$, $\epsilon(a) = \epsilon_{**}(a)$ and 
 $\epsilon(1) = 1 = \epsilon_{**}(e_{j})$, for all $j$.

Using Definition~\ref{defn:split-DGA-homotopy}, a slight modification of the proof of \cite[Proposition 5.19]{NRSSZ} gives the following proposition, whose proof is given in Appendix~\ref{appendix:equivalences}.

\begin{proposition} \label{prop:equivalences} Given a Legendrian link $\Lambda \subset \R^3_{std}$, two augmentations $\e^1,\e^2:\alg(\Lambda)\to\F$ are equivalent in $\aug_+(\Lambda)$ if and only if the corresponding 
split augmentations $\e_{**}^{1}$ and $\e_{**}^{2}$ are split-DGA homotopic.
\end{proposition}

Proposition~\ref{prop:equivalences} gives us a simple way to determine if two augmentations are not equivalent in $\aug_{+}(\leg)$. 

\begin{corollary} \label{cor:distinct-not-equiv} Suppose that the Legendrian link $\Lambda = (\leg^{1}, \dots, \leg^{n})$ does not have any degree $-1$ Reeb chords.  Then any two augmentations  
 $\e^1$ and $\e^2$ of $\leg$ are equivalent in $\aug_+(\Lambda)$ if and only if for all $i, j \in \{1, \dots, n\}$  there exist $\alpha_{i}, \alpha_{j} \in \F^*$ such that $\alpha_{i}\e^1(a) = \alpha_{j} \e^2(a)$, for all degree $0$ Reeb chords $a \in R(\Lambda^j, \Lambda^i)$. If $\F=\Z_2$, then  two augmentations are equivalent in $\aug_+(\Lambda)$ if and only if they are identically the same.
\end{corollary}

\begin{proof}  Recall that the support of an augmentation is contained in the degree $0$ portion of $\alg(\Lambda)$.
By Proposition~\ref{prop:equivalences}, it suffices to show that when a Legendrian $\leg$ does not have any degree $-1$ Reeb chords,  
$\e^{1}, \e^{2}:  \left( \mathcal A(\leg), \partial\right) \to (\mathbb F, 0)$ are split-DGA homotopic if and only if for all $i, j$ there exist $\alpha_{i}, \alpha_{j} \in \F^*$ such that 
$\alpha_{i}\e^1(a) - \alpha_{j} \e^2(a) = 0$, for all degree $0$ Reeb chords $a \in R(\Lambda^j, \Lambda^i)$.
 Suppose $\e^{1}, \e^{2}$ are split-DGA homotopic via
 $K: \left( \mathcal A(\leg), \partial\right) \to (\mathbb F, 0)$. 
Since $K$ is degree $1$ and $\mathbb F$ is in degree $0$, $K$ is supported in the degree $-1$ portion of $\mathcal A(\leg)$, which since
 there are no $-1$ degree Reeb chords is spanned by monomials of words length at least $2$.  Then an induction argument using the condition (3) of Definition~\ref{defn:split-DGA-homotopy} tells us that $K = 0$. 
 It follows that for an arbitrary degree $0$ Reeb chord $a \in R(\Lambda^j, \Lambda^i)$, $\alpha_{i}\e^1(a) -\alpha_{j} \e^2(a)=0$. 
For the other direction, if for all $i, j \in \{1, \dots, n\}$  there exist $\alpha_{i}, \alpha_{j} \in \F^*$ such that $\alpha_{i}\e^1(a) - \alpha_{j} \e^2(a)=0$, for all degree $0$ Reeb chords $a \in R(\Lambda^j, \Lambda^i)$, by setting $K=0$, we get the desired split-DGA homotopy.
 \end{proof}

\begin{remark}\label{rem:aug-equiv}    
Recall the map
$	\mathcal{F}_\lag:Aug(\Lambda_-;\F) \to Aug(\Lambda_+; \F) $
in Definition \ref{defn:aug}, induced by an embedded, Maslov-$0$, exact Lagrangian cobordism $\lag$ from $\Lambda_-$ to $\Lambda_+$.
\begin{enumerate}
\item It is known   
that this map descends to 
$$	\mathcal{F}_\lag:Aug(\Lambda_-;\F)/\sim_{{DGA\  hom}} \to Aug(\Lambda_+; \F)/\sim_{DGA\ hom},$$
where $\sim_{DGA\ hom}$ denotes the equivalence relation defined by DGA homotopy: if $K_-$ is a DGA-homotopy between two augmentations $\e^1$ and $\e^2$ of $\Lambda_-$, then $K_-\circ\Phi_{L}$ is a DGA-homotopy between $\mathcal F_{\lag}(\e^1)$ and $\mathcal F_{\lag}(\e^2)$. Thus, as observed in Remark~\ref{rem:equiv-DGA-homotopy}
if $\leg_{\pm}$ are Legendrian knots or if $\leg_{\pm}$ are Legendrian links and  $\F = \Z_{2}$, $\Sim$ is the same as $\sim_{DGA\ hom}$, and the map
 $$
 \mathcal{F}_\lag:Aug(\Lambda_-;\F)/\Sim\to Aug(\Lambda_+; \F)/\Sim
$$
exists.  
\item In general the map $\mathcal F_{\lag}$ does not descend to augmentations defined up to equivalence by split-DGA homotopy. For example, there exists an embedded, Maslov-$0$, exact Lagrangian cobordism $L$ from a Hopf link to the trefoil such that the Hopf link has two augmentations over $\Z$ that are split-DGA homotopic, while their images under $\mathcal F_{\lag}$ are not (split-)DGA homotopic augmentations of the trefoil.
\end{enumerate}
\end{remark}

\section{Wrapped Floer Theory}\label{sec:wrap}

In this section we review the setup and some properties of Floer theory for Lagrangian cobordisms as developed in \cite{CDRGG} for our setting of interest. Namely, we consider the Cthulhu complex $Cth(\lag^1,\lag^2)$ over a unital, commutative ring $\F$ associated to a pair of transverse, embedded, Maslov-$0$, exact Lagrangian cobordisms $\lag^1,\lag^2$.  If $\F$ is not characteristic $2$, we further assume the cobordisms $\lag^1$ and $L^2$ are spin. Without loss of generality, we assume  that the constant value of the primitive of any cobordism we consider vanishes on the negative end, i.e. $\mathfrak{c}_-=0$; see Remark~\ref{rem:cobord}. 
We review the result established in \cite{Pan1} that we can construct an isomorphism $\phi_{*}$ between the cohomology of a quotient complex of the
 Cthulhu complex, $H^* (C_{-\infty}, d_{-\infty})$,
and $H^*Hom_+(\e_+^1,\e_+^2)$, see Equation~\eqref{eqn:C-Hom}. In fact, the cohomology groups on both sides of this isomorphism possess a product structure, $m_{2}^{-\infty}, m_{2}^{+}$,  and we review the fact, from \cite{Noe}, that $\phi_{*}$ preserves the product structure, see Proposition~\ref{prop:ring}.  Understanding the definition of $m_{2}^{-\infty}$ will be important in Section~\ref{sec:main} where we will establish in Proposition~\ref{prop:unit}, the key result needed to prove Theorem~\ref{thm:main}.

\subsection{A special pair}\label{sec:pair}

Let $\lag$ be an embedded, Maslov-$0$, exact Lagrangian cobordism in the symplectization of $\R^3_{std}$ from $\Lambda_-$ to $\Lambda_+$.
Consider a perturbed $2$-copy of $\lag$, $2\lag=\lag^1\cup \lag^2$, where $\lag^1$ is a push-off of $\lag^2:=\lag$ in the positive $z$-direction via a Morse function $F:\lag \to \R^{+}$ such that $\lag^1\cup \lag^2$ on the two cylindrical ends agrees with a cylinder over the $2$-copies $\Lambda_{\pm}^1\cup \Lambda_{\pm}^2$ in the corresponding $\aug_+$ categories; for details see \cite{Pan1}.
In particular, the Morse function $F$ on $[N, \infty)\times \Lambda_+$ and $(-\infty, -N]\times \Lambda_-$ agrees with $e^t f_{\pm}$, where $f_{\pm}$ are Morse functions on $\Lambda_{\pm}$ that have the same critical points as the ones used in the construction of $2\Lambda_{\pm}$ in $\aug_+(\Lambda_{\pm})$; see Section \ref{ssec:aug-defns}.
Moreover, we assume that on $\left([-N,N]\times \R^3\right)\cap \lag$ the value of the Morse function $F$ on each point is less than the \textit{cobordism action} of any pure Reeb chord $\gamma$ of $\Lambda_-$, given by $e^{-N}\int_\gamma\alpha$. Such an assumption is necessary in order to get the identifications of complexes in Proposition \ref{prop:des} below.
\begin{remark} 
	We refer to \cite[Section 3.4.2]{CDRGG} 
	 for more details on 
	 the relation between the energy of pseudo-holomorphic disks with boundary on $\lag^1\cup\lag^2$ and the action of intersection points and Reeb chords.
	 In our special pair case, intersection points in $\lag^1\cap\lag^2$ are in one-to-one correspondence with critical points of the Morse function $F$, and
	the action of $p\in\lag^1\cap\lag^2$ is given by the value of $F$ at $p\in\lag$.
\end{remark}

The particular type of perturbation used on the cylindrical ends implies that the algebras $\alg(\Lambda^1_\pm)$ and $\alg(\Lambda^2_\pm)$ are canonically isomorphic: there are canonical identifications of Reeb chords and the differentials agree under this identification.   
An augmentation $\e_-$ of $\alg(\Lambda^2_-)$ gives under this identification an augmentation of $\alg(\Lambda^1_-)$. Moreover, if the cobordisms $\lag^1$ and $\lag^2$ are sufficiently $C^1$-close, then they induce the same augmentation of $\alg(\Lambda_+^1)$ and $\alg(\Lambda_+^2)$, i.e. $\e_-\circ\Phi_{\lag^1}=\e_-\circ\Phi_{\lag^2}$, under the canonical identification of generators, see \cite[Theorem 2.15]{CDRGG1}.

\subsection{The Cthulhu complex $Cth(\lag^1,\lag^2)$}\label{sec:cth}

Given the special pair of cobordisms $L^1,L^2$ as above, for $i=1,2$  suppose that $\e_-^i$ is an augmentation for $\alg(\Lambda_-^i)$ and $\e_+^i = \mathcal F_{\lag^{i}}(\e_{-}^{i})$ is the augmentation of $\alg(\Lambda_+^i)$ induced by $\e_-^i$ through $L^i$.
The \textbf{Cthulhu complex} $Cth(\lag^1, \lag^2)$ can be described as follows.
It is a graded $\F$-module generated by three types of generators:
$$Cth(\lag^1, \lag^2)=C_+(\lag^1, \lag^2)\oplus C_0 (\lag^1, \lag^2)\oplus C_-(\lag^1, \lag^2),$$ 
where
\begin{itemize}
	\item $C_+(\lag^1, \lag^2)= C(\Lambda_+^1,\Lambda_+^2)[2]$ is the $\F$-module generated by Reeb chords from $\Lambda_+^2$ to $\Lambda_+^1$ with a grading shift, i.e. a Reeb chord 
	$a\in C_+(\lag^1, \lag^2)$ has grading $| a|_{Cth} = |a|+2$, for $|a|$ as in Equation (\ref{eqn:CZ-grade}). 
	\item $C_-(\lag^1, \lag^2)= C(\Lambda_-^1,\Lambda_-^2)[1]$.
	\item $C_0 (\lag^1,\lag^2)$ is the $\F$-module generated by intersection points in $\lag^1\cap\lag^2$. The grading of intersection points is given by the Conley-Zehnder index of the corresponding Reeb chords in the Legendrian lift, which is the same as the grading in Lagrangian intersection homology.
\end{itemize}
We use the shortened notation $Cth(\lag^1, \lag^2)= C_+\oplus C_0\oplus C_-$.
The fact that the Morse function $F$ is positive implies, by energy restrictions, that
the differential $d$ on $Cth(\lag^1,\lag^2)$ is upper triangular \cite[Lemma 7.2]{CDRGG}:
$$d=\left(\begin{array}{ccc}
d_{++} & d_{+0}& d_{+-}\\
0& d_{00} & d_{0-}\\
0& 0& d_{--}\\
\end{array}\right),$$
where each component is defined by a count of rigid pseudo-holomorphic disks with boundary on $\lag^1\cup\lag^2$ that we now describe.
 Let $\mathcal{J}^{cyl}$,  
 $\mathcal{J}^{adm}$ be respectively the sets of cylindrical and admissible almost complex structures on $(\R\times\R^3,d(e^t\alpha))$, defined as in  \cite[Section 2.2]{CDRGG}. 
\begin{enumerate}
	\item The maps $d_{\pm\pm}$ are the bilinearized codifferentials with respect to $(\e_{\pm}^1, \e^2_{\pm})$, as reviewed in Section \ref{sec:aug}, and therefore count rigid holomorphic disks with boundary on $\R\times(\Lambda^1_{\pm}\cup \Lambda^2_{\pm})$ with one puncture positively asymptotic and one puncture negatively asymptotic  to
	mixed Reeb chords in $R(\Lambda^1_{\pm}, \Lambda^2_{\pm})$.   
	More explicitly,
	$$\displaystyle{d_{\pm\pm}(b^{12}_{\pm})=\sum_{\dim\left(\M_{J^\pm}^{\R\times(\Lambda^1_{\pm}\cup \Lambda^2_{\pm})} (a^{12}_{\pm}; {\bf p}_{\pm}^{11} b^{12}_{\pm}{\bf q}_{\pm}^{22})\right)=0} |\M_{J^\pm}^{\R\times(\Lambda^1_{\pm}\cup \Lambda^2_{\pm})} (a^{12}_{\pm}; {\bf p}_{\pm}^{11} b^{12}_{\pm}{\bf q}_{\pm}^{22})| \e^1_{\pm}({\bf p}_{\pm}^{11}) \e^2_{\pm}({\bf q}_{\pm}^{22})\  a^{12}_{\pm}},$$
	where $J^{\pm} \in \mathcal{J}^{cyl}$, 
$b^{12}_{\pm}$, $a^{12}_{\pm}$ are generators in $C_{\pm}$, and ${\bf p}_{\pm}^{11}$, and ${\bf q}_{\pm}^{22}$ are words of pure Reeb chords of $\Lambda^1_{\pm}$ and $\Lambda^2_{\pm}$, respectively.
	
	\item The maps $d_{ij}$ from $C_j$ to $C_i$, for $(i,j)=(+-), (+,0), (0,0)$ or $(0,-)$, are defined by a count of rigid holomorphic disks with boundary on $\lag^1\cup\lag^2$ 
	with a puncture positively asymptotic to a generator $c_+$ of $C_i$, a puncture negatively asymptotic to a generator $c_-$ of $C_j$, and possibly other punctures 
	negatively asymptotic to  pure Reeb chords of $\Lambda^1_-\cup \Lambda^2_-$, as shown in Figure \ref{fig:diff}. 
	We use $(\e_-^1,\e_-^2)$ to augment the Reeb chords at the negative pure punctures. The definition of such moduli spaces is similar to the definition of mixed LCH moduli spaces in Section \ref{sec:aug} except that the Lagrangian boundary condition is not cylindrical anymore.  This means that there is no $\R$-action, and so we need a path $\mathbf{J}_s$ of almost complex structures in $\mathcal{J}^{adm}$ to ensure transversality (see \cite[Section 3]{CDRGG} for more details):
	$$\displaystyle{d_{i,j}(c_-)=\sum_{\dim\left(\M_{\mathbf{J}_s}^{\lag^1\cup \lag^2} (c_+; {\bf p}_-^{11} c_- {\bf q}_{-}^{22})\right)=0} |\M_{\mathbf{J}_s}^{\lag^1\cup \lag^2} (c_+; {\bf p}_-^{11} c_- {\bf q}_{-}^{22})| \e^1_{-}({\bf p}_{-}^{11}) \e^2_{-}({\bf q}_{-}^{22})\ c_+},$$
\end{enumerate}

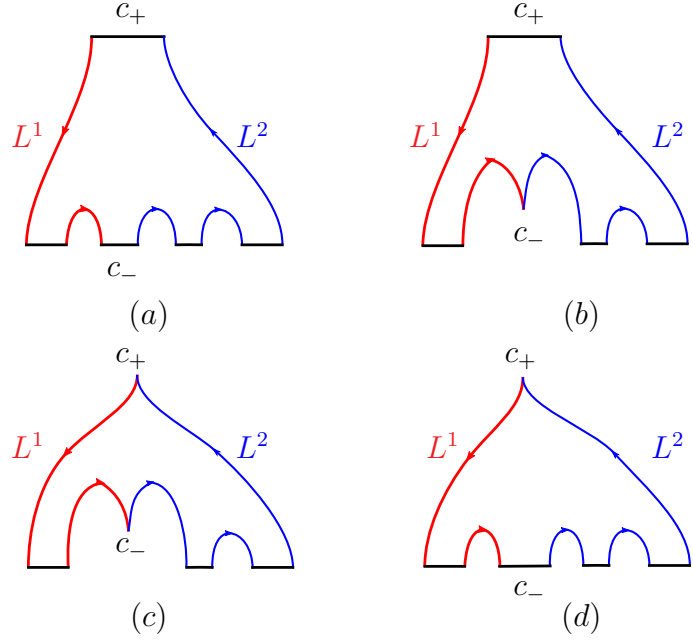
\begin{figure}[!ht]
	\labellist
	{\color{Red}
		\pinlabel $\lag^1$ at  30 280
		\pinlabel $\lag^1$ at  30 100
		\pinlabel $\lag^1$ at  260 280
		\pinlabel $\lag^1$ at  270 100
	}
	{\color{blue}
		\pinlabel $\lag^2$ at  160 280
		\pinlabel $\lag^2$ at  160 100
		\pinlabel $\lag^2$ at  400 280
		\pinlabel $\lag^2$ at  400 100
	}
	\pinlabel $c_+$ at  90 350
	\pinlabel $c_-$ at 85 200
	\pinlabel $(a)$ at  100 175
	
	\pinlabel $c_+$ at  90 150
	\pinlabel $c_-$ at 90 40
	\pinlabel $(c)$ at  100 0
	
	\pinlabel $c_+$ at  320 350
	\pinlabel $c_-$ at 320 220
	\pinlabel $(b)$ at  350 175
	
	\pinlabel $c_+$ at  315 150
	\pinlabel $c_-$ at 320 15
	\pinlabel $(d)$ at  350 0
	\endlabellist
	\includegraphics[width=4in]{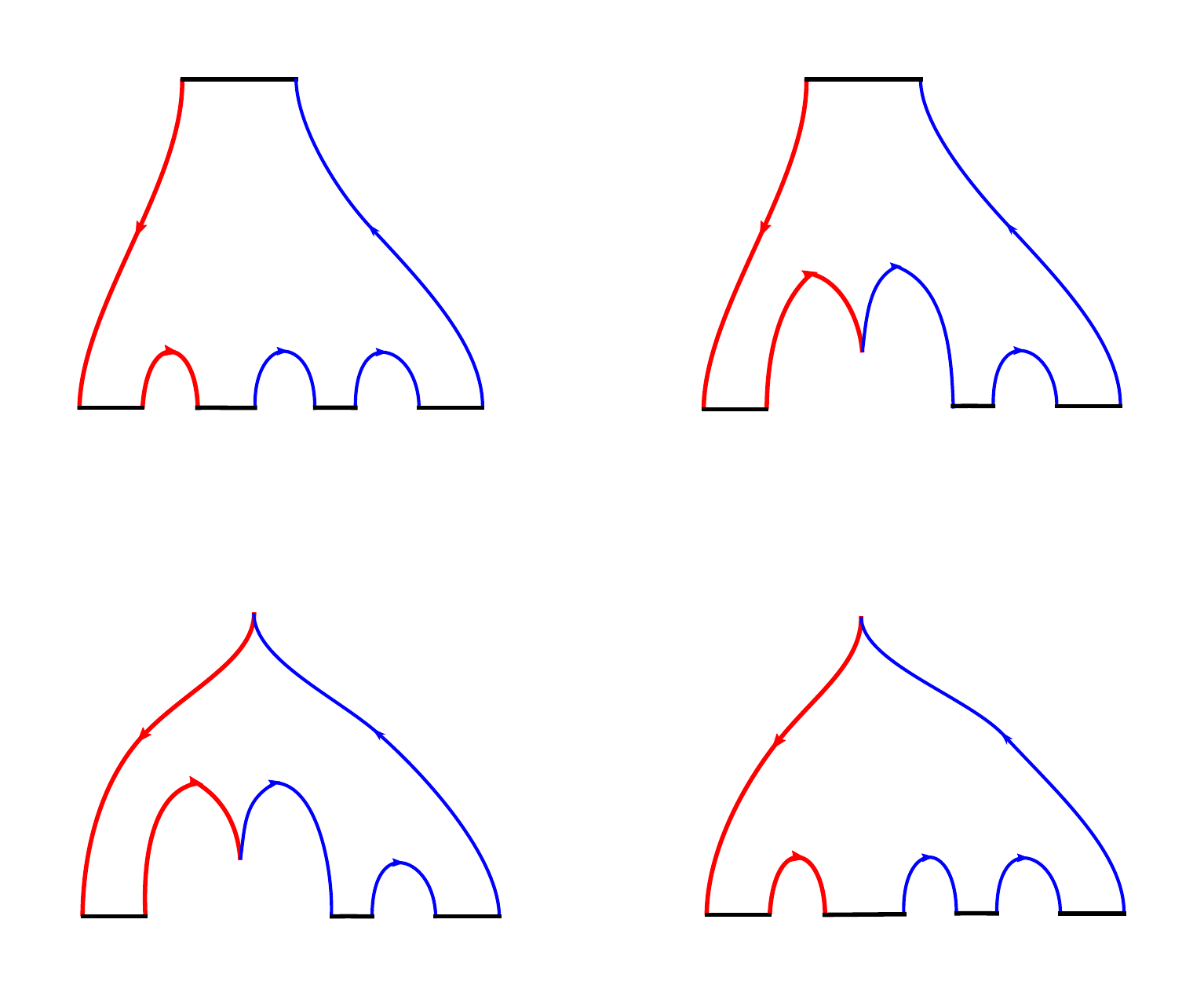}
	\caption{$(a)-(d)$ are the types of holomorphic disks counted by $d_{i,j}$ for $(i,j)=(+-), (+,0), (0, 0)$ and $(0,-)$, respectively, where $c_+$ is a generator of $C_i$ and $c_-$ a generator of $C_j$.}
	\label{fig:diff}
\end{figure}

We can identify some subcomplex and quotient complex of $(Cth(\lag^1, \lag^2), d)$ with cochain complexes defined in Section~\ref{sec:augcat}.
\begin{proposition}[{\cite[Theorem 5.1]{Pan1}}] \label{prop:des}
	The top and bottom cochain complexes admit the following identifications
	$$(C_+, d_{++})= (Hom_+(\e_+^1, \e_+^2)[1], m_1^+),\qquad (C_-, d_{--})= (Hom_+(\e_-^1,\e_-^2), m_1^-),$$
	and $(C_0, d_{00})$ is the cochain complex of the Morse cohomology for $F$ with differential counting Morse flow lines of $F$.
\end{proposition}

Observe that the Cthulhu complex $(Cth(\lag^1, \lag^2), d)$ is the cone of $\phi:=d_{+0}+d_{+-}$. 
The long exact sequence induced by the cone together with the fact that the complex is acyclic \cite[Theorem 6.6]{CDRGG} implies that $\phi$ induces an isomorphism 
\begin{equation} \label{eqn:lower-upper}
\phi_*: H^{*}(C_{-\infty}, d_{-\infty}) \to H^{*+1}(C_+, d_{++}),
\end{equation}
 where
$$C^{*}_{-\infty}= C^{*}_{0}\oplus C_-^{*}, \quad \mbox{ and } \quad  d_{-\infty}=\left(\begin{array}{cc} d_{00}& d_{0-}\\0& d_{--}\\\end{array}\right).$$
Note that $\phi_*$ may depend on the perturbation $F$.
 We have also that $C_{-\infty}$ is the cone of $d_{0-}$. The long exact sequence induced by a cone together with the isomorphism $\phi_*$ and the identifications in Proposition \ref{prop:des}, give the following long exact sequence:
\begin{equation}\label{eq:les}
\cdots \to H^k(\lag, \Lambda_-)\to H^kHom_+(\e^1_+, \e^2_+) \to H^kHom_+(\e^1_-, \e^2_-)\to H^{k+1}(\lag, \Lambda_-)\to \cdots.
\end{equation}

\subsection{Product structure}\label{sec:prod}

The isomorphism  $\phi_*$ from Equation~\eqref{eqn:lower-upper} together with the identification of $H^{*+1} (C_+, d_{++}) = H^*Hom_+(\e_+^1,\e_+^2)$ given by Proposition \ref{prop:des}
  gives an isomorphism
\begin{equation} \label{eqn:C-Hom}
\phi_*: H^* (C_{-\infty}, d_{-\infty}) \to H^*Hom_+(\e_+^1,\e_+^2).
\end{equation}
As recalled in Section~\ref{ssec:aug-defns}, there is a product structure on $H^*Hom_+$ given by the map $m_2^+$ in the category $\aug_+(\Lambda_+)$.
There is also a product structure $m_2^{-\infty}$ on $H^*(C_{-\infty}, d_{-\infty})$ defined by the second author of this paper \cite{Noe}.   In fact, the isomorphism $\phi_*$ preserves the product structures; see Proposition~\ref{prop:ring}.

Let us give a more detailed overview of the construction of $m_2^{-\infty}$ in the case of a $3$-copy $3\lag=\lag^1\cup\lag^2\cup \lag^3$ such that any pair $(\lag^i,\lag^j)$ for $i<j$ has the same Cthulhu complex as $2\lag$; that is, for $i<j$, the cobordism $\lag^i$
is a push-off of $\lag^j$ using a positive Morse function $F^{ij}$ satisfying the same conditions as the Morse function $F$.
In particular, the Morse functions $F^{ij}$ are chosen so that
the action of any intersection point in the Cthulhu complex $Cth(\lag^i,\lag^j)$ is less than the action of any pure Reeb chord of $\Lambda_-$ and that the top and bottom cylinders of $3\lag$ agree with the cylinder over $3\Lambda_{\pm}$ in $\aug_+(\Lambda_{\pm})$. 

For $i=1,2,3$, given   augmentations $\e_-^{i}$ of $\alg(\Lambda_-)$, the induced augmentations $\e_+^i$  of $\alg(\Lambda_+)$, and a \textit{domain dependent almost complex structure} $\mathbf{J}_z$ with values in $\mathcal{J}^{adm}$,
the operator 
$$m_2^{-\infty}: C_{-\infty}(\lag^2, \lag^3)\otimes C_{-\infty}(\lag^1, \lag^2)\to C_{-\infty}(\lag^1, \lag^3)$$
counts rigid holomorphic disks with a puncture positively asymptotic to  a generator of $C_{-\infty}(\lag^1,\lag^3)$, two punctures negatively asymptotic to a generator of  $C_{-\infty}(\lag^1,\lag^2)$ and 
a generator of $C_{-\infty}(\lag^2,\lag^3)$, and punctures negatively asymptotic to pure Reeb chords of $3\Lambda_-$ that are augmented by $\e_-^i$.
The operator $m_2^{-\infty}$ can be decomposed as the sum of the maps
$$\mu_{i,j}^k: C_{i}(\lag^2, \lag^3)\otimes C_{j}(\lag^1, \lag^2)\to C_{k}(\lag^1, \lag^3)$$
for $i,j,k\in\{0,-\}$, satisfying the following:
\begin{enumerate}
	\item
	In the special setting we consider here, namely the $3$-copy $3\lag$, energy restrictions guarantee that if one of the inputs of $m_2^{-\infty}$ is in $C_0$, then the output is in $C_0$, i.e. $\mu_{0,j}^{-}=\mu_{i, 0}^{-}=0$, for $i,j=0,-$.
	\item
	The map $\mu_{-,-}^{-}$ agrees with the usual $m_2^-$ of $\aug_+(\Lambda_-)$.
	\item
	For the rest of the cases, $(i,j,k)=(-,0,0), (0,-,0),(0,0,0), (-,-,0)$,
	the map $\mu_{i,j}^k$ 
	counts rigid holomorphic disks with boundary on $3\lag$:
	\small
	$$\displaystyle{\mu_{i,j}^k(c^{23}, b^{12}) = \sum_{\dim \M_{\mathbf{J}_z}(a^{13}; {\bf p}_-^{11} b^{12} {\bf q}_-^{22} c^{23} {\bf r}_-^{33})=0} |\M_{\mathbf{J}_z}^{\lag^1\cup\lag^2\cup\lag^3}(a^{13}; {\bf p}_-^{11} b^{12} {\bf q}_-^{22} c^{23} {\bf r}_-^{33})| \e_-^1({\bf p}_-^{11}) \e_-^2 ({\bf q}_-^{22}) \e_-^3({\bf r}_-^{33}) \ a^{13} },$$
	\normalsize
	where $a^{13}$ is a generator of  $C_k(\lag^1, \lag^3)$, $b^{12}$ is a generator of $C_j(\lag^1,\lag^2)$, $c^{23}$ is a generator of  $C_i(\lag^2, \lag^3)$, and ${\bf p}_-^{11} , {\bf q}_-^{22}, {\bf r}_-^{33}$  are words of pure Reeb chords of $\Lambda_-^1$, $\Lambda_-^2$, $\Lambda_-^3$, respectively.
\end{enumerate}

In \cite[Section 5.2]{Noe},  it is shown that $m_2^{-\infty}$ commutes with the differentials $d_{-\infty}$ and thus induces a product map on cohomology
$$
m_2^{-\infty}: H^{m} C_{-\infty}(\lag^2, \lag^3) \otimes H^{n}C_{-\infty}(\lag^1, \lag^2)\to H^{m+n} C_{-\infty}(\lag^1, \lag^3).
$$
Moreover, we have the following proposition:

\begin{proposition}[{\cite[Theorem 2]{Noe}}]\label{prop:ring}
	The map $\phi_*$  from Equation~(\ref{eqn:C-Hom}) preserves the product structures, i.e.
	$$\phi_* \circ m_2^{-\infty}([a], [b])=m_2^+(\phi_*[a], \phi_*[b])$$
	for $[a]\in H^{*} C_{-\infty}(\lag^2, \lag^3)$ and $[b]\in H^{*} C_{-\infty}(\lag^1, \lag^2)$.
\end{proposition}

\begin{remark}\label{rmk:abuse}
	Note that we abuse the notation of $\phi_*$ here, for simplicity of notation. 
	This isomorphism $\phi_*$ is defined on each pair of cobordisms $2\lag$ in $3\lag$ and could be different for the different $2$-copies.
	 A more rigorous way of writing the identity in the proposition above would be
	$$\phi^{(\e^1_-,\e^3_-)}_* \circ m_2^{-\infty}\left([a], [b]\right)=
	m_2^+\left(\phi^{(\e^2_-,\e^3_-)}_*[a], \phi^{(\e^1_-,\e^2_-)}_*[b]\right).$$
\end{remark}

\section{Obstructions to exact Lagrangian cobordisms between links}\label{sec:main}

In this section, we give an obstruction to the existence of embedded, Maslov-$0$, exact Lagrangian cobordisms  through a count of augmentations 
of the bottom and top Legendrian links.   
We will count augmentations up to $\Sim$, the equivalence in  $\aug_+(\leg_\pm)$ (Definition~\ref{defn:equal}}), which   
by Proposition~\ref{prop:equivalences} 
 is the same as the split-DGA homotopy equivalence (Definition~\ref{defn:split-DGA-homotopy}).
The obstruction through a count of augmentations is proven in Section \ref{sec:proof}, with the proofs of key propositions provided in \ref {sec:indep} and \ref{sec:unitalring}.
Section \ref{sec:obs} provides other obstructions in terms of linearized contact homology and ruling polynomials.

\subsection{Proof of Theorem \ref{thm:main}}\label{sec:proof}   
Throughout this subsection,  we suppose that $\e_-$ is an augmentation for $\alg(\Lambda_-)$ and $\e_+ = \mathcal F_{\lag}(\e_{-})$ is the augmentation of $\alg(\Lambda_+)$ induced by $\e_-$ through $L$.
Since we will be counting augmentations up to equivalence in  $\aug_+(\leg_\pm)$,
 we first define maps  $\iota: H^0Hom_+(\e_+^i, \e_+^j) 
\to H^0Hom_+(\e_-^i, \e_-^j)$, for $i,j\in\{1,2\}$.
Consider the special pair of cobordisms $2\lag$ as described in Section \ref{sec:pair} and the isomorphism $\phi_*: H^*(C_{-\infty}, d_{-\infty})\to H^*Hom_+(\e^i_+, \e^j_+)$ in Equation~(\ref{eqn:C-Hom}).  Note that $(C_0, d_{00})$ 
is a subchain complex of $(C_{-\infty}, d_{-\infty})$.
Combining this fact with  Proposition \ref{prop:des}, it follows that the quotient map $\pi: C_{-\infty}\to C_-$ induces a map on cohomology:
$$\pi_*: H^*(C_{-\infty}, d_{-\infty})\to H^*Hom_+(\e^i_-, \e^j_-).$$
Precomposing with $\phi_*^{-1}$ gives a map 
$$  \iota= \pi_*\circ \phi_*^{-1}: H^*Hom_+(\e_+^i, \e_+^j) \to H^*Hom_+(\e_-^i, \e_-^j).
$$
 
The next proposition shows
that $\iota$ is ``natural'':  although $\phi_\ast$ may depend on the Morse perturbation function $F$ used to construct $2\lag$, $\iota$ does not.  The proof of this proposition can be found in Section~\ref {sec:indep}.

\begin{proposition}\label{prop:indep}
The maps $\iota:H^*Hom_+(\e_+^i, \e_+^j) \to H^*Hom_+(\e_-^i, \e_-^j)$  
are independent of the choice of the Morse perturbation function $F$, up to compactly supported homotopy.
\end{proposition}

The following properties of the  $\iota$ map are used in the proof of Theorem~\ref{thm:main}, and are proved in Section \ref{sec:unitalring}.
\begin{proposition}\label{prop:unit}
 The map $\iota:H^kHom_+(\e_+^i, \e_+^j) \to H^kHom_+(\e_-^i, \e_-^j)$ satisfies the following properties:
	\begin{enumerate}
		\item $\iota$ preserves the product structures, i.e.
		\begin{alignat*}{1}
		m^-_2(\iota[a], \iota[b])=\iota(m^+_2([a],[b]))
		\end{alignat*}
		for $[a]\in H^*Hom_+(\e_+^2,\e_+^3)$ and $[b]\in H^*Hom_+(\e_+^1,\e_+^2)$, where $m^{\pm}_2$ are the products in the augmentation categories $\aug_+(\Lambda_{\pm})$,
		\item $\iota$ is unital, meaning that when $\e_\pm^1=\e_\pm^2=\e_\pm$, we have 
		$\iota([e_{\e_+}])=[e_{\e_-}]$. 
	\end{enumerate}
\end{proposition}

\begin{proof}[Proof of Theorem \ref{thm:main}]
Let $\lag$ be an embedded,  Maslov-$0$, exact Lagrangian cobordism from $\Lambda_-$ to $\Lambda_+$, and $\e_{-}^1, \e_-^2$ be two augmentations of $\alg(\leg_{-})$. 
To show that 
$$  |Aug(\Lambda_-;\F)/\Sim| \leq |Aug(\Lambda_+;\F)/\Sim|$$
we show that if 
the induced augmentations $\e_+^1 = \mathcal{F}(\e_-^1)$ and $\e_+^2 = \mathcal{F}(\e_-^2)$ are equivalent  then
 $\e_-^1$ and $\e_-^2$ are also equivalent.
 Since $\e_+^1, \e_+^2$ are equivalent,  there exist $[\alpha]\in H^0Hom_+(\e_+^1, \e_+^2)$ and $[\beta]\in H^0Hom_+(\e_+^2, \e_+^1)$ such that 
$$m^+_2([\alpha],[\beta])= [e_{\e_+^2}]\in H^0Hom_+(\e_+^2, \e_+^2), \mbox{ and } m^+_2([\beta],[\alpha])= [e_{\e_+^1}]\in H^0Hom_+(\e_+^1, \e_+^1),$$
where $[e_{\e_+^i}]$ is the unit in $H^0Hom_+(\e_+^i, \e_+^i)$, for $i=1,2$.
By Proposition~\ref{prop:unit}, 
  $$m^-_2(\iota[\alpha], \iota[\beta])=\iota(m^+_2([\alpha],[\beta]))= \iota([e_{\e_+^2}])= [e_{\e_-^2}].$$
Analogously, one can prove that $m^-_2(\iota[\beta], \iota[\alpha] )=[e_{\e_-^1}]$.  It follows that  $\e_{-}^{1}$ and $\e_{-}^{2}$ are equivalent, as desired.
\end{proof}
If $\leg_{\pm}$ are Legendrian knots or if $\leg_{\pm}$ are Legendrian links and  $\F = \Z_{2}$, as mentioned in Remark~\ref{rem:aug-equiv}(1), the map
$$
 \mathcal{F}_\lag:Aug(\Lambda_-;\F)/\Sim\to Aug(\Lambda_+; \F)/\Sim
$$
exists; the above argument shows that $ \mathcal{F}_\lag$ is injective.

\subsection{Proof of Proposition~\ref{prop:indep}}\label{sec:indep}

\begin{proof}[Proof of Proposition~\ref{prop:indep}]
Following the construction in Section~\ref{sec:pair}, 
suppose that $F$ and $F'$ are two Morse functions on $\lag$, homotopic through a homotopy  with compact support, 
and let 
$2\lag=\lag^1\cup \lag^2$ and $2\lag'=\lag^{1'}\cup \lag^2$ denote the corresponding $2$-copies.
The  homotopy between $F$ and $F'$ induces a compactly supported Lagrangian isotopy between $2\lag$ and $2\lag'$;
note that the isotopy keeps the two cylindrical ends fixed. 
According to \cite[Proposition 6.4]{CDRGG},  the isotopy induces a chain map 
$$\varphi:\  Cth(\lag^1, \lag^2)\to Cth(\lag^{1'}, \lag^2).$$

Following \cite{EkhSFT},  we will show that the map $\varphi$ is the identity map on $C_+(\lag^1, \lag^2)\to C_+(\lag^{1'}, \lag^2)$.  
Along a generic isotopy $\{\lag^1_s\}_{s\in[0,1]}$
from $\lag^1_{0}:=\lag^1$ to $\lag^1_1:=\lag^{1'}$, one can assume that except for a finite number of distinct points $0<s_0<s_1<\dots<s_r<1$, the cobordisms $\lag_s^1$ and $\lag^2$ are transverse and the moduli spaces contributing to the differential of $Cth(\lag_s^1,\lag^2)$ are transversely cut out. At the points $s_j$, two different situations can occur:
	\begin{enumerate}
		\item The birth/death of a pair of intersection points, $c_1,c_2\in C_0$ with $|c_1|=|c_2|+1$;
		\item The appearance of a $(-1)$-disk  $u\in \M(c_1;{\bf p}c_2{\bf q})$ with boundary on the non-cylindrical parts of the cobordisms.
	\end{enumerate}
Moreover, one can assume that these two cases do not occur simultaneously. Hence, from now on, let us assume that $s_0\in(0,1)$ is the only point in the isotopy when situations $(1)$ or $(2)$ can occur.
 Suppose first that case $(1)$ occurs, and denote the Cthulhu chain complex with (resp. without) the pair of intersection points by $(C[+],d[+])$ (resp. $(C[-],d[-])$). We have $d[+](c_2)=c_1+v$ where $v$ does not contain $c_1$.
 The induced chain map $C[+]\to C[-]$ corresponding to the death of the pair of intersection points $c_1, c_2$ maps $c_2\to 0$, $c_1\to -v$ and other elements to themselves.
 The induced chain map $C[-]\to C[+]$ corresponding to the birth of $c_1, c_2$ sends an element $c$ to $c-c_1^*(d[+] c)c_2$, where $c_1^*$ is the dual element for $c_1$.
 Note that both $c_1$ and $c_2$ are intersection points, thus the induced chain maps are identity maps on $C_+$.
In the second case, a $(-1)$-disk $u\in \M(c_1;{\bf p}c_2{\bf q})$ appears. The induced map $\varphi$ sends $c_2$ to $c_2+ \lambda c_1$, for some number $\lambda$, and all other elements to themselves.
Since the negative puncture $c_2$ is not in $C_+$, the induced chain map is the identity on $C_+$.

Denote $\varphi_{-\infty}$ the component
$$\varphi_{-\infty}:  \ (C_{-\infty}(\lag^1, \lag^2), d_{-\infty}) \to (C_{-\infty}(\lag^{1'}, \lag^2), d'_{-\infty})$$
The fact that $\varphi$ is a chain map and fixes $C_+$ implies that $\varphi_{-\infty}$ is a chain map, i.e. $\varphi_{-\infty}\circ d_{-\infty}=d_{-\infty}'\circ\varphi_{-\infty}$.
Let us then denote $\varphi_+$ the component 
$$\varphi_+:Cth(\lag^1,\lag^2)\to C_+(\lag^{1'},\lag^2)$$
The fact that $\varphi$ is a chain map implies that for any cycle $c\in C_{-\infty}(\lag^1, \lag^2)$ one has
$$\varphi_+\circ(\phi+d_{-\infty})(c)=d_{++}'\circ\varphi_+(c)+\phi'\circ\varphi_{-\infty}(c)$$
Using the fact that $\varphi$ is the identity map on $C_+$ and $d_{-\infty}(c)=0$, this equation becomes,
$$\phi (c)= d'_{++}\circ\varphi_+(c)+\phi' \circ \varphi_{-\infty}(c).$$

It follows that
$$[ \phi (c)]= [\phi' \circ \varphi_{-\infty}(c)] \in H^{*}(C_{+}(\lag^{1'},\lag^2))= H^*Hom_+(\e^i_+, \e^j_+).$$
In order to show that $\iota:=\pi_*\circ \phi_*^{-1}= \pi'_*\circ (\phi'_*)^{-1}=:\iota'$, where $\pi$ and $\pi'$ are the projection maps from $C_{-\infty}\to C_-$ for the two cases, respectively,
 we will prove  that 
\begin{equation}\label{eq:toprove}
\mbox{if } c\in C_{-\infty}(\lag^1, \lag^2),
\mbox{ then } \pi (c)= \pi' \circ \varphi_{-\infty} (c).
\end{equation} 
Again, it suffices to understand how $\varphi_{-\infty}$ behaves when either case $(1)$ or $(2)$ occurs in the isotopy.
If a $(-1)$-disk $u\in \M(c_1;{\bf p}c_2{\bf q})$ occurs, the positive puncture $c_1$ can be an element in $C_+$ or $C_0$, and by definition of $\varphi_{-\infty}$ we only need to consider disks $u$ with $c_1\in C_0.$
Then, the induced chain map $\varphi_{-\infty}$ sends any element $c$ to $c+m$ for $m=0$ or $m \in C_0$.
If case $(1)$ occurs, and we have a birth/death of intersection points $c_1,c_2$ in $C_0$, denote the chain complex with (resp. without) the pair of intersection points by $(C_{-\infty}[+],d_{-\infty}[+])$ (resp. $(C_{-\infty}[-],d_{-\infty}[-])$). Suppose that $d_{-\infty}[+](c_2)=c_1+v$.
Since the differential of $C_{-\infty}$ is upper triangular, we know that $v$ is in $C_0$.
Thus, the map from $C_{-\infty}[+]$ to $C_{-\infty}[-]$ maps $c_1$ to $C_0$ and $c_2$ to $0$. If we have a birth of intersection points, the map from $C_{-\infty}[-]$ to $C_{-\infty}[+]$ sends an element $c$ to  $c-c_1^*(d_{-\infty}[+] c)c_2$, which is also in $C_0$.
In both cases we have shown \eqref{eq:toprove} is true and can conclude that the map $\iota$
does not depend on the choice of Morse function $F$.
\end{proof}

\subsection{Proof of Proposition \ref{prop:unit}}\label{sec:unitalring}
To prove the first statement of  Proposition \ref{prop:unit}, first 
recall that $\iota= \pi_*\circ \phi_*^{-1}$ and that $\phi^{-1}_*$  preserves the product structures; see Proposition \ref{prop:ring}.
Thus  Proposition \ref{prop:unit} (1) follows immediately from

\begin{lemma}
The map $\pi_*: H^*(C_{-\infty}, d_{-\infty})\to H^* (C_-, d_{--})$ preserves the products.
\end{lemma}

\begin{proof}
Recall  that $m_2^{-\infty}(a, b)\in C_0$ if  $a$ or $b$ is in $C_0$.
Thus the component of $m_2^{-\infty}(a, b)$ with values in $C_-$ only comes from $m_2^{-}(\pi(a), \pi(b))$, i.e.
$$
\pi \circ m_2^{-\infty}(a, b)= m_2^{-}(\pi(a), \pi(b)).
$$
\end{proof}

In order to prove  Proposition \ref{prop:unit} (2),
we need  that for any augmentation $\e_-$ of $\alg(\Lambda_-)$ and its induced augmentation $\e_+$ of $\alg(\Lambda_+)$, the map 
$$\iota: H^0Hom_+(\e_+,\e_+)\to H^0Hom_+(\e_-,\e_-)$$ preserves the unit.
Note that $\phi_*^{-1}$ is an isomorphism that preserves the product structures and thus sends the unit $[e_+]$ of $H^0Hom_+(\e_+, \e_+)$ to a unit $[e_{-\infty}]$ of $H^0(C_{-\infty})$.
In order to show $\pi_*([e_{-\infty}])\in H^0Hom_+(\e_-,\e_-)$ is the unit $[e_-]$ of $H^0Hom_{+}(\e_{-},\e_{-})$, 
we only need to prove the following lemma.
\begin{lemma}\label{lem:closed}
There is an element $e=e_-+e_0\in C_{-\infty}$, where $e_0$ is an element in $C_0$, such that $d_{-\infty} (e)=0$.
\end{lemma}

\begin{proof}[Proof of Proposition \ref{prop:unit}]
With Lemma \ref{lem:closed} in hand, the fact that $\pi_*$ preserves the product structure, and the fact that $[e_{-\infty}]$ and $[e_-]$ are the units of $H^*(C_{-\infty})$ and $H^0Hom_+(\e_-, \e_-)$ respectively, we have that 
\begin{alignat*}{1}
	\pi_*([e_{-\infty}])= m_2^-([e_-], \pi_*([e_{-\infty}]))&= m_2^-(\pi_*[e], \pi_*([e_{-\infty}]))\\
	&=\pi_*\circ m_2^{-\infty}([e], [e_{-\infty}])=\pi_*[e]=[e_-].
\end{alignat*}
Thus, $\iota([e_+])= \pi_*\circ \phi^{-1}_*([e_+])= \pi_*([e_{-\infty}])=[e_-].$
\end{proof}

\begin{proof}[Proof of Lemma \ref{lem:closed}]
Recall that the unit $e_{-}$ of $Hom^0_+(\e_{-},\e_{-})$ is given by 
$e_-=-\sum y^{12}_i$, where $y^{12}_i$ are the Reeb chord of $2\Lambda_-$ corresponding to the Morse minima of the Morse function $f_-$ used to define $2\Lambda_-$.
Let $e_0$ be negative of the sum of all the intersections that corresponds to the minima of the Morse function $F$, and then let $e=e_0+e_-$.
We have that $$d_{-\infty}(e)=d_{00}(e_0)+ d_{0-}(e_-)+ d_{--}(e_-).$$
The fact that $e_-$ is closed in $Hom_+(\e_-,\e_-)$ implies $d_{--}(e_-)=0$.
It follows from Proposition \ref{prop:des}  that $d_{00}$ counts negative Morse flow lines of the Morse function $F$.
We need to interpret the holomorphic disks counted by $d_{0-}(e_-)$ in terms of Morse flow lines of a Morse function $\wt{F}$ that agrees with $F$ in the main part but also encodes the Morse function $f_-$ on the bottom cylinder.
This can be done by concatenating a cobordism from the bottom and comparing the Cthulhu complexes of the two pairs of cobordisms using a \textit{transfer map} defined in \cite{CDRGG}. The remainder of the proof is dedicated to describing $d_{00}(e_0)$ and $d_{0-}(e-)$ in detail. 

Recall that $\Lambda_-^1$ is a push off of $\Lambda_-^2$ using a  very small positive Morse function $f_-$. 
Let $A \in \mathbb R^+$ be twice the maximum value of $f_-$. Consider the cylinder $\R\times \Lambda^1_-$ and push the negative end of the cylinder in the $-z$ direction by $A$.
Denote this new Legendrian in the negative end by 
$\Lambda^1_--A$.
Thus, we get a cobordism $W^1$ from $\Lambda^1_--A$ to $\Lambda^1_-$ as shown in Figure \ref{fig:wrapbot}.
Concretely, consider a non-increasing Morse function $\delta(t):\R\to \R$ which is $0$ when $t>-N-1$ and is equal to the constant $A$ 
when $t<N'$, for some $N'<-N-1$. Note that $X_H=-\delta(t)\dd/\dd z$ is a Hamiltonian vector field, and denote its time $1$ flow by $\Phi_H$.
It follows that $W^1:=\Phi_H(\R\times\Lambda^1_-)$ is an exact Lagrangian cobordism.
Denote by $W^2$ the cylinder $\R\times \Lambda^2_-$.

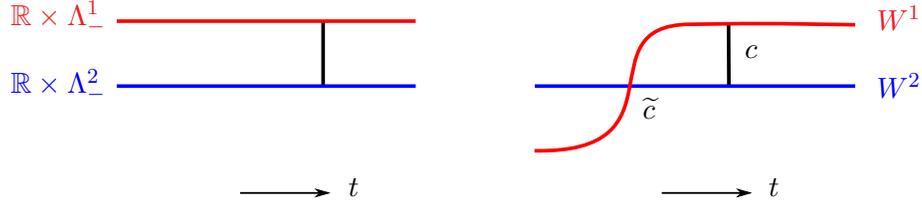
\begin{figure}[!ht]
\labellist
\small
{\color{Red}
\pinlabel $\R\times\Lambda_-^1$ at -20 85
\pinlabel $W^1$ at  350 85
}
{\color{blue}
\pinlabel $\R\times\Lambda_-^2$ at -20 55
\pinlabel $W^2$ at 350 55
}
\pinlabel $t$ at 110 10
\pinlabel $t$ at 295 10
\pinlabel $c$ at  285 70
\pinlabel $\wt{c}$ at 240 45
\endlabellist
\includegraphics[width=4in]{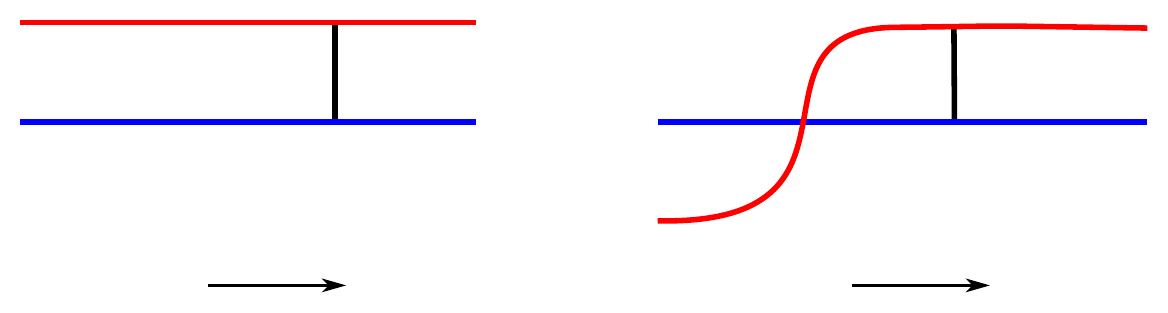}
\caption{A schematic picture of wrapping the negative end of a cobordism.}
\label{fig:wrapbot}
\end{figure}

Observe that there is a natural bijection between $C_0(W^1, W^2)$ and the Morse Reeb chords in $C_+(W^1, W^2)$ with degree shifted up by $1$.
Moreover, we can show that $d_{+0}^{W^1, W^2}$ sends an intersection point to the corresponding Morse Reeb chord, as follows. First, the projection map $\pi_{xy}: \R\times \R^3\to \R_{xy}$ sends $W^1\cup W^2$ to $\pi_{xy}(\Lambda^1_-\cup \Lambda^2_-)$. Then, according to \cite[Proposition 5.11]{DR}, the projection map also sends holomorphic disks with boundary on $W^1\cup W^2$ to holomorphic disks with boundary on $\pi_{xy}(\Lambda^1_-\cup \Lambda^2_-)$.
Suppose that a disk $u\in \M^{W^1\cup W^2}(a, {\bf p}_-^{11}b {\bf q}_-^{22})$ contributes to $d^{W^1,W^2}_{+0}$, i.e. $a$ is a mixed Reeb chord of $\Lambda^1_-\cup \Lambda^2_-$, $b$ is an intersection point in $W^1\cap W^2$ and ${\bf p}_-^{11},{\bf q}_-^{22}$ are words of pure degree $0$ Reeb chords of $\Lambda^1_--A$ and $\Lambda^2_-$, respectively.
The rigidity of $u$ implies that $|a|-|b|=1$ using the grading in the Cthulhu complex.
Projecting down to the $xy$-plane, we have that $\pi_{xy}(u)\in \M(\pi_{xy}(a); {\bf p}_-^{11} \pi_{xy}(b) {\bf q}_-^{22})$ is a holomorphic disk with boundary on $\pi_{xy}(\Lambda^1_-\cup \Lambda^2_-)$. 
Comparing the grading in the Cthulhu complex and the grading in $\alg(\Lambda^1_-\cup \Lambda^2_-)$, we have 
$$|a|= |\pi_{xy}(a)|_{LCH}+2 \mbox{ and } |b|=|\pi_{xy}(b)|_{LCH}+1.$$
It follows that $|\pi_{xy}(a)|_{LCH}-|\pi_{xy}(b)|_{LCH}=0$, or in other words, the expected dimension of $\M(\pi_{xy}(a); {\bf p}_-^{11} \pi_{xy}(b) {\bf q}_-^{22})$ is $-1$, which implies that $\pi_{xy}(u)$ is constant and thus $|\pi_{xy}(a)|=|\pi_{xy}(b)|$.
Therefore, we have proved that $d_{+0}^{W^1,W^2}$ sends an intersection point in $W^1\cap W^2$ to the corresponding Morse Reeb chord of $\Lambda_-^1\cup \Lambda_-^2$.

Consider now the Cthulhu complex of the pair of concatenated cobordisms $(W^1\odot \lag^1, W^2\odot \lag^2)$. Its generators can be decomposed into four types.
$$Cth(W^1\odot \lag^1, W^2\odot \lag^2)=
C_-(W^1, W^2)\oplus C_0(W^1, W^2)\oplus C_0(\lag^1, \lag^2) \oplus C_+(\lag^1, \lag^2)$$
According to \cite{CDRGG}, there is a chain map 
$$\Psi^W: Cth(W^1\odot \lag^1, W^2\odot \lag^2)\to Cth(\lag^1, \lag^2)$$
which is $d^{W^1,W^2}_{+0}$ on $C_0(W^1, W^2)$, is $d^{W^1,W^2}_{+-}$ on $C_-(W^1, W^2)$ and is the identity on $C_0(\lag^1, \lag^2) \oplus C_+(\lag^1, \lag^2)$ (in the case of the special pair of cobordisms we are considering in this paper).
Due to action restrictions, Morse Reeb chords do not show up in the image of 
$d^{W^1,W^2}_{+-}$ but only in the image of $d^{W^1,W^2}_{+0}$.

Denote the intersection point in $W^1\cap W^2$ corresponding to a Morse Reeb chord $c$ of $\Lambda^1_- \cup \Lambda^2_-$ by $\wt{c}$, as shown in Figure \ref{fig:wrapbot}.
Due to the description of $d^{W^1,W^2}_{+0}$,
the chain map $\Psi^W$ identifies the holomorphic disks counted by $d^{W^1\odot \lag^1,W^2\odot \lag^2}_{00}(\wt{c})$ such that the positive puncture is in $C_0(\lag^1,\lag^2)$, with the holomorphic disks counted by $d^{\lag^1,\lag^2}_{0-} (c)$.
Thus, we can describe $d^{\lag^1,\lag^2}_{0-} (e_-)$ through $d^{W^1\odot \lag^1,W^2\odot \lag^2}_{00}(\wt{e}_-)$, where 
$\wt{e}_-=-\sum \wt{y}_i$ and $\wt{y}_i$ are the intersection points corresponding to the Morse Reeb chords $y_i$ of $\Lambda_-^1\cup\Lambda^2_-$.

Observe that $W^1\odot \lag^1$ happens on a small neighborhood of $W^2\odot\lag^2=\lag^2$ and thus can be described as a push-off of $\lag^2$ along a Morse function $\wt{F}$.
Note that the Morse function $\wt{F}$ agrees with $F$ on $([-N, N]\times\R^3)\cap \lag^2$ but has also minima at $\wt{y}_i$ and saddle points at $\wt{x}_i$.
Since $W^1\odot \lag^1$ and $W^2\odot\lag^2$ are close enough, the differential 
$d^{W^1\odot \lag^1,W^2\odot\lag^2}_{00}$ counts the negative Morse flow lines of $\wt{F}$. 
Let $\wt{e}= \wt{e}_-+e_0$ be the negative sum of all the minima of $\wt{F}$.
Observe that $d^{W^1\odot \lag^1,W^2\odot\lag^2}_{00} (\wt{e})=0$ since each saddle point of $\wt{F}$ has two Morse trajectories  flowing down with the opposite sign and they have to approach some minima. 
It follows from $\Psi^W$ being a chain map that
$$d_{00}(e_0)+ d_{0-}(e_-)=\pi_0\circ d^{\lag^1,\lag^2}\circ\Psi^W(\wt{e})=
\pi_0\circ \Psi^W\circ d^{W^1\odot \lag^1,W^2\odot\lag^2} (\wt{e})=0,$$
where $\pi_0$ is the projection map: $Cth(\lag^1,\lag^2)\to C_0(\lag^1,\lag^2)$.
\end{proof}

\subsection{Other obstructions.}\label{sec:obs}

In this section, we give two additional obstructions to the existence of exact Lagrangian cobordisms in terms of linearized contact homology and ruling polynomials, which generalize the results in \cite{Pan1}.

\begin{proposition}\label{prop:LCH}
Assume $\F$ is a field and let $\lag$ be an exact Lagrangian cobordism from $\Lambda_-$ to $\Lambda_+$ with Maslov-$0$.
 Suppose that $\e_-$ is an augmentation of $\Lambda_-$ and $\e_+$ is the induced augmentation of $\Lambda_+$.
Then we have that 
\begin{equation}\label{eq:iso}
LCH^{\e_+}_{k}(\Lambda_+)\cong LCH^{\e_-}_{k}(\Lambda_-) 
\end{equation}
    for $k<0$ and $k>1$.
\end{proposition}
 
\begin{proof}
From Equation~\eqref{eq:les}, we have a long exact sequence
$$\cdots \to H^k(\lag, \Lambda_-)\to H^kHom_+(\e_+, \e_+) \to H^kHom_+(\e_-, \e_-)\to H^{k+1}(\lag, \Lambda_-)\to \cdots.
$$
Note that $H^k(\lag, \Lambda_-)=0$ when $k<0$ and $k>2$. For $k=2$, we know that $H^2(\lag,\Lambda_-)$ is $0$ because any two components of $\Lambda_-$ cannot bound a closed surface in $\lag$, i.e. a Lagrangian cap of two components of $\Lambda_-$. 
Otherwise we get a cobordism from a subset of $\Lambda_-$ (that admits an augmentation restricted from $\e_-$) to the  empty set, which is a contradiction by \cite[Corollary 1.9]{dimitroglou_rizell_2015}. 

The long exact sequence implies that 
$$H^kHom_+(\e_-,\e_-)\cong H^kHom_+(\e_+,\e_+),$$
for $k<-1$ and $k>1$.
Recall that $H^kHom_+(\e,\e)\cong LCH^{\e}_{1-k}(\Lambda)$, so we get $$LCH^{\e_+}_k(\Lambda_+)\cong LCH^{\e_-}_k(\Lambda_-),$$ for $k>2$ and $k<0$.

The isomorphism for $k=2$ comes from the Sabloff duality \cite{EESduality}, which gives a long exact sequence:
$$\cdots \to H^k(\Lambda) \to LCH_{\e}^k(\Lambda)\to LCH_{-k}^{\e}(\Lambda)\to H^{k+1}(\Lambda)\to \cdots.$$
The fact that $H^k(\Lambda)$ vanishes unless $k=0$ or $1$
implies that $LCH^{\e}_{-k}(\Lambda)\cong LCH_{\e}^k(\Lambda)$ for $k>1$ and $k<-1$.
Note that $LCH^{\e}_k(\Lambda)$ are vector spaces over a field $\F$. 
It follows from the universal coefficient theorem that $\dim LCH^k_{\e}(\Lambda)=\dim LCH_k^{\e}(\Lambda)$.
Thus, we have that $\dim LCH^{\e}_{-k}(\Lambda)\cong \dim LCH^{\e}_k(\Lambda)$ for $k>1$.
Since the isomorphism \eqref{eq:iso} holds for $k=-2$, the dimension of the LCH homologies are the same for $k=2$,
which implies the isomorphism for $k=2$ as they are vector spaces over $\F$.
\end{proof}

We do not get the relation between the LCH's on degree $0$ and $1$ as Pan did for cobordisms between knots in \cite[Corollary 1.4]{Pan1}. 

\begin{example}
Take $\F=\Z_2$ and consider two exact Lagrangian cobordisms $\lag^1, \lag^2$ from the Hopf link $\Lambda_{\Ho}^0$ to the trefoil  obtained by pinching the chords $b_1$ and $b_2$ of the trefoil, respectively, as shown in Figure \ref{fig:23}.
Let $\e_-^1$, resp. $\e_-^2$, be the augmentation of $\Lambda_{\Ho}^0$ which sends the two Reeb chords $(c_1, c_2)$ to $(0,0)$, resp. $(0,1)$. Both augmentations $\e_-^i,$ induce through $\lag^i$ for $i=1,2$ the augmentation of the trefoil $\e_+$, which sends the three Reeb chords $(b_1, b_2, b_3)$ to $(1,1,0)$.
However, the Legendrian contact homology of $\Lambda_{\Ho}^0$ linearized by $\e_-^1$ has rank one in degrees $0$ and $1$, while linearized by  $\e_-^2$ it has rank $2$ in degrees $0$ and $1$.
Thus, the data $(\lag, \Lambda_+,\Lambda_-, \e_+)$ cannot determine $LCH^{\e_-}(\Lambda_-)$.

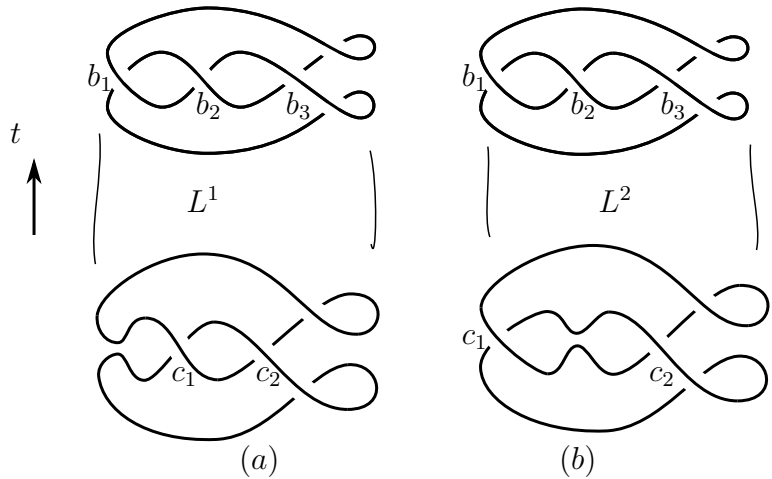
\begin{figure}[!ht]
\labellist
\pinlabel $(a)$ at 130 -5
\pinlabel $(b)$ at  300 -5
\pinlabel $t$ at 0 170
\pinlabel $b_1$ at  45 200
\pinlabel $b_2$ at  103 185
\pinlabel $b_3$ at  151 185
\pinlabel $b_1$ at  245 200
\pinlabel $b_2$ at  303 185
\pinlabel $b_3$ at  351 185
\pinlabel $c_1$ at 90 40
\pinlabel $c_2$ at 135 40
\pinlabel $c_1$ at 245 60
\pinlabel $c_2$ at 345 40
\pinlabel  $\lag^1$ at  100 135
\pinlabel  $\lag^2$ at  320 135
\endlabellist
\includegraphics[width=4in]{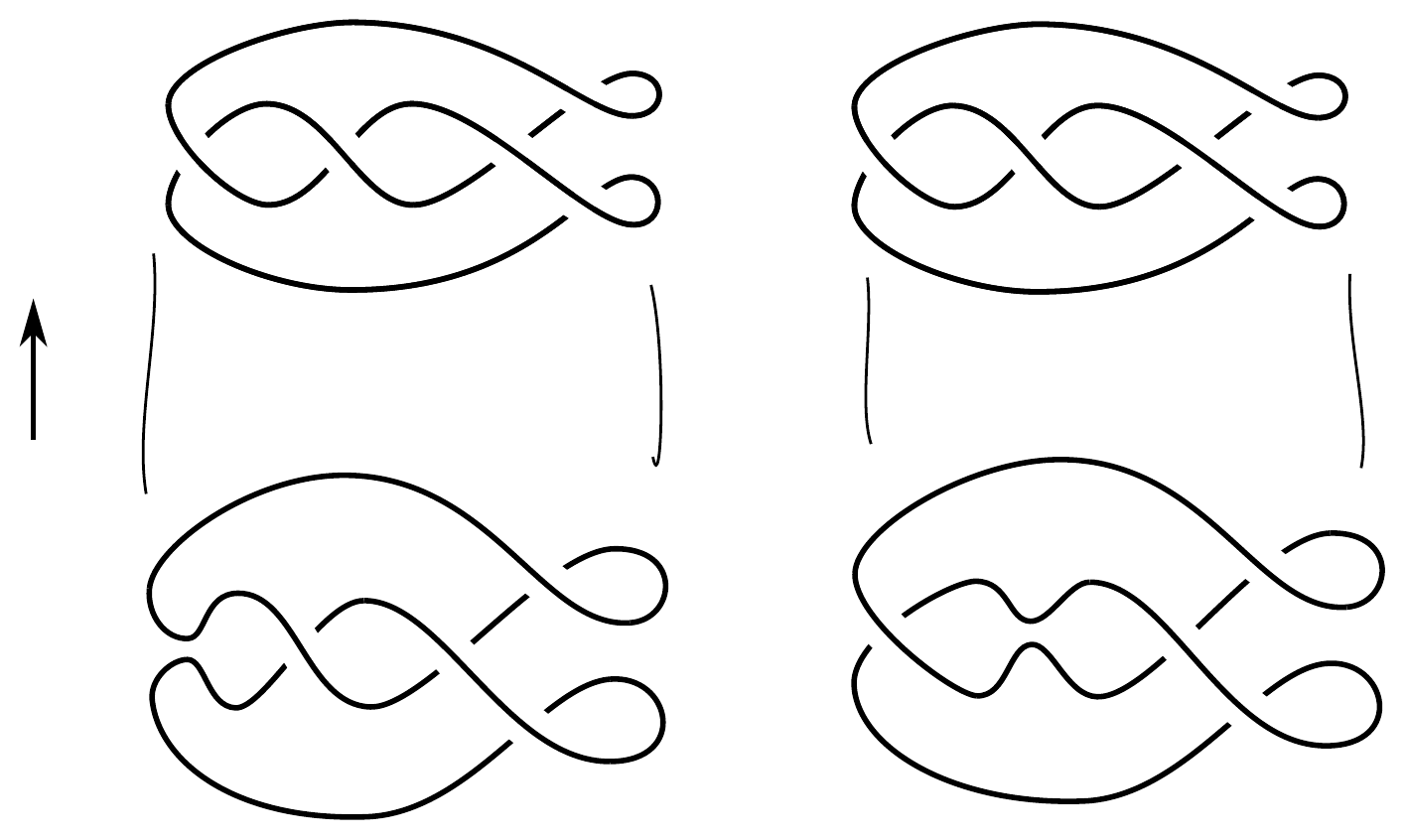}
\caption{Part $(a)$ and $(b)$ shows two cobordisms obtained by doing pinch move on $b_1$ and $b_2$, respectively.}
\label{fig:23}
\end{figure}
\end{example}

Another way to count the number of augmentations in the augmentation category is the \textbf{homotopy cardinality} \cite{NRSS}, which is defined by 
$$\pi_{\ge 0} \aug_+(\Lambda; \F_q)^*= \displaystyle{\sum_{[\e]\in \aug_+(\Lambda;\F_q)/\sim } \frac{1}{|Aut(\e)|}\cdot \frac{|H^{-1}Hom_+(\e,\e)| \cdot |H^{-3}Hom_+(\e,\e)| \cdots }{|H^{-2}Hom_+(\e,\e)|\cdot |H^{-4}Hom_+(\e,\e)| \cdots}},$$
where $[\e]$ is the equivalence class of $\e$ in the augmentation category and $|Aut(\e)|$ is the number of invertible elements in $H^0Hom_+(\e,\e)$.
\begin{proposition}
Let $\lag$ be a spin exact Lagrangian cobordism from $\Lambda_-$ to $\Lambda_+$  with Maslov number $0$. Then for any finite field $\F_q$, we have that 
$$\pi_{\ge 0} \aug_+(\Lambda_+; \F_q)^*\ge \pi_{\ge 0} \aug_+(\Lambda_-; \F_q)^*.$$
\end{proposition}
\begin{proof}
 For each  equivalence class in $Aug_+(\Lambda_-;\F_q)$, we take a representative $\e_-$ and compare the term of $[\e_-]$ in the sum with the term of the induced augmentation $\e_+$ for $Aug_+(\Lambda_+;\F_q)$.
It follows from Proposition \ref{prop:LCH} that the $H^kHom_+$ spaces are isomorphic between  $\e_-$ and $\e_+$ for $k<0$.
Moreover, it follows from Theorem \ref{thm:main} that if an element $[\alpha_+]\in H^0Hom_+(\e_+,\e_+)$ is invertible, then $\iota[\alpha_+]\in H^0Hom_+(\e_-,\e_-)$ is invertible.
Thus $H^0Hom_+(\e_-,\e_-)$ may have more invertible elements than $H^0Hom_+(\e_+,\e_+)$.
It follows that for each equivalent class represented by $\e_-$, the term in the summand for $\e_+$ is bigger than or equal to the term for $\e_-$.
Moreover, there may be more equivalence classes in $\aug_+(\Lambda_+)$ than in $\aug_+(\Lambda_-)$. 
Thus the proposition follows.
\end{proof}
  
The homotopy cardinality is related to the \textit{ruling polynomial} $R_{\Lambda}(z)$, a combinatorial invariant of Legendrian knots that is easily computed, in the following way:
  $$\pi_{\ge 0} \aug_+(\Lambda; \F_q)^*= q^{tb(\Lambda)/2} R_{\Lambda}(q^{1/2}-q^{-1/2}).$$
See Section~\ref{ssec:g-p-augs} for more details on the ruling polynomial. 
Thus we have the following corollary.
\begin{corollary}\label{cor:rul}
Let $\lag$ is a spin exact Lagrangian cobordism from $\Lambda_-$ to $\Lambda_+$ with Maslov number $0$.
Then, we have that $$R_{\Lambda_-}(q^{1/2}-q^{-1/2})\leq q^{-\chi(\lag)/2} R_{\Lambda_+}(q^{1/2}-q^{-1/2})$$
for any $q$ that is a power of a prime number.
\end{corollary}

\section{Examples of Obstructed Fillings}\label{sec:ex}
In this section, we will prove Theorem~\ref{thm:obstruct}.   To prove that certain immersed Lagrangian fillings of a Legendrian knot $\leg$ do exist, we will use the ``decomposable'' moves described below to prove the existence of embedded Lagrangian cobordisms from a disjoint union of Legendrian Hopf links to $\leg$; recall that, by definition, the Legendrian Hopf link $\leg_{\Ho}^{k}$ admits an immersed, Maslov-$0$, exact Lagrangian filling with one action-$0$ double point of index $k$.
We will  prove that certain types of Lagrangian fillings of $\leg$ cannot exist by applying Theorems~\ref{thm:maingen1} and~\ref{thm:main}. 
Throughout this section, we consider DGAs over $\Z_2$ and augmentations to $\Z_2$.  For the family $\leg_{k}$ in Theorem~\ref{thm:obstruct}(1), we will count augmentations
directly, while for the family $\leg_{g}^{p}$ 
 in Theorem~\ref{thm:obstruct}(2), we will employ the theory of rulings to count augmentations.

 All of the embedded, Maslov-$0$, exact Lagrangian fillings and cobordisms that we construct in this section are \emph{decomposable} in the following sense. It is known that there exists an embedded, Maslov-$0$, exact Lagrangian cobordisms between two Legendrian links $\leg_{\pm}$ if $\Lambda_+$ differs from $\leg_-$ by Legendrian isotopy, pinch moves, and the death of a max $tb$ unknotted component. Figure~\ref{fig:decomp} illustrates  the local front projections of an orientable downward in time pinch move and the downward in time death of a max $tb$ unknot.  In order to produce an orientable surface, the pinch move can only be performed on strands
with opposite orientations, and in order for the Lagrangian to be Maslov-$0$, pinch moves can only be performed on strands whose upper branch has a  Maslov potential $1$ greater than that of the lower branch, as shown in Figure~\ref{fig:decomp}.
A Lagrangian cobordism $L$ from $\leg_{-}$ to $\leg_{+}$ is called {\bf elementary} if it arises from isotopy, a single pinch move, or a single disk filling. A Lagrangian
cobordism is {\bf decomposable} it is obtained by stacking elementary cobordisms.

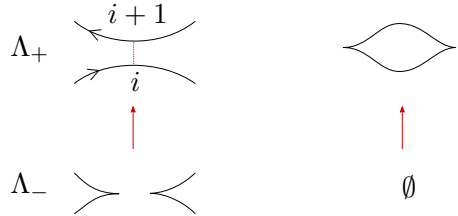
\begin{figure}[!ht]
	\labellist
	\small
	\pinlabel $i+1$ at 58 182
	\pinlabel $i$ at 56 121
	\pinlabel $\leg_{+}$ at -40 150
	\pinlabel $\emptyset$ at 305 25
	\pinlabel $\leg_{-}$ at -40 25
	\endlabellist
	\includegraphics[width=2in]{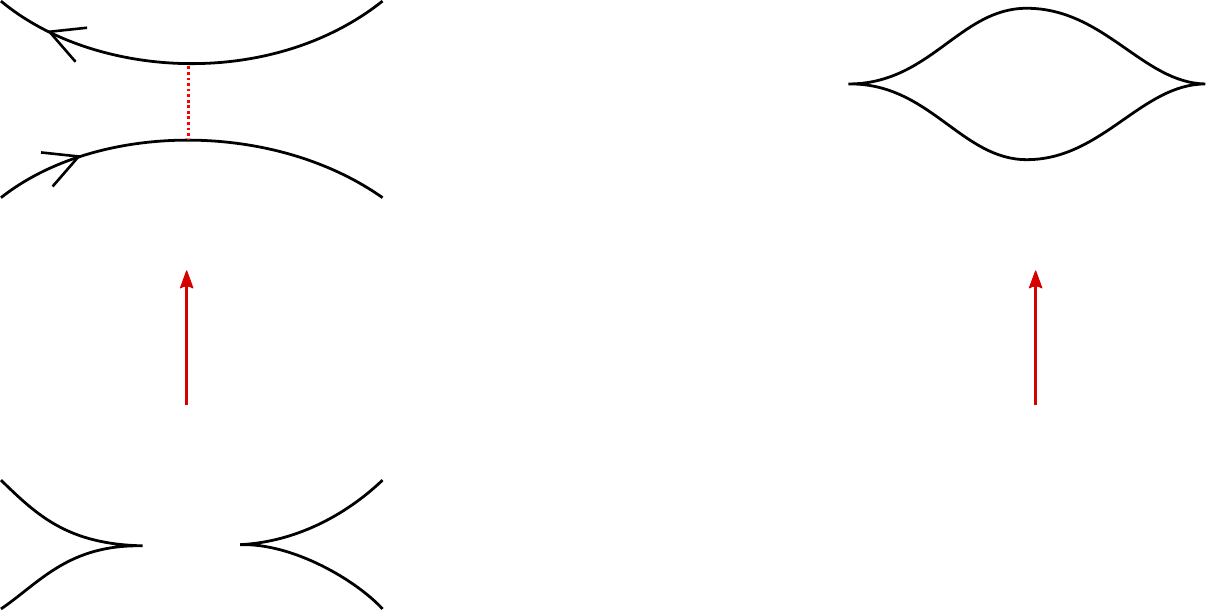}
	\caption{Local front projections of a merge/pinch move (topologically a saddle cobordism/$1$-handle)  and the birth/death of a max $tb$ unknot (topologically a disk/$0$-handle). The red arrows represent the positive $t$ direction and the labels on the strands indicate the Maslov potential. 
	}
	\label{fig:decomp}
\end{figure}

As we apply Theorem~\ref{thm:main}, it will be useful to have the following augmentation count. 
\begin{lemma} \label{lem:hopf-aug-counts}
$$|Aug( \Lambda^k_{\Ho}; \Z_2)/\Sim|=
\begin{cases}
3, & k = 0\\
0, & k \neq 0. 
\end{cases}
$$
Moreover, 
$$|Aug( \bigsqcup_{i = 1, \dots, m} \leg_{\Ho}^{k_{i}}; \Z_2)/\Sim|=3^{Z},$$
where $Z = |\{i : k_{i} = 0\}|.$
\end{lemma}

\begin{proof} As explained in Example~\ref{ex:aug-hopf}, the Hopf link $\Lambda^k_{\Ho}$ has 3 augmentations when $k = 0$ and no augmentations otherwise.  
 When $k = 0$, there are no degree $-1$ chords, and thus, by Corollary~\ref{cor:distinct-not-equiv}, the count of augmentations up to the equivalence relation $\Sim$ is the
same as the count of augmentations.
\end{proof}

\subsection{Proof of Theorem~\ref{thm:obstruct}(1)}
\label{sec:9_48}

We construct the family of Legendrian knots $\leg_{k}$ such that $\leg_{1} = \leg_{9_{48}}$ as follows.
Consider the tangle $T$ in Figure~\ref{fig:fam_k_front}. Arrange $k$ copies $T_1,\dots,T_k$ of $T$ in a row and connect them by a tangle sum; then perform the standard rainbow tangle closure after introducing 1 more crossing, as shown in Figure~\ref{fig:fam_k_front}. The resulting Legendrian $\Lambda_k$ admits a Maslov potential whose values on each strand is also indicated in the figure. When $k=1$, the Legendrian knot obtained this way is a $9_{48}$ knot; its front projection is shown in Figure~\ref{fig:9_48} and its Lagrangian projection in Figure~\ref{fig:9_48lag}.
\begin{figure}[!ht]
	\labellist
	\small
	\pinlabel $T_1$ at  260 58
	\pinlabel $T_2$ at  320 58
	\pinlabel $T_k$ at  430 58
	\pinlabel $\scriptstyle{0}$ at  480 20
	\pinlabel $\scriptstyle{1}$ at  480 35
	\pinlabel $\scriptstyle{1}$ at  480 83
	\pinlabel $\scriptstyle{2}$ at  480 98
	\pinlabel $\scriptstyle{0}$ at  130 31
	\pinlabel $\scriptstyle{1}$ at  130 46
	\pinlabel $\scriptstyle{1}$ at  130 61
	\pinlabel $\scriptstyle{2}$ at  130 76
	\pinlabel $\scriptstyle{1}$ at  142 19
	\pinlabel $\scriptstyle{1}$ at  142 95
	\endlabellist
	\includegraphics[width=4.5in]{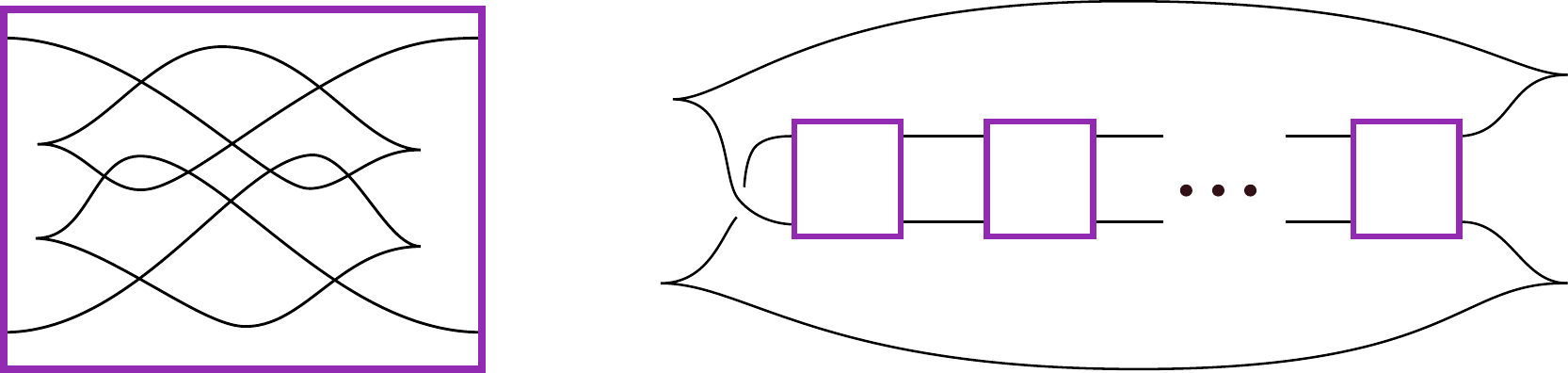}
	\caption{Left: front projection of the tangle $T$; right: front projection of $\Lambda_k$. The numbers indicate the Maslov potential.}
	\label{fig:fam_k_front}
\end{figure}
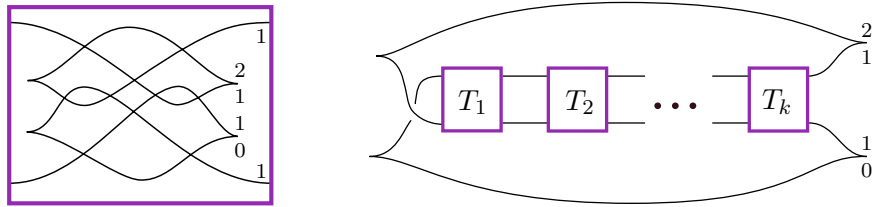

\begin{proposition} \label{prop:948-gen1-fill}   $\leg_{k}$ admits an immersed, Maslov-$0$, exact Lagrangian filling $F_{k}^{k}$ of genus $k$ with k double points, each of which has action $0$ and index $1$.
\end{proposition}

\begin{proof} When $k = 1$, by
performing a sequence of pinch moves as indicated by the red lines in the Figure~\ref{fig:9_48} and Reidemeister moves, we obtain an embedded,
Maslov-$0$, exact Lagrangian cobordism from the Hopf link $\Lambda^{1}_{\Ho}$ to $\leg_{9_{48}}$.  
For $k \geq 2$, by performing pinch moves on each copy of the tangle $T$ as in the case of $\Lambda_{9_{48}}$, we obtain an embedded, Maslov-$0$, exact Lagrangian cobordism of genus $k$ from $\sqcup_k \Lambda_{\Ho}^1$ to $\Lambda_k$.  
 Each
$\Lambda^{1}_{\Ho}$ has an immersed, Maslov-$0$, exact Lagrangian filling with a double point of action $0$ and index $1$.  Stacking these Lagrangian cobordisms
produces the desired filling $F_{k}^{k}$ of $\leg_{k}$.
\end{proof} 

\begin{figure}[!ht]
	\labellist
	\pinlabel $\Lambda_{9_{48}}$ at 0 30
	\pinlabel $\Lambda_{\Ho}^1$ at 360 30
	\endlabellist
	\includegraphics[width=4.5in]{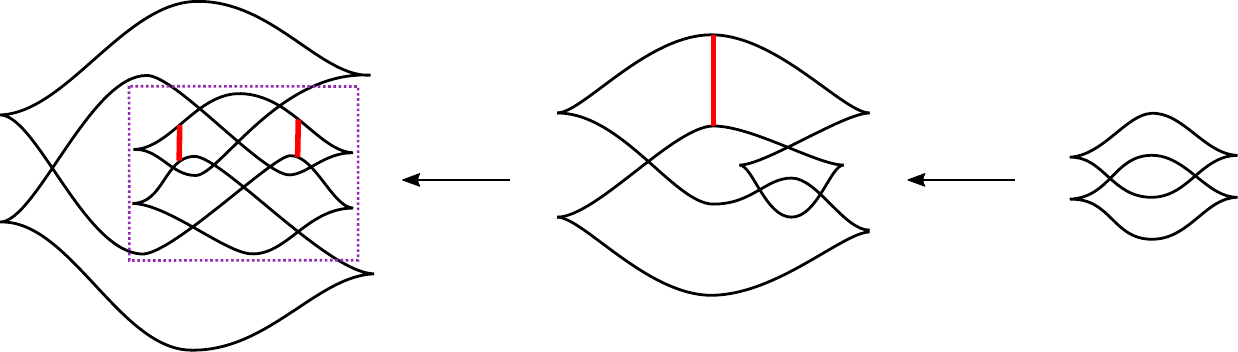}
	\caption{Sequence of three pinch moves that prove the existence of an embedded, Maslov-$0$, exact Lagrangian cobordism of genus $1$ from the Hopf link $\Lambda_{\Ho}^{1}$ to $\leg_{ 9_{48}}$.}
	\label{fig:9_48}
\end{figure}

\begin{proposition}\label{prop:948-disk-obstruct}  $\leg_{k}$ does not admit an immersed, Maslov-$0$, exact Lagrangian disk filling $F_{k-1}^{k+1}$ with $k+1$ double points, all of action $0$ and $k$
of index $1$ and one of index $0$.
\end{proposition}

The proof of this proposition will follow easily once we prove the following count of augmentations.

\begin{lemma}\label{lem:948-augs}  For all $k \geq 1$, $|Aug(\Lambda_{k}; \Z_2)/\Sim|=1.$
\end{lemma}
  
\begin{proof}  When $k=1$, the DGA $\alg(\Lambda_{9_{48}})$ is generated by $a_i, i=1,\cdots 6$, $b_i, i=1, \cdots, 7$, $c_i, i=1,2$ with grading $|a_i|=1, |b_i|=0, |c_i|=-1$ as shown in Figure~\ref{fig:9_48lag}.
\begin{figure}[!ht]
\labellist
\small
\pinlabel $a_1$ at 330 290
\pinlabel $a_2$ at 350 220
\pinlabel $a_3$ at 350 110
\pinlabel $a_4$ at 330 50
\pinlabel $a_5$ at  300 230
\pinlabel $a_6$ at  300 70
\pinlabel $c_1$ at 160 230
\pinlabel $c_2$ at 170 70
\pinlabel $b_1$ at 318 150
\pinlabel $b_2$ at 265 155
\pinlabel $b_3$ at 238 220
\pinlabel $b_4$ at 240 125
\pinlabel $b_5$ at 200 150
\pinlabel $b_6$ at 140 150
\pinlabel $b_7$ at 30 130
\endlabellist
\includegraphics[width=2.5in]{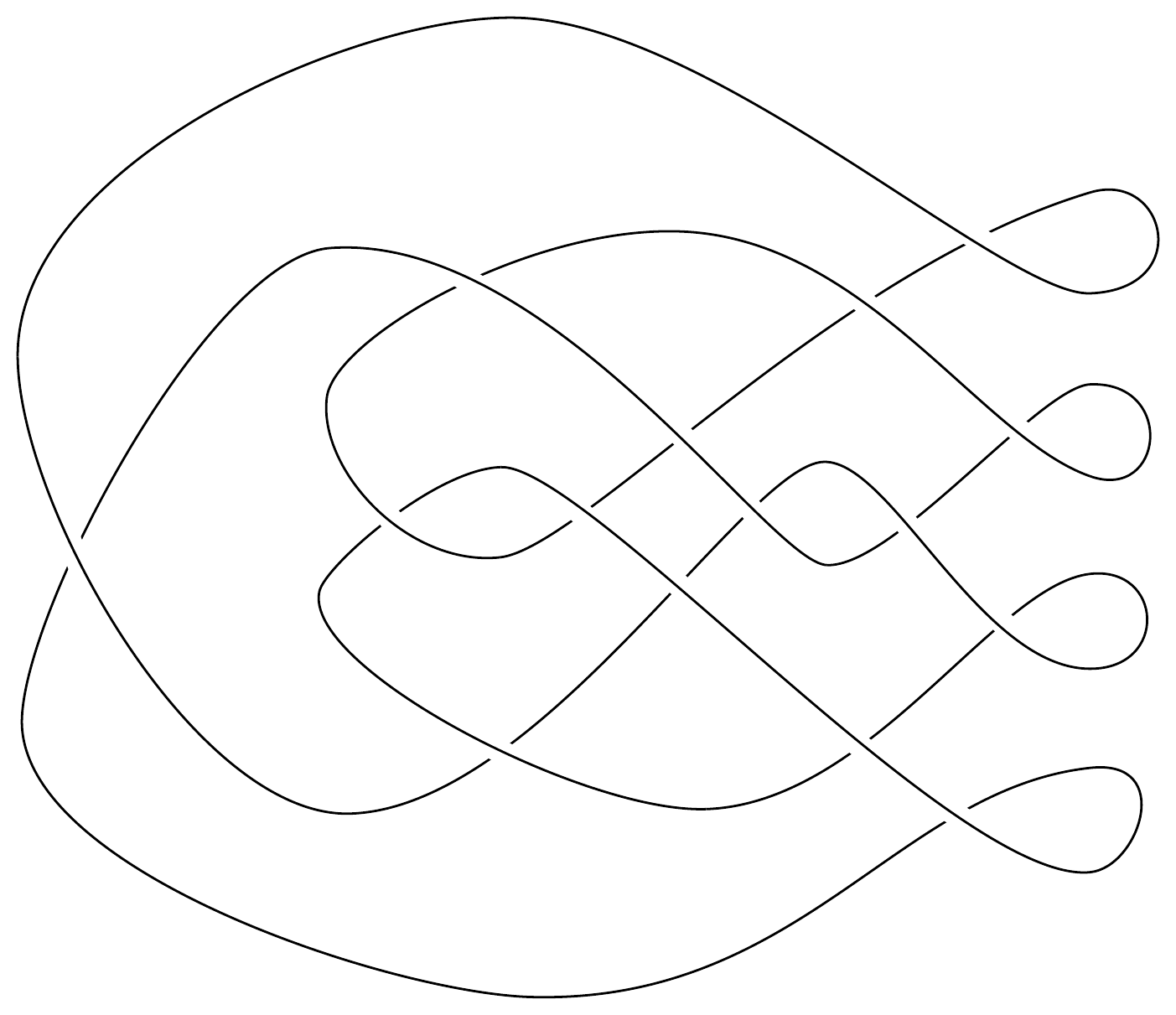}
\caption{A Lagrangian projection for $\Lambda_{9_{48}}$.}
\label{fig:9_48lag}
\end{figure}
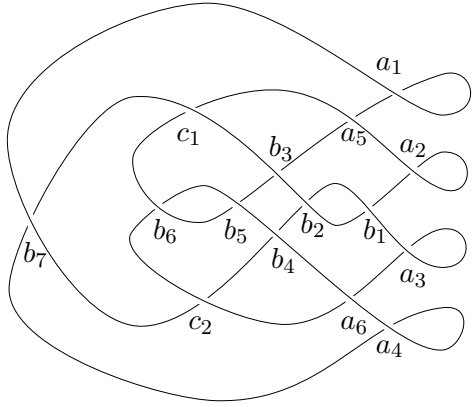
The differential is given by
 $$
 \begin{array}{lcl}
\dd a_1=1+b_7(b_3+c_1 a_5)&&\dd b_2= c_1b_6b_4+b_3b_6c_2\\
\dd a_2= 1+ a_5b_6c_2b_1+b_6 b_4 b_1&&\dd b_3=c_1(1+b_6b_5)\\
\dd a_3=1+ b_1c_1b_6a_6+ b_1b_3b_6&&\dd b_4=(1+b_5b_6)c_2\\
\dd a_4=1+ (b_4+a_6c_2)b_7&&\dd b_i=0, \mbox{ for } i\neq 2,3,4\\
\dd a_5= 1+b_6b_5&&\dd c_i=0,\mbox{ for } i=1,2.\\
\dd a_6= 1+ b_5b_6&&
\end{array}
$$
There are two augmentations $\e_0$ and $\e_1$ of $\alg(\Lambda_{9_{48}})$ to $\Z_2$ with $\e_i(b_2)=i$, and  $\e_i(b_j)=1$ for $j\neq 2$, $i=0,1$.
These two augmentations are DGA homotopic since $\e_0-\e_1=K\circ \dd$, where $K$ sends $c_1$ to $1$ and the other Reeb chords to $0$. Since $\F= \Z_2$, 
By Proposition~\ref{prop:equivalences} and Remark~\ref{rem:equiv-DGA-homotopy}, equivalence with respect to DGA homotopy is the same
as equivalent with respect to $\Sim$, and thus we have that $|Aug(\Lambda_{9_{48}}; \Z_2)/\Sim|=1$.

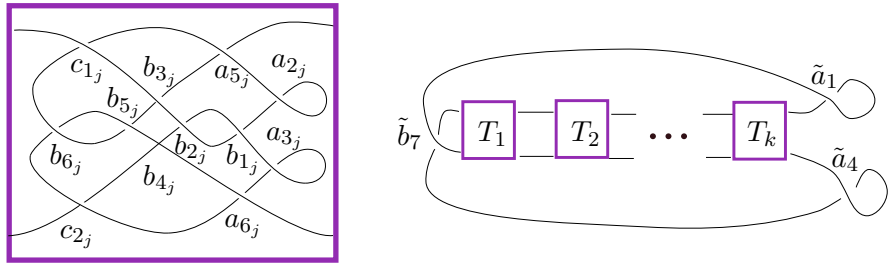
\begin{figure}[!ht]
	\labellist
	\small
	\pinlabel $T_1$ at  540 125
	\pinlabel $T_2$ at  630 125
	\pinlabel $T_k$ at  800 125
	\pinlabel $a_{5_j}$ at 290 180 
	\pinlabel $a_{2_j}$ at 345 190 
	\pinlabel $a_{3_j}$ at 340 120
	\pinlabel $a_{6_j}$ at 300 35 
	\pinlabel $c_{2_j}$ at 140 25 
	\pinlabel $b_{1_j}$ at 300 100
	\pinlabel $b_{2_j}$ at 250 110
	\pinlabel $b_{3_j}$ at 220 185
	\pinlabel $b_{4_j}$ at 220 80
	\pinlabel $b_{5_j}$ at 185 155
	\pinlabel $b_{6_j}$ at 130 100
	\pinlabel $c_{1_j}$ at 150 185
	\pinlabel $\tilde{b}_{7}$ at 460 125
	\pinlabel $\tilde{a}_{1}$ at 860 180
	\pinlabel $\tilde{a}_{4}$ at 880 100
	\endlabellist
	\includegraphics[width=5in]{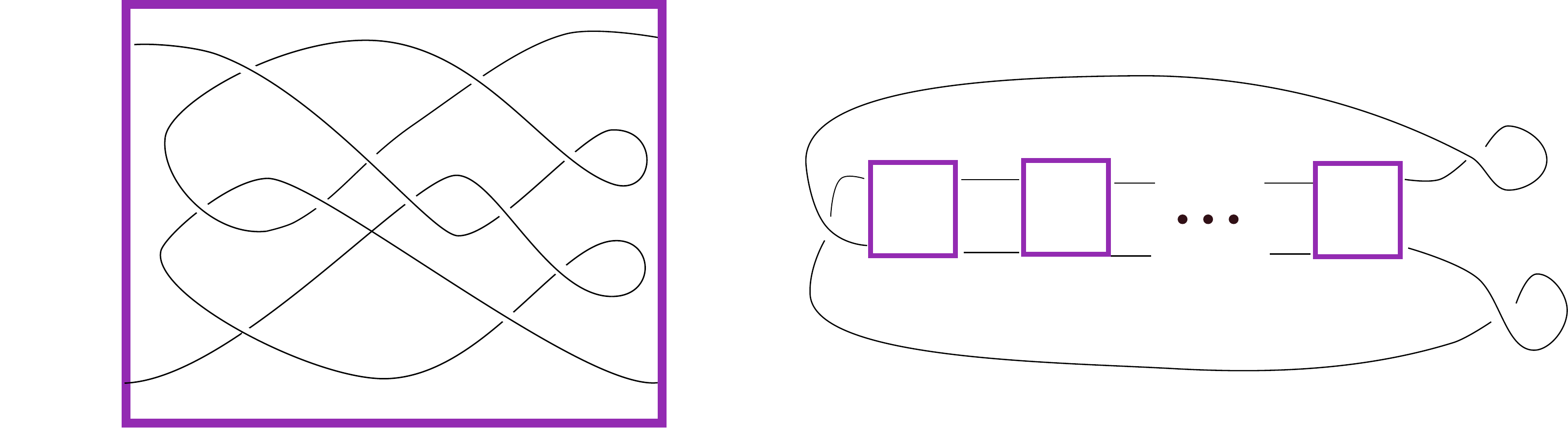}
	\caption{Lagrangian projections of the tangle $T$ and the Legendrian $\leg_{k}$.}
	\label{fig:9_48fam}
\end{figure}
 
The calculation for $k \geq 2$ is similar.
Label the Reeb chords in the $j$th tangle of $\Lambda_k$ by $b_{i_j}, a_{i_j}$ and $c_{i_j}$ following a similar labeling scheme as for $\Lambda_{9_{48}}$, see Figure~\ref{fig:9_48fam}. Let $\tilde{b}_7$, $\tilde{a}_1$ and $\tilde{a}_4$ denote the Reeb chords of $\Lambda_k$ not contained in any of the $k$ tangles, and such that $|\tilde{b}_7|=0$ and $|\tilde{a}_1|=|\tilde{a}_4|=1$. Then, one can find that any augmentation $\e$ of $\Lambda_k$ to $\Z_2$ takes the following values: $\e(\tilde{b}_7)=1,~\e(b_{i_j})=1$ for any $i_j\neq 2_j$, and $\e(b_{2_j})\in \{0,1\}$. Therefore, for any $\Lambda_k$ we have  $2^k$ augmentations to $\Z_2$. Suppose that $\e_1$ and $\e_2$ are two augmentations of  $\Lambda_k$ such that $\e_1(b_{2_j})-\e_2(b_{2_j})=1$ for $j$ contained in some subset $J \subset \{1,\ldots, k\}$. Then, there exists a DGA homotopy $K$ from $\e_1$ to $\e_2$ where $K(c_{1_j})=1$ for $j\in J$, and which maps all other Reeb chords to $0$. Therefore, the Legendrians $\Lambda_k$ have a unique augmentation to $\Z_2$ up to DGA homotopy and thus up to $\Sim$.
\end{proof}

\begin{proof}[Proof of Proposition~\ref{prop:948-disk-obstruct}] By Theorem~\ref{thm:maingen1}, the existence of the filling $F_{k-1}^{k+1}$ is equivalent to
the existence of an embedded, Maslov-$0$, exact Lagrangian cobordism from $\sqcup_{k} \leg_{\Ho}^{1} \sqcup \leg_{\Ho}^{0}$ to $\leg_{k}$.  
By Lemma~\ref{lem:hopf-aug-counts} and Lemma~\ref{lem:948-augs},
$$ |Aug( \sqcup_{k} \leg_{\Ho}^{1} \sqcup \leg_{\Ho}^{0}  ; \Z_2)/\Sim|=3,\quad\text{and} \quad  |Aug(\leg_{k}; \Z_2)/\Sim|=1,$$
and thus by Theorem~\ref{thm:main} such an embedded cobordism from $\sqcup_{k} \leg_{\Ho}^{1} \sqcup \leg_{\Ho}^{0}$ to $\leg_{k}$ does not exist.
 \end{proof}

 We now have all the ingredients to prove our first part of Theorem~\ref{thm:obstruct}.
 
 \begin{proof}[Proof of Theorem~\ref{thm:obstruct}(1)]  Fix $\leg_{k}$. Proposition~\ref{prop:948-gen1-fill} shows the existence of the immersed, Maslov-$0$, exact Lagrangian filling $F_{k}^{k}$ with genus $k$ that has $k$ double points, each with action $0$ and index $1$.
 Proposition~\ref{prop:948-disk-obstruct} shows there does not exist  an immersed, Maslov-$0$, exact Lagrangian disk filling $F_{k-1}^{k+1}$ with $(k+1)$ double points, all of action $0$, $k$ of index $1$, and one
 of index $0$. Thus, 
 by Definition~\ref{defn:not-from-surgery}, $F_{k}^{k}$  does not arise from Lagrangian surgery.   
 
 For the smooth comparison,  $\leg_{1} = \leg_{9_{48}}$ admits a smooth disk filling with one immersed point~\cite[Section 4.6]{OS}, and thus it also admits a disk filling with $p$ immersed points, for any $p\geq1$.  One can more easily see that by two ``unclasping'' moves, ${9_{48}}$ has an unknotting number of $2$: it follows that there exists a smooth disk filling of $\leg_{9_{48}}$ with 2 double points.  Similarly, when $k \geq 2$, by performing unclasping moves in each of the $k$ tangles,  we see that $\leg_{k}$ admits a smooth disk filling with $2k$ double points, and thus by smooth surgery a smooth genus $j$ filling with $2k-j$ double points for all $0 \leq j \leq k$.
 \end{proof}

\subsection{Proof of Theorem~\ref{thm:obstruct}(2)} For all $g \geq 1$ and $p \geq 0$, we will  show the existence of a Legendrian knot $\leg_{g}^{p}$ that has an immersed, Maslov-$0$, exact Lagrangian filling $F_{g}^{p}$, which has genus $g$
 and $p$ double points of action and index $0$, that does not arise from Lagrangian surgery. 
The construction of  $\leg_{g}^{p}$ is an example of the Mondrian diagrams of \cite{Ng}.

 To construct the Legendrian {\bf checkerboard knot} $\leg_{g}^{0}$, $g \geq 1$, begin with a $(2g+2) \times 4$ shaded checkerboard,  
 with the lower left square shaded.  For every shaded square,
 replace the right (resp. left) edge with a right (resp. left) cusp. If two shaded squares  
 share a vertex, replace the vertex with a crossing, and otherwise replace the vertex with a smoothing of the vertex.  An example is given in  Figure~\ref{fig:Checkerboard_construction_3}. We can directly check that $\leg_{g}^{0}$ has a single component, for all $g \geq 1$.
 
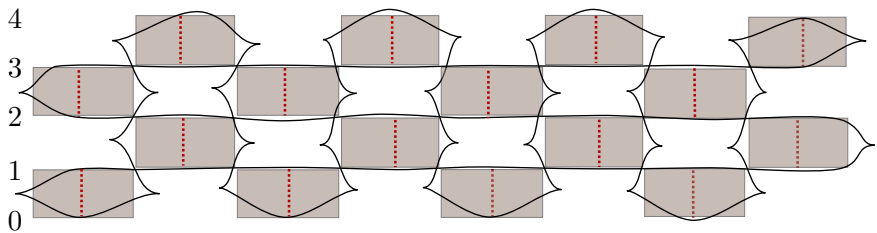
\begin{figure}[!ht]
	\labellist
	\small
	\pinlabel $4$ at 0 112
	\pinlabel $3$ at 0 85
	\pinlabel $2$ at 0 58
	\pinlabel $1$ at 0 28
	\pinlabel $0$ at 0 0
	\endlabellist
	\includegraphics[width=4.5in]{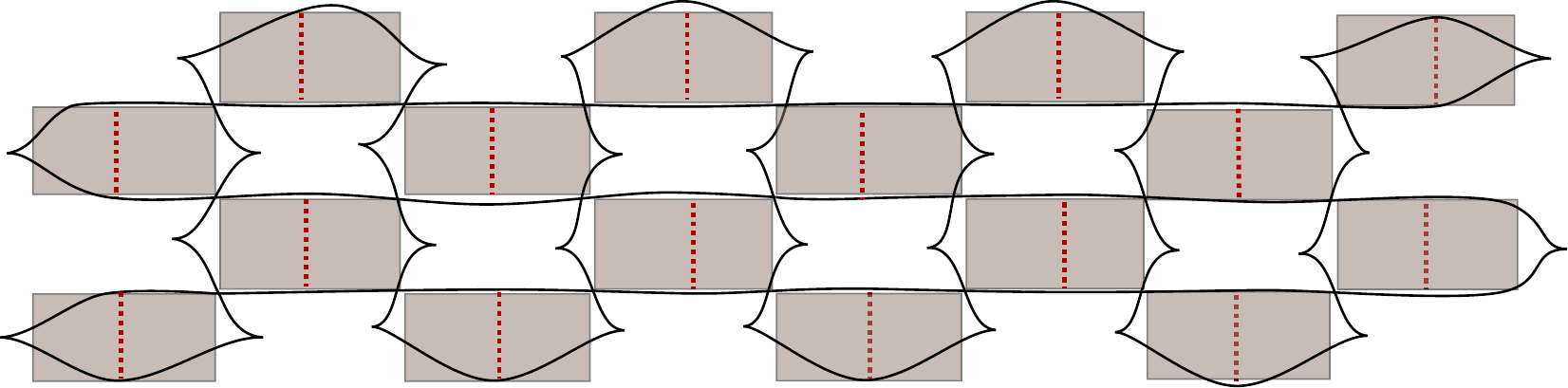}
	\caption{ The Legendrian $\leg_{\bf 3}^{0}$ constructed by starting with a $(2({\bf 3}) + 2) \times 4$ shaded checkerboard; the red lines denote the pinches used in the construction of $F_{3}^{0}$.}
	\label{fig:Checkerboard_construction_3}
\end{figure}

 For $p \geq 1$, the Legendrian knot $\leg_{g}^{p}$ will be constructed by  applying  Legendrian Reidemeister I moves and  adding  $p$ \textit{clasps} to $\leg_{g}^{0}$, as shown in Figure~\ref{fig:claspbis}.
To form $\leg_{g}^{1}$, for $g \geq 1$, start with  the two shaded regions corresponding to the bottom row, first and third columns in the shaded $(2g+2) \times 4$-checkerboard used to construct  $\leg_{g}^{0}$.
Perform one downward Reidemeister I move on each portion of $\leg_{g}^{0}$ corresponding to these two shaded regions, and clasp the pair of cusps facing each other as schematized on Figure \ref{fig:claspbis}.
We form $\leg_{g}^{2}$ by again starting with the two bottom left shaded regions of $\leg_{g}^{0}$, performing {$6$} 
Reidemeister I moves, and then forming 2 clasps in the  shaded tiles of the plane.
Similarly, for all $p \geq 1$, we can form the {\bf clasped checkerboard} Legendrian $\leg_{g}^{p}$, by starting with $\leg_{g}^{0}$, performing {$4p-2$}  
 Reidemeister moves, and adding $p$ clasps, as shown in Figure~\ref{fig:claspbis}.  Observe that $\leg_{g}^{p}$ has a single component.

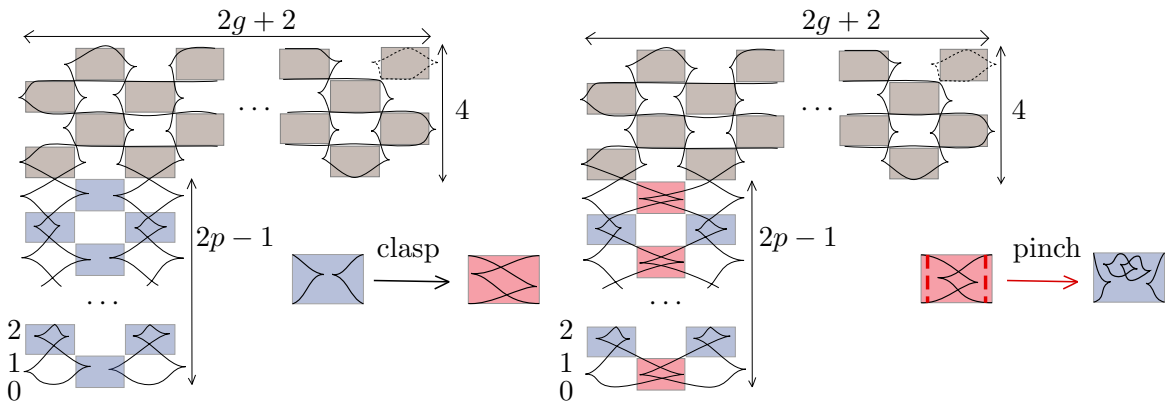
\begin{figure}[!ht]
\labellist
\small
\pinlabel $\cdots$ at 40 40
\pinlabel $\cdots$ at 300 40
\pinlabel $2$ at -2 28
\pinlabel $1$ at -2 13
\pinlabel $0$ at -2 0
\pinlabel $2p-1$ at 100 70
\pinlabel $\text{clasp}$ at 180 65
\pinlabel $4$ at 205 130
\pinlabel $2p-1$ at 360 70
\pinlabel $\text{pinch}$ at 475 65
\pinlabel $2g+2$ at 110 170
\pinlabel $\cdots$ at 110 130
\pinlabel $2g+2$ at 380 170
\pinlabel $\cdots$ at 370 130
\pinlabel $4$ at 462 130
\pinlabel $2$ at 253 28
\pinlabel $1$ at 253 13
\pinlabel $0$ at 253 0
\endlabellist
\includegraphics[width=6in]{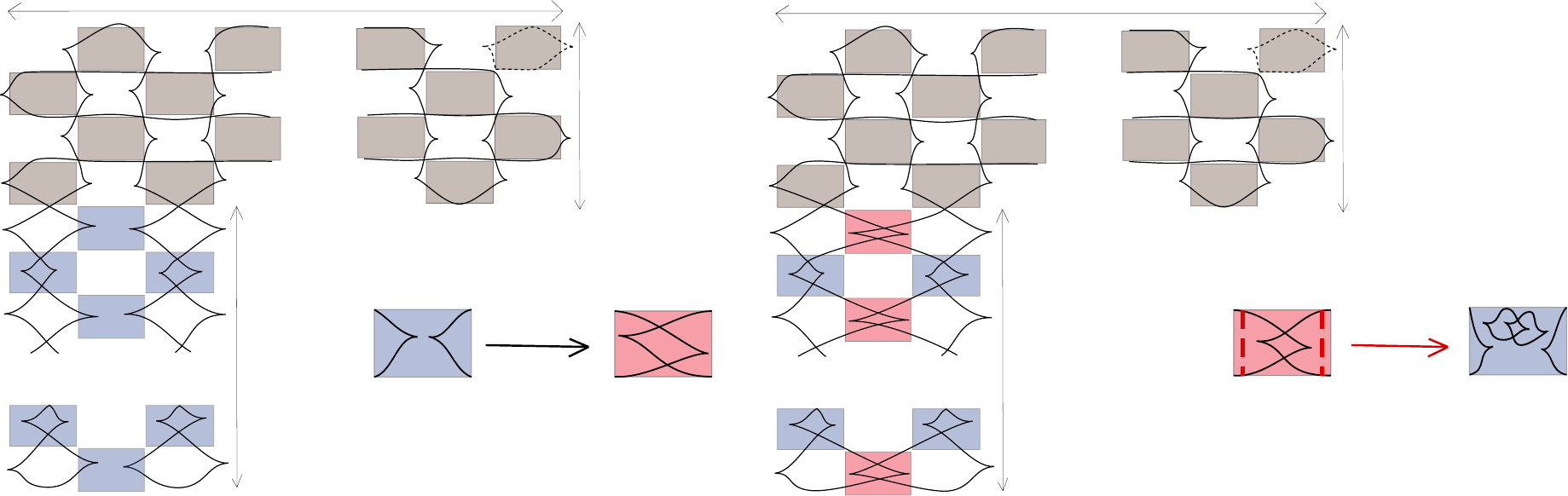}
\caption{Construction of the clasped checkerboard Legendrian $\leg_{g}^{p}$, $p \geq 1$, and 
a schematization of pinch moves around each clasp that shows the existence of a cobordism from $\sqcup_p \Lambda_{\Ho}^0 \cup \leg_{g}^{0}$ to $\leg_{g}^{p}$.}
\label{fig:claspbis}
\end{figure}

\begin{proposition} \label{prop:leg-g-p-fill}   For all $g\geq 1$ and $p \geq 0$, the Legendrian knot $\leg_{g}^{p}$ admits an immersed, Maslov-$0$, exact Lagrangian filling $F_{g}^{p}$ of genus $g$ with $p$ double points, each of which has action $0$ and index $0$.
\end{proposition}

\begin{proof} First fix $\leg_{g}^{0}$, for some $g \geq 1$. By performing  pinch moves on each pair of strands that correspond to the top and bottom edges of each shaded square in the $(2g+2) \times 4$ shaded checkerboard that was used to constuct $\leg_{g}^{0}$,   we obtain an embedded, exact, Lagrangian cobordism from a disjoint union of max $tb$ Legendrian unknots  to $\leg_{g}^{0}$;  see an illustration in Figure~\ref{fig:Checkerboard_construction_3}.    The Maslov potential on the strands on which we perform the pinch moves ensures that this cobordism has Maslov class $0$. Each Legendrian unknot can be filled with a disk to obtain $F_{g}^{0}$, an embedded, Maslov-$0$, exact Lagrangian filling of $\leg_{g}^{0}$. As we perform $\frac12(2g + 2)4$ pinch moves and obtain $4+(2g+1)$ unknots, we see that this filling does indeed have genus $g$, as desired.

Now fix $\leg_{g}^{p}$, for $p \geq 1$. By performing pinch moves along the red dash lines besides the clasps as schematized in Figure~\ref{fig:claspbis}, we build
a genus $0$ embedded, Maslov-$0$, exact Lagrangian cobordism from $\sqcup_p \Lambda_{\Ho}^0 \cup \leg_{g}^{0}$ to $\leg_{g}^{p}$.
The $\leg_{g}^{0}$ has an embedded, Maslov-$0$, exact Lagrangian filling of genus $g$, while  
each Hopf link $\Lambda_{\Ho}^0$ can be filled by an immersed, Maslov-$0$, 
exact Lagrangian filling with one double point of action $0$ and index $0$.  By stacking this Lagrangian cobordism and these fillings, we obtain the desired $F_{g}^{p}$.
\end{proof}

\begin{proposition}\label{prop:leg-g-p-obstruct} The Legendrian knot $\leg_{g}^{p}$ does not admit an immersed, Maslov-$0$, exact Lagrangian disk filling $F_{g-1}^{p+1}$ with $p+1$ double points, all of action $0$ and index $0$.   
\end{proposition}

The proof follows easily from the following calculation, which will be proved in Section~\ref{ssec:g-p-augs}.

\begin{lemma}\label{lem:leg-g-p-augs} $|Aug(\leg_{g}^{p}; \Z_2)/\Sim|=3^{p}$.
\end{lemma}

\begin{proof}[Proof of Proposition~\ref{prop:leg-g-p-obstruct}] By Theorem~\ref{thm:maingen1}, the existence of the filling $F_{g-1}^{p+1}$ is equivalent to
the existence of an embedded, Maslov-$0$, exact Lagrangian cobordism of genus $g$ from $\sqcup_{p+1} \leg_{\Ho}^{0}$ to $\leg_{g}^{p}$.  
By Lemma~\ref{lem:hopf-aug-counts} and Lemma~\ref{lem:leg-g-p-augs},
$$ |Aug( \sqcup_{p+1} \leg_{\Ho}^{0}  ; \Z_2)/\Sim|=3^{p+1},\quad\text{and} \quad  |Aug(\leg_{g}^p; \Z_2)/\Sim|=3^{p},$$
and thus by Theorem~\ref{thm:main} such an embedded cobordism from $\sqcup_{p+1} \leg_{\Ho}^{0} $ to $\leg_{k}$ does not exist.
\end{proof}

\begin{proof}[Proof of Theorem~\ref{thm:obstruct}(2)]  Proposition~\ref{prop:leg-g-p-fill} shows the existence of the immersed, Maslov-$0$, exact Lagrangian genus $g$ filling $F_{g}^{p}$ of $\leg_{g}^{p}$ that has $p$ double points of action $0$ and index $0$.
 Proposition~\ref{prop:leg-g-p-obstruct} shows there does not exist  an immersed, Maslov-$0$, exact Lagrangian genus $(g-1)$ filling $F_{g-1}^{p+1}$ of $\leg_{g}^{p}$ with $(p+1)$ double points, all of action $0$ and index $0$. Thus,   by Definition~\ref{defn:not-from-surgery}, $F_{g}^{p}$  does not arise from Lagrangian surgery.   
 \end{proof}

It remains to prove Lemma~\ref{lem:leg-g-p-augs}, which we do in the next subsection. 

\subsection{Proof of Lemma~\ref{lem:leg-g-p-augs}} \label{ssec:g-p-augs}

 As opposed to the more direct counting strategy we employed in Lemma~\ref{lem:948-augs},  
here we count augmentations of these arbitrarily high crossing knots $\leg_{g}^{p}$ using the theory of rulings.  So
we begin with some background on rulings, first defined in  \cite{CP, Fuchs},  and review the definition of the ruling polynomial.  

Following \cite{Sab20},
a (graded, normal) {\bf ruling} of a Legendrian knot $\Lambda$ is a set of crossings (called {\bf switches}) such that resolving the switches yields a link of unknots $\leg_{1}, \dots, \leg_{m}$ such that
\begin{enumerate}
\item At each switch, the two strands have the same Maslov potential;
\item Each  $\leg_{i}$, $i = 1, \dots, m$, is a Legendrian unknot with $0$ crossings and $2$ cusps that bounds a ruling disk $D_{i}$; 
\item Exactly two components are incident to any switch;
\item Near each switch, the incident ruling disks $D_{i}$ bounded by $\leg_{i}$ are either nested or disjoint as shown in Figure~\ref{fig:rul}.
\end{enumerate}
 For each ruling $R$ of a Legendrian $\Lambda$, denote the number of switches and disks by $s(R)$ and $d(R)$, respectively.
For a Legendrian $\Lambda$, the {\bf ruling polynomial} $R_{\Lambda}(z)$ is the polynomial
\begin{equation} \label{eqn:ruling-poly} R_{\Lambda}(z)= \displaystyle{\sum_R z^{s(R)-d(R)}}.
\end{equation}
\begin{figure}[!ht]
\includegraphics[width=3in]{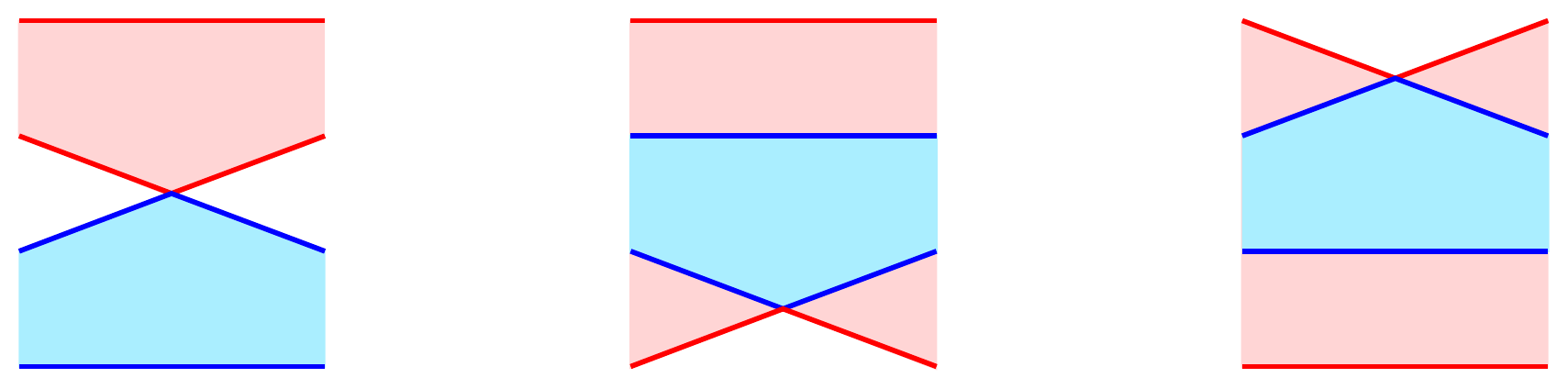}
\caption{Around a switch, the  possible nested or disjoint configurations of the incident disks.}
\label{fig:rul}
\end{figure}

Rulings and augmentations are related: the existence of one implies the existence of the other \cite{Fuchs, FI, Sab}. The following lemma shows how we can use the ruling
polynomial to find our desired count of augmentations.
\begin{lemma}\label{lem:rel}
Let $\Lambda$ be a Legendrian knot with no   negative degree Reeb chords.
Then 
\begin{equation}\label{eq:aug}
 |Aug(\Lambda; \Z_2)/\Sim| = 2^{\chi(\Lambda)/2} R_{\Lambda}(2^{-1/2}),
 \end{equation}
 for $\chi(\Lambda)=\displaystyle{\sum_{k\geq 0} (-1)^ka_k},$ where $a_{k}$ is the number of Reeb chords of degree $k$.
\end{lemma}

\begin{proof}
Under the assumption of the lemma, following \cite[Remark 3.3(ii)]{HR} the number of  augmentations $| Aug(\leg; \Z_{2}) |$ is related to the ruling polynomial $R_{\Lambda}(z)$ in the following way:
$$2^{-\chi(\Lambda)/2} | Aug(\leg; \Z_{2})  |=R_{\Lambda}(2^{1/2}-2^{-1/2}),$$
where $ \chi(\Lambda)$ is the Euler characteristic of $(\mathcal{A}(\Lambda),\partial)$, defined as $ \chi(\Lambda)=\sum_{k\geq0} (-1)^ka_k,$
where $a_k$ is the number of generators of $\mathcal{A}(\Lambda)$ of degree $k$.
By Corollary~\ref{cor:distinct-not-equiv}, we know that {$| Aug(\leg; \Z_{2}) | = |Aug(\Lambda;\Z_2)/\Sim|$}, and our result follows.
\end{proof}

\begin{proof}[Proof of Lemma~\ref{lem:leg-g-p-augs}]

First consider the checkerboard Legendrian knot $\leg_{g}^{0}$.    Observe that  $\leg_{g}^{0}$ 
has a unique ruling that switches at every crossing: one can check this by considering the shaded regions in the top row: the left cusp of a shaded square has to match with the right cusp of the shaded square, thus it forces the crossings in the  line below to be switches. An inductive argument then proves that each crossing must be a switch.  In this unique ruling $R$, using Equations~\eqref{eqn:CZ} and \eqref{eqn:CZ-grade}, we see that all Reeb chords in $\mathcal A(\leg_{g}^0)$

 have degree $0$ or $1$.  Furthermore, $s(R)-d(R)=\chi(\Lambda_{g}^0)$
since there is a switch at each degree $0$ chord and a one-to-one correspondence between disks and right cusps (which correspond to Reeb chords of degree $1$).  Then applying Lemma~\ref{lem:rel}, we find
$$ |Aug(\leg_{g}^{0}; \Z_2)/\Sim| = 2^{\chi(\leg_{g}^{0})/2} R_{\leg_{g}^{0}}(2^{-1/2}) = 2^{\chi(\leg_{g}^{0})/2} (2^{-1/2})^{\chi(\leg_{g}^{0})} = 1.$$

Now consider $\leg_{g}^{p}$, for $p \geq 1$. 
As shown in Figure~\ref{fig:claspbis}, the first clasp in the construction  $\leg_{g}^{p}$ introduces four new degree $0$ Reeb chords (two from the Reidemeister moves, two in the clasp region), and two rulings.
With just one clasp, the ruling polynomial changes from $R_{ \leg_{g}^{0}}(z)=z^{\chi\left(\leg_{g}^{0}\right)}$ to 
$$R_{\leg_{g}^{1} }(z)=z^{\chi\left(\leg_{g}^{1} \right)}(z^{-2}+1).$$
Each additional  clasp introduces $6$ new degree $0$ Reeb chords ($4$ from Reidemeister moves, $2$ in the clasp region). 
Considering rulings, each new chord coming from a Reidemeister move must be a switch and then one can either switch at both or neither of the two crossings in the clasp region. 
Thus, the ruling polynomial becomes 
$$R_{\leg_{g}^{p}}(z)=z^{\chi\left( \leg_{g}^{p} \right)}(z^{-2}+1)^p.$$
Using Equation~\eqref{eq:aug}, we find that the number of augmentations of $\leg_{g}^{p}$ to $\Z_2$ is $3^p$:
$$|Aug(\leg_{g}^{p}; \Z_2)/\Sim| = 2^{\chi\left(\leg_{g}^{p} \right)/2} (2^{-1/2})^{\chi\left( \leg_{g}^{p} \right)}((2^{-1/2})^{-2}+1)^p  = 3^{p}.   $$
\end{proof}

\appendix
\section{Equivalence in $Aug_+(\Lambda)$ for Legendrian links}\label{appendix:equivalences}

In this appendix we will provide the proof of Proposition~\ref{prop:equivalences} following the proof for the case of single component knots in~\cite{NRSSZ}]Proposition~5.19]. We start by setting some basic notation. Let $\Lambda=\cup_{k=1}^m \Lambda_k$ be a Legendrian link with $m$ link components. For a mixed Reeb chord $a$ that starts on an ith link component $\Lambda_i$ and ends on the jth link component $\Lambda_j$, that is $a\in R(\Lambda_j, \Lambda_i)$, we let $c(a)=i$ and $r(a)=j$.

 Let $\Lambda^n_f$ denote the $n$-copy of $\Lambda$ that has been perturbed by a Morse function $f$ with a single maximum and minimum as in~\cite{NRSSZ}. Note that if $\Lambda= \cup_{k=1}^m \Lambda_k$ is a link with $m$ link components, then $\Lambda^n_f=\cup_{i=1}^n \cup_{k=1}^m \Lambda_{k}^i$ is a link with $mn$ link components. Given a Legendrian link $\Lambda$, and its perturbed two copy $\Lambda_f^2=\Lambda^1\cup \Lambda^2$, 
for any Reeb chord $a\in R(\Lambda^1, \Lambda^2)$, there is a corresponding element $\check{a}\in Hom_+(\e^1, \e^2)$ with degree $|\check{a}|_+=|a|+1$.
Observe that this is a different notation convention than what we use in Section~\ref{sec:augcat}. 
 
Let $(\e^1, \ldots, \e^{n+1})$ be a tuple of augmentations of $(\alg(\Lambda), \partial)$. Define $((\alg(\Lambda^{n+1}_f))^{\e}, \partial_{\e}^{n+1})$ as follows. Let $\alg(\Lambda^{n+1}_f))^{\e}:=(\alg(\Lambda^{n+1}_f)\otimes \F)/(t_k=\e(t_k))$ and set $\partial^{n+1}_{\e}=\phi_{\e}\circ \partial^{n+1} \circ \phi^{-1}_{\e}$, where $\phi_{\e}(a)=a+\e(a)$. Then, the composition maps
$$m_n: Hom_+(\e^n, \e^{n+1})\otimes \cdots \otimes Hom_+(\e^2, \e^3) \otimes Hom_+(\e^1, \e^2)\rightarrow Hom_+(\e^1, \e^{n+1}),$$
are given by
$$m_n(\check{\alpha}_n, \ldots, \check{\alpha}_1)=(-1)^{\sigma}\sum_{a\in \mathcal{R}\cup x_k\cup y_k} \check{a} \cdot \text{Coeff}_{\alpha_1^{12}\alpha^{23}_2\cdots \alpha^{n,n+1}_n}(\partial_{\e}^{n+1}a^{1,n+1})$$
where $\alpha_i\in \{a_1, \ldots, a_r, x_1, \ldots, x_m,y_1, \ldots, y_m\}$ for each $i$, and $\sigma=n(n-1)/2+\sum_{p<q}|\check{\alpha}_p|_+|\check{\alpha}_q|_++|\check{\alpha}_{n-1}|_++|\check{\alpha}_{n-3}|_++\cdots$.

\begin{proposition}[{\cite[Proposition $4.14$]{NRSSZ}}]~\label{prop4.14} Let $\Lambda \subset \R^3_{std}$ be a Legendrian link with $m$ link components, one basepoint $t_k$ per link component and Reeb chords $\mathcal{R}=\{a_1, \ldots, a_r\}$. The DGA of the perturbed $n$-copy of $\Lambda$, $\Lambda^n_f$, is generated by
\begin{enumerate}
\item $(t^i_k)^{\pm1}$ for $1\leq i\leq n, 1\leq k \leq m$, with $|t_k^{i}|=0;$
\item $a_h^{ij}$ for $1\leq i,j\leq n$, and $1\leq h\leq r$ with $|a_h^{ij}|=|a_h|$;
\item $x^{ij}_k$ for $1\leq i,j\leq n$, and $1\leq k \leq m$, with $|x^{ij}_k|=0$;
\item $y^{ij}_k$ for $1\leq i,j\leq n$, $1\leq k \leq m$, with $|y^{ij}|=-1$,
\end{enumerate}
and satisfies the relations $t^i_k(t^i_k)^{-1}=(t^i_k)^{-1}t^i_k=1$ for each $i$ and $k$. The differential of 
$\mathcal{A}(\Lambda^n_f, \partial^n)$
 can be described as follows. Assemble the generators of  $\mathcal{A}(\Lambda^n_f, \partial^n)$
 into $n\times n$ matrices: $A_h=(a_h^{ij})$, $\Delta_k=Diag(t_k^1, \ldots, t_k^n)$, 
$$
X_k=\begin{pmatrix}
1&x_k^{12}&\cdots &x_k^{1n}\\
0&1&\cdots &x_k^{2n}\\
\vdots & \vdots &\ddots & \vdots\\
0&0&\cdots &1
\end{pmatrix}, ~\text{and}~ Y_k=\begin{pmatrix}
0&y_k^{12}&\cdots& y_k^{1n}\\
0&0&\cdots &y_k^{2n}\\
\vdots & \vdots &\ddots & \vdots\\
0&0&\cdots &0
\end{pmatrix} $$
where  $1\leq k \leq m$, and $1\leq h\leq r$. Then, applying $\partial^n$ to matrices entry-by-entry, we have 
\begin{align*}
\partial^n(A_h)&=\Phi(\partial(a_h))+Y_{r(a_h)}A_h+(-1)^{|a_h|+1} A_h Y_{c(a_h)}\\
\partial^n(X_k)&=\Delta_k^{-1}Y_{k}\Delta_kX_k-X_kY_{k}\\
\partial^n(Y_k)&=Y^2_k
\end{align*}
where $\Phi: \alg(\Lambda)\rightarrow Mat(M, \alg^n)$ is a ring of homomorphism such that $\Phi(a_h)=A_h, \Phi(t_k)=\Delta_k X_k, \Phi(t^{-1}_k)=X^{-1}_k\Delta_k$.
\end{proposition}
The following Lemma~\ref{lem:products}, and Proposition~\ref{prop:alpha} are generalizations of Lemma $5.16$, Proposition $5.17$ and Proposition $5.18$ in \cite{NRSSZ}. Lemma~\ref{lem:products} is an immediate consequence of Proposition~\ref{prop4.14} which allows us to compute $m_1$ and $m_2$ from $(\alg(\Lambda^2_f), \partial^2)$ and  $(\alg(\Lambda^3_f), \partial^3)$.
\begin{lemma}~\label{lem:products}Let $\Lambda\subset \R^3_{std}$ be an $m$ component Legendrian link with one basepoint $t_k$ per link component, Reeb chords $\mathcal{R}=\{a_1, \ldots, a_r\}$ and augmentations $\e^1, \e^2$. In $Hom_+(\e^1, \e^2),$ we have that \begin{align*}
m_1(\check{a}_h)&=\sum_{1\leq l\leq n}\delta_{b_l,a_h} \sigma_u|(\overline{\mathcal{M}}_J^{\R \times \Lambda}(a_{h'}; b_1, \ldots, b_n)|  \e^1(b_1\cdots b_{l-1})\e^2(b_{l+1}\cdots b_n)\check{a}_{h'}\\
m_1(\check{y}_k)&=(\e^1(t_k)^{-1}\e^2(t_k)-1)\check{x}_k+\sum_{a_h\in \{a\in \mathcal{R}| c(a)=k\}}\e^2(a_h)\check{a}_h\\
&+\sum_{a\in \{a\in \mathcal{R}| r(a)=k\}}(-1)^{|a_{h}|+1}\e^1(a_{h})\check{a}_{h}\\
m_1(\check{x}_k) &\in span_{\F}\{\check{a}_1\ldots, \check{a}_r\}~\text{for}~k\in \{1, \ldots, N\}.\end{align*}
where 
$b_i\in span_{\F}\{a_1, \ldots, a_r, t_1, \ldots, t_m\}$, and $\sigma_u\in\{\pm1\}$ 
 denotes the product of all orientation signs at the corners of the disk $u$. 
 We also have that for $i,j\in \{1, \ldots, m\}$, and $1\leq h, h'\leq r$,
\begin{align*}
m_2(\check{y}_i,\check{y}_j)&=\begin{cases}-\check{y}_i &\text{if}~\, i=j\\0 &\text{if}~\, i\neq j\\ \end{cases}&
m_2(\check{x}_i, \check{y}_j)&=\begin{cases}-\e^1(t_i)^{-1}\e^2(t_i)\check{x}_{i} &\text{if}~\, i=j\\0 &\text{if}~\, i\neq j\\ \end{cases}\\
m_2(\check{y}_i, \check{x}_j)&=\begin{cases} -\check{x}_i &\text{if}~\, i=j\\0 &\text{if}~\, i\neq j\\ \end{cases}&
m_2(\check{y}_i, \check{a}_{h})&=\begin{cases}-\check{a}_h &\text{if}~\, c(a_h)=i\\0 &\text{if}~\, c(a_h)\neq i\\ \end{cases}\\
m_2(\check{a}_{h},\check{y}_i)&=\begin{cases}-\check{a}_h &\text{if}~\, r(a_h)=i\\0 &\text{if}~\, r(a_h)\neq i\\ \end{cases}& &
\end{align*}

$m_2(\check{a}_h, \check{a}_{h'}), m_2(\check{x}_i, \check{x}_j), m_2(\check{x}_i, \check{a}_h), m_2(\check{a}_h, \check{x}_j)\in span_{\F}\{\check{a}_1, \ldots, \check{a}_r\}$.  Moreover, If we assume that the Reeb chords of $\Lambda$ are labeled by increasing height, $h(a_1)\leq h(a_2) \leq \cdots \leq h(a_2)$, then $m_2(\check{a}_h, \check{a}_{h'})\in span_{\F}\{ \check{a}_l~|~l\geq max(h,h'), 1\leq h,h'\leq r\}.$
\end{lemma}
\begin{proposition}~\label{prop:alpha} Consider an element $\alpha\in Hom^0_+(\e^1, \e^2)$ of the form $\alpha=-\sum_{k=1}^m c_k \check{y}_k -\sum_{h=1}^r K(a_h) \check{a}_h$

where $K:(\alg(\Lambda), \partial) \rightarrow (\F, 0)$ is an $\F$ linear map. Then, $m_1(\alpha)=0$ if and only if $K$ is a split DGA homotopy from $\e^1$ to $\e^2$.
\end{proposition}
\begin{proof}
First, observe that\begin{align*}
&\sum_{h=1}^r K(a_h) m_1(\check{a}_h)\\&=\sum_{h=1}^r K(a_h) [\sum_{1\leq l\leq n}\delta_{b_l,a_h} |(\overline{\mathcal{M}}_J^{\R \times \Lambda}(a_{h'}; b_1, \ldots, b_n)|  \e^1(b_1\cdots b_{l-1})\e^2(b_{l+1}\cdots b_n)\check{a}_{h'}]\\
&= \sum_{h=1}^r[\sum_{1\leq l\leq n}\delta_{b_l,a_h} (-1)^{|b_1\cdots b_{l-1}|}|(\overline{\mathcal{M}}_J^{\R \times \Lambda}(a_{h'}; b_1, \ldots, b_n)|  \e^1(b_1\cdots b_{l-1}) K(b_l)\e^2(b_{l+1}\cdots b_n)]\check{a}_{h'}\\
&= \sum_{h=1}^r K\circ \partial(a_h) \check{a}_h
\end{align*}
where the last equality follows from the fact that $K(t_k^{\pm})=0$ for all $k\in \{1, \ldots, m\}$. Therefore, using Lemma~\ref{lem:products}, we know that 
\begin{align*}
-m_1(\alpha)&= m_1(\sum_{k=1}^m c_k \check{y}_k + \sum_{h=1}^r K(a_h) \check{a}_h)\\
&=\sum_{k=1}^m c_k m_1(\check{y}_k) + \sum_{h=1}^r(K\circ \partial(a_h))\check{a}_h\\
&= \sum_{k=1}^m c_k(\e^1(t_k)^{-1}\e^2(t_k)-1)\check{x}_k+ \sum_{h=1}^r [c_{c(a_h)} \e^2(a_h)+(-1)^{|a_h|+1}c_{r(a_h)}\e^1(a_h)]\check{a}_h\\&+\sum_{h=1}^r(K\circ \partial(a_h))\check{a}_h\\
\end{align*}

Thus, $m_1(\alpha)=0$ if and only if $K\circ \partial(a_h)= -c_{c(a_h)} \e^2(a_h)-(-1)^{|a_h|+1}c_{r(a_h)}\e^1(a_h)$
 for all $h\in\{1, \ldots, r\}$, and $(\e^1(t_k)^{-1}\e^2(t_k)-1)=0$ for all $k\in \{1, \ldots, m\}$. Note that $\F$ is supported in grading $0$, and therefore $\e^1(a_h)=(-1)^{|a_h|}\e^1(a_h)$ for all $h$ since $\e^1$ is supported in grading $0$. If $\mathcal{A}(\Lambda)$ has a $\Z_n$ grading, and $\e^1$ is an $n$-graded augmentation, recall that the grading is defined mod $n$.
 Therefore,  $K\circ \partial(a_h)=c_{r(a_h)}\e^1(a_h) -c_{c(a_h)} \e^2(a_h)$, and $\e^1$ and $\e^2$ are split DGA homotopic via the operator $K$.
\end{proof}

\begin{proof}[ Proof of Proposition~\ref{prop:equivalences}]
Suppose that $\e^1$ and $\e^2$ are equivalent in $\aug_+(\Lambda)$. Then, as stated in Definition~\ref{defn:equal}, there exist cocycles $\alpha\in Hom_+(\e^1, \e^2)$, and $\beta\in Hom_+(\e^2, \e^1)$ such that
$[m_2(\alpha, \beta)]=-\sum_{k=1}^m[\check{y}_k]\in H^0Hom_+(\e^1, \e^2).$That is, $m_2(\alpha, \beta)+\sum_{k=1}^m\check{y}_k=m_1(\gamma)$ for some $\gamma\in Hom_+(\e^2, \e^2)$. By Lemma~\ref{lem:products} and the fact that $\gamma\in Hom_+(\e^2, \e^2)$, we know that $\langle m_1(\gamma), \check{y}_k\rangle=0$ and $\langle m_1(\gamma), \check{x}_k\rangle=0$. Therefore, $m_1(\gamma)=\sum_{h=1}^r K(a_h)\check{a}_h$ for some $\F$ linear map $K: (\alg(\Lambda), \partial)\rightarrow (\F, 0)$ which is naturally split. We can now write $m_2(\alpha, \beta)= -\sum_{k=1}^m \check{y}_k +\sum_{h=1}^r K(a_h)\check{a}_h.$ Again by Lemma~\ref{lem:products} and the fact that $|\alpha|_+= |\beta|_+=0$, while $|\check{x}_k|_+=1,$ we know that
\begin{align*}
\alpha &= \sum_{k=1}^m (c_{\alpha})_k \check{y}_k+\sum_{h=1}^rK_{\alpha}(a_h)\check{a}_h\\
\beta &= \sum_{k=1}^m (c_{\beta})_k \check{y}_k+\sum_{h=1}^rK_{\beta}(a_h)\check{a}_h
\end{align*}
such that $(c_{\alpha})_k, (c_{\beta})_k\in \F^*$ and $(c_{\alpha})_k(c_{\beta})_k=1$ for each $k$, and for some $\F$ linear maps $K_{\alpha}, K_{\beta}: (\alg(\Lambda), \partial) \rightarrow (\F,0)$. Both $\alpha$ and $\beta$ are cocycles so by Proposition~\ref{prop:alpha}, $K_{\alpha}$ and $K_{\beta}$ are DGA homotopies between $\e^1$ and $\e^2$.

Suppose that $\e^1$ and $\e^2$ are split DGA homotopic, such that for any Reeb chord $a$, 
$$c_{c(a)} \e^1(a)-c_{r(a)}\e^2(a)= K\circ \partial(a)$$
for constants $c_i\in \F^*$ and some split DGA homotopy $K: (\alg(\Lambda), \partial) \rightarrow (\F,0)$. We know that $K(a)=0$ for any Reeb chord $a$ such that $|a|\neq -1$. 
By Lemma~\ref{prop:alpha}, $\alpha= \sum_{k=1}^m (c_{\alpha})_k \check{y}_k +\sum_{h=1}^r K(a_h)\check{a}_h$ is a cocycle in $H^0Hom_+(\e^1, \e^2)$. We now construct cocycles $\beta, \gamma \in Hom_+(\e^1, \e^2)$ such that $|\beta|_+=|\gamma|_+=0$, $m_2(\beta, \alpha)=m_2(\alpha, \gamma)=-\sum_{k=1}^m \check{y}_k$. This implies that  $[\beta]=[\gamma]\in H^0Hom_+(\e^1, \e^2)$ is the multiplicative inverse of $[\alpha]$ in $\aug_+(\Lambda)$. The construction of $\gamma$ is similar to the construction of $\beta$ which we now provide.

Suppose that the Reeb chords $\{\check{a}_1, \ldots, \check{a}_r\}$ are ordered by height. Then we can write $\alpha=\sum_{k=1}^m (c_{\alpha})_k \check{y}_k+A$ where $A\in span_{\F}\{\check{a}_1, \ldots, \check{a}_r\}$. Let $\beta=\sum_{k=1}^m (c_{\beta})_k \check{y}_k+B$ where $(c_{\alpha})_k(c_{\beta})_k=1$ for $1\leq k\leq m$, $B\in span_{\F}\{\check{a}_1, \ldots, \check{a}_r\}$ and is defined inductively to satisfy $B=A+m_2(B,A)$. Then, $m_2(\beta, \alpha)=-\sum_k \check{y}_k$. To verify that $\beta$ is a cocycle note that the $A_{\infty}$ relations on $\aug_{+}(\Lambda)$ imply that 
$$m_1(-\sum_{k=1}^m\check{y}_k)=m_1(m_2(\beta,\alpha)=m_2(m_1(\beta), \alpha)+m_2(\beta, m_1(\alpha)).$$
We know that $m_1(\check{y}_k)=0$ for all $1\leq k\leq m$ and that $m_1(\alpha)=0$ so $m_2(\beta, m_1(\alpha))=0$. Therefore, $m_2(m_1(\beta), \alpha)=0$. 

We will show that if $X\in span_{\F}\{\check{a}_1, \ldots, \check{a}_r, \check{x}_1, \ldots, \check{x}_N, \check{y}_1, \ldots, \check{y}_N\}$, then $m_2(X, \alpha)=0$ implies that $X=0$. Note that $m_2(X, A)\in span_{\F}\{\check{a}_1, \ldots, \check{a}_r, \check{x}_1, \ldots, \check{x}_m, \check{y}_1, \ldots, \check{y}_m\}$ by Lemma~\ref{lem:products}. Then,
$$0=m_2(X, \alpha)=m_2( X, \sum_{k=1}^m (c_{\alpha})_k \check{y}_k+A)=m_2(X, \sum_{k=1}^m (c_{\alpha})_k \check{y}_k)+m_2(X,A) $$
Thus, $m_2(X, \sum_{k=1}^m (c_{\alpha})_k \check{y}_k) = m_2(X,A)$.
Note that $m_2(X, A) \in span_{\F}\{\check{a}_1, \ldots, \check{a}_r\}$ because  $A\in span_{\F}\{\check{a}_1, \ldots, \check{a}_r\}$ by Lemma~\ref{lem:products}.  Therefore, we know that $\langle X, \check{x}_k\rangle=\langle X, \check{y}_k\rangle =0$ for all $1\leq k\leq m$, and so $X\in span_{\F}\{\check{a}_1, \ldots, \check{a}_r\}$. Moreover, by induction on the height of Reeb chords, and Lemma~\ref{lem:products}, we know that $\langle X, \check{a}_h\rangle=0$ for all $1\leq h\leq r$. Thus, for $X=m_1(\beta)\in  span\{\check{a}_1, \ldots, \check{a}_r, \check{x}_1, \ldots, \check{x}_m, \check{y}_1, \ldots, \check{y}_m\}$, since $m_2(m_1(\beta), \alpha)=0$ as shown above, we can conclude that $m_1(\beta)=0$. 
\end{proof}

\bibliographystyle{alpha}
\bibliography{WIG}

\newcommand{\etalchar}[1]{$^{#1}$}
\begin{thebibliography}{CDRGG20}

\bibitem[BC14]{BCh}
Fr\'ed\'eric Bourgeois and Baptiste Chantraine.
\newblock Bilinearized {L}egendrian contact homology and the augmentation
  category.
\newblock {\em J. Symplectic Geom.}, 12(3):553--583, 2014.

\bibitem[Boo69]{Boothby}
William~M. Boothby.
\newblock Transitivity of the automorphisms of certain geometric structures.
\newblock {\em Trans. Amer. Math. Soc.}, 137:93--100, 1969.

\bibitem[Cas22]{Casals_sing}
Roger Casals.
\newblock Lagrangian skeleta and plane curve singularities.
\newblock {\em Journal of Fixed Point Theory and Applications}, 24(2), 2022.

\bibitem[CDRGG]{CDRGG2}
Baptiste Chantraine, Georgios Dimitroglou~Rizell, Paolo Ghiggini, and Roman
  Golovko.
\newblock Geometric generation of the wrapped {F}ukaya category of {W}einstein
  manifolds and sectors.
\newblock arXiv:1712.09126 [math.SG], to appear in Annales Scientifiques de
  l'{\'E}cole Normale Sup{\'e}rieure.

\bibitem[CDRGG15]{CDRGG1}
Baptiste Chantraine, Georgios Dimitroglou~Rizell, Paolo Ghiggini, and Roman
  Golovko.
\newblock Floer homology and {L}agrangian concordance.
\newblock In {\em Proceedings of the {G}\"okova {G}eometry-{T}opology
  {C}onference 2014}, pages 76--113. G\"okova Geometry/Topology Conference
  (GGT), G\"okova, 2015.

\bibitem[CDRGG20]{CDRGG}
Baptiste Chantraine, Georgios Dimitroglou~Rizell, Paolo Ghiggini, and Roman
  Golovko.
\newblock Floer theory for {L}agrangian cobordisms.
\newblock {\em J. Differential Geom.}, 114(3):393--465, 2020.

\bibitem[CE12]{CE}
Kai Cieliebak and Yakov Eliashberg.
\newblock {\em From {S}tein to {W}einstein and back}, volume~59 of {\em
  American Mathematical Society Colloquium Publications}.
\newblock American Mathematical Society, Providence, RI, 2012.
\newblock Symplectic geometry of affine complex manifolds.

\bibitem[CG22]{CG}
Roger Casals and Honghao Gao.
\newblock Infinitely many {L}agrangian fillings.
\newblock {\em Annals of Mathematics}, 195(1), 2022.

\bibitem[Cha10]{Cha}
Baptiste Chantraine.
\newblock Lagrangian concordance of {L}egendrian knots.
\newblock {\em Algebr. Geom. Topol.}, 10(1):63--85, 2010.

\bibitem[Cha15]{Cha15}
Baptiste Chantraine.
\newblock A note on exact {L}agrangian cobordisms with disconnected
  {L}egendrian ends.
\newblock {\em Proc. Amer. Math. Soc.}, 143(3):1325--1331, 2015.

\bibitem[Che02]{Che}
Yuri Chekanov.
\newblock Differential algebra of {L}egendrian links.
\newblock {\em Invent. Math.}, 150(3):441--483, 2002.

\bibitem[CMP19]{CMP}
Roger Casals, Emmy Murphy, and Francisco Presas.
\newblock Geometric criteria for overtwistedness.
\newblock {\em J. Amer. Math. Soc.}, 32(2):563--604, 2019.

\bibitem[CN13]{ChNg}
Wutichai Chongchitmate and Lenhard Ng.
\newblock An atlas of {L}egendrian knots.
\newblock {\em Exp. Math.}, 22(1):26--37, 2013.

\bibitem[CN21]{CN}
Roger Casals and Lenhard Ng.
\newblock Braid loops with infinite monodromy on the {L}egendrian contact dga.
\newblock {\em arXiv preprint arXiv:2101.02318}, 2021.

\bibitem[CS21]{orsola_thesis}
Orsola Capovilla-Searle.
\newblock Infinitely many planar fillings and symplectic {M}ilnor fibers.
\newblock {\em arXiv preprint arXiv:2201.03081}, 2021.

\bibitem[CZ22]{CZ}
Roger Casals and Eric Zaslow.
\newblock Legendrian weaves: N-graph calculus, flag moduli and applications.
\newblock {\em Geometry \& Topology, to appear}, 2022.

\bibitem[DR11]{DR:knotted}
Georgios Dimitroglou~Rizell.
\newblock Knotted {L}egendrian surfaces with few {R}eeb chords.
\newblock {\em Algebr. Geom. Topol.}, 11(5):2903--2936, 2011.

\bibitem[DR15]{dimitroglou_rizell_2015}
Georgios Dimitroglou~Rizell.
\newblock Exact {L}agrangian caps and non-uniruled {L}agrangian submanifolds.
\newblock {\em Arkiv för Matematik}, 53(1):37–64, 2015.

\bibitem[DR16]{DR}
Georgios Dimitroglou~Rizell.
\newblock Lifting pseudo-holomorphic polygons to the symplectisation of
  {$P\times\Bbb{R}$} and applications.
\newblock {\em Quantum Topol.}, 7(1):29--105, 2016.

\bibitem[EES05a]{EESu}
Tobias Ekholm, John Etnyre, and Michael Sullivan.
\newblock Non-isotopic {L}egendrian submanifolds in {$\Bbb R^{2n+1}$}.
\newblock {\em J. Differential Geom.}, 71(1):85--128, 2005.

\bibitem[EES05b]{EESorientation}
Tobias Ekholm, John Etnyre, and Michael Sullivan.
\newblock Orientations in {L}egendrian contact homology and exact {L}agrangian
  immersions.
\newblock {\em Internat. J. Math.}, 16(5):453--532, 2005.

\bibitem[EES09]{EESduality}
Tobias Ekholm, John~B. Etnyre, and Joshua~M. Sabloff.
\newblock A duality exact sequence for {L}egendrian contact homology.
\newblock {\em Duke Math. J.}, 150(1):1--75, 2009.

\bibitem[EGH00]{EGH}
Yakov Eliashberg, Alexander Givental, and Helmut Hofer.
\newblock {\em Introduction to symplectic field theory}.
\newblock Number Special Volume, Part II. 2000.
\newblock GAFA 2000 (Tel Aviv, 1999).

\bibitem[EHK16]{EHK}
Tobias Ekholm, Ko~Honda, and Tam{\'a}s K{\'a}lm{\'a}n.
\newblock Legendrian knots and exact {L}agrangian cobordisms.
\newblock {\em J. Eur. Math. Soc. (JEMS)}, 18(11):2627--2689, 2016.

\bibitem[Ekh12]{EkhSFT}
Tobias Ekholm.
\newblock Rational {SFT}, linearized {L}egendrian contact homology, and
  {L}agrangian {F}loer cohomology.
\newblock In {\em Perspectives in analysis, geometry, and topology}, volume 296
  of {\em Progr. Math.}, pages 109--145. Birkh\"auser/Springer, New York, 2012.

\bibitem[Eli98]{Eli}
Yakov Eliashberg.
\newblock Invariants in contact topology.
\newblock In {\em Proceedings of the {I}nternational {C}ongress of
  {M}athematicians, {V}ol. {II} ({B}erlin, 1998)}, number Extra Vol. II, pages
  327--338, 1998.

\bibitem[EN19]{NE}
John~B. Etnyre and Lenhard Ng.
\newblock Legendrian contact homology in $\mathbb{R}^3$.
\newblock {\em arXiv preprint arXiv:1811.10966}, 2019.

\bibitem[FI04]{FI}
Dmitry Fuchs and Tigran Ishkhanov.
\newblock Invariants of {L}egendrian knots and decompositions of front
  diagrams.
\newblock {\em Mosc. Math. J.}, 4(3):707--717, 783, 2004.

\bibitem[Fuc03]{Fuchs}
Dmitry Fuchs.
\newblock Chekanov-{E}liashberg invariant of {L}egendrian knots: existence of
  augmentations.
\newblock {\em J. Geom. Phys.}, 47(1):43--65, 2003.

\bibitem[Gei08]{Geiges}
Hansj\"{o}rg Geiges.
\newblock {\em An introduction to contact topology}, volume 109 of {\em
  Cambridge Studies in Advanced Mathematics}.
\newblock Cambridge University Press, Cambridge, 2008.

\bibitem[GSW20a]{Gao2}
Honghao Gao, Linhui Shen, and Daping Weng.
\newblock Augmentations, fillings, and clusters.
\newblock {\em arXiv preprint arXiv:2008.10793}, 2020.

\bibitem[GSW20b]{Gao1}
Honghao Gao, Linhui Shen, and Daping Weng.
\newblock Positive braid links with infinitely many fillings.
\newblock {\em arXiv preprint arXiv:2009.00499}, 2020.

\bibitem[HR15]{HR}
Michael~B. Henry and Dan Rutherford.
\newblock Ruling polynomials and augmentations over finite fields.
\newblock {\em J. Topol.}, 8(1):1--37, 2015.

\bibitem[K{\'a}l05]{kalman:one-parameter}
Tam\'as K{\'a}lm{\'a}n.
\newblock Contact homology and one parameter families of {L}egendrian knots.
\newblock {\em Geom. Topol.}, 9:2013--2078 (electronic), 2005.

\bibitem[Kar20]{karl}
Cecilia Karlsson.
\newblock A note on coherent orientations for exact {L}agrangian cobordisms.
\newblock {\em Quantum Topol.}, 11(1):1--54, 2020.

\bibitem[KM21]{KroMro19}
Peter~B. Kronheimer and Thomasz~S. Mrowka.
\newblock Instantons and some concordance invariants of knots.
\newblock {\em Journal of the {L}ondon {M}athematical {S}ociety},
  104(2):541--571, 2021.

\bibitem[Leg20]{Noe}
No{\'e}mie Legout.
\newblock Product structures in {F}loer theory for {L}agrangian cobordisms.
\newblock {\em J. Symplectic Geom.}, 18(6):1647--1750, 2020.

\bibitem[LS91]{LS}
Fran\c{c}ois Lalonde and Jean-Claude Sikorav.
\newblock Sous-vari\'et\'es lagrangiennes et lagrangiennes exactes des fibr\'es
  cotangents.
\newblock {\em Comment. Math. Helv.}, 66(1):18--33, 1991.

\bibitem[LVC18]{LivVanC}
Charles Livingston and Cornelia Van~Cott.
\newblock The four-genus of connect sums of torus knots.
\newblock {\em Mathematical Proceedings of the Cambridge Philosophical
  Society}, 164(3):531--550, 2018.

\bibitem[Mis03]{Mishachev03}
K.~Mishachev.
\newblock The {$N$}-copy of a topologically trivial {L}egendrian knot.
\newblock {\em J. Symplectic Geom.}, 1(4):659--682, 2003.

\bibitem[MS95]{MS95}
Dusa McDuff and Dietmar Salamon.
\newblock {\em Introduction to symplectic topology}.
\newblock Oxford Mathematical Monographs. The Clarendon Press Oxford University
  Press, New York, 1995.
\newblock Oxford Science Publications.

\bibitem[Ng03]{Ng:computable}
Lenhard Ng.
\newblock Computable {L}egendrian invariants.
\newblock {\em Topology}, 42(1):55--82, 2003.

\bibitem[Ng05]{Ng}
Lenhard Ng.
\newblock A {L}egendrian {T}hurston-{B}ennequin bound from {K}hovanov homology.
\newblock {\em Algebr. Geom. Topol.}, 5:1637--1653, 2005.

\bibitem[NRS{\etalchar{+}}20]{NRSSZ}
Lenhard Ng, Dan Rutherford, Vivek Shende, Steven Sivek, and Eric Zaslow.
\newblock Augmentations are sheaves.
\newblock {\em Geom. Topol.}, 24(5):2149--2286, 2020.

\bibitem[NRSS17]{NRSS}
Lenhard Ng, Dan Rutherford, Vivek Shende, and Steven Sivek.
\newblock The cardinality of the augmentation category of a {L}egendrian link.
\newblock {\em Math. Res. Lett.}, 24(6):1845--1874, 2017.

\bibitem[NT04]{NgTraynor04}
Lenhard Ng and Lisa Traynor.
\newblock Legendrian solid-torus links.
\newblock {\em J. Symplectic Geom.}, 2(3):411--443, 2004.

\bibitem[OS16]{OS}
Brendan Owens and Sa\v{s}o Strle.
\newblock Immersed disks, slicing numbers and concordance unknotting numbers.
\newblock {\em Comm. Anal. Geom.}, 24(5):1107--1138, 2016.

\bibitem[Pan17]{Pan1}
Yu~Pan.
\newblock The augmentation category map induced by exact {L}agrangian
  cobordisms.
\newblock {\em Algebr. Geom. Topol.}, 17(3):1813--1870, 2017.

\bibitem[PC05]{CP}
Petr~E. Pushkar and Yuri~V. Chekanov.
\newblock Combinatorics of fronts of {L}egendrian links, and {A}rnold's
  4-conjectures.
\newblock {\em Uspekhi Mat. Nauk}, 60(1(361)):99--154, 2005.

\bibitem[Pez18]{Samthesis}
Samantha Pezzimenti.
\newblock Immersed {L}agrangian fillings of {L}egendrian submanifolds via
  {G}enerating {F}amilies.
\newblock {\em {P}h{D} {D}iss., {B}ryn {M}awr {C}ollege}, 2018.

\bibitem[Pol91]{Pol91}
Leonid Polterovich.
\newblock The surgery of {L}agrange submanifolds.
\newblock {\em Geom. Funct. Anal.}, 1(2):198--210, 1991.

\bibitem[PR22]{PR}
Yu~Pan and Dan Rutherford.
\newblock Wrapped floer theory for immersed surfaces.
\newblock {\em in preparation}, 2022.

\bibitem[PT22]{sam_lisa}
Samantha Pezzimenti and Lisa Traynor.
\newblock The geography of immersed {L}agrangian fillings of {L}egendrian
  submanifolds via generating families.
\newblock {\em in preparation}, 2022.

\bibitem[Sab05]{Sab}
Joshua~M. Sabloff.
\newblock Augmentations and rulings of {L}egendrian knots.
\newblock {\em Int. Math. Res. Not.}, 2005(19):1157--1180, 2005.

\bibitem[Sab20]{Sab20}
J.M. Sabloff.
\newblock {\em Ruling and augmentation invariants of {L}egendrian knots {\rm in
  Encyclopedia of Knot Theory}}.
\newblock Routledge Handbooks Online. CRC Press, 2020.

\bibitem[Sei08]{Seidel}
Paul Seidel.
\newblock {\em Fukaya categories and {P}icard-{L}efschetz theory}.
\newblock Zurich Lectures in Advanced Mathematics. European Mathematical
  Society (EMS), Z\"urich, 2008.

\bibitem[ST13]{SabTra}
Joshua~M. Sabloff and Lisa Traynor.
\newblock Obstructions to {L}agrangian cobordisms between {L}egendrian
  submanifolds.
\newblock {\em Algebr. Geom. Topol.}, 13:2733--2797, 2013.

\bibitem[Yau13]{Yau}
Mei-Lin Yau.
\newblock Surgery and invariants of {L}agrangian surfaces.
\newblock {\em arXiv preprint arXiv:1306.5304}, 2013.

\end{thebibliography}

\end{document}